\documentclass{article}
\usepackage[utf8]{inputenc}
\usepackage{float}
\usepackage[english]{babel}
\usepackage{subcaption}
\usepackage{hyperref}
\usepackage{graphicx}
\usepackage[T1]{fontenc}
\usepackage{amsthm,amsmath,amssymb}
\usepackage{bbold}
\usepackage{pst-all}
\usepackage[utf8]{inputenc}
\usepackage{import}
\usepackage{float}
\usepackage{stmaryrd}
\usepackage{bbm}
\usepackage[left=3cm,right=3cm,top=3cm,bottom=3.5cm]{geometry}
\newtheorem{lemma}{Lemma}[section]
\newtheorem{proposition}{Proposition}[section]

\newtheorem{ex}{Example}
\newtheorem{rem}{Remark}
\newtheorem{cor}{Corollary}[section]
\newtheorem{theorem}{Theorem}
\newtheorem{definition}{Definition}
\newtheorem{hyp}{Assumption}
\newcommand\nr{\bar{r}}
\newcommand\nc{\bar{c}}
\newcommand\bc{\mathbbm{c}}
\newcommand\rinf{\underline{r}}

\newcommand\pr{\mathbf{X}}

\newcommand\br{\mathbbm{r}} 
\newcommand\ir{\underline{r}}
\newcommand\vep{\varepsilon}
\newcommand\eps{\varepsilon}
\newcommand\iep{\lfloor\vep^{-1}\rfloor}
\newcommand{\Var}{\mathrm{Var}}
\newcommand{\vpk}{{\varphi(k)}}
\usepackage[authormarkup=none,final]{changes}
\definechangesauthor[name=julie, color=red]{j}
\numberwithin{equation}{section}

\definecolor{ppink}{cmyk}{99, 0, 35, 0}
\newcommand{\jj}[1]{#1}
\newcommand{\gr}[1]{#1}

\usepackage{mathtools}
\title{Spreading speed of locally regulated population models in  macroscopically heterogeneous environments}
\author{Pascal Maillard\thanks{Institut de Mathématiques de Toulouse (CNRS UMR 5219) and Université de Toulouse and Institut Universitaire de France (IUF). Partially supported by grant ANR-20-CE92-0010-01. \emph{e-mail:} pascal DOT maillard AT math.univ-toulouse.fr.}, Gaël Raoul\thanks{CMAP, CNRS, \'Ecole polytechnique, I.P. Paris, 91128 Palaiseau, France.}, Julie Tourniaire\thanks{Laboratoire de mathématiques de Besançon, UMR 6623, Université de
		Franche-Comté, CNRS, F-25000 Besançon, France. \emph{e-mail:} julie.tourniaire AT univ-fcomte.fr}}
\date{December 21, 2024}
\begin{document}
\maketitle

\begin{abstract}
We consider a certain lattice branching random walk with on-site competition and in an environment which is heterogeneous at a macroscopic scale $1/\eps$ in space and time. This can be seen as a model for the spatial dynamics of a biological population in a habitat which is heterogeneous at a large scale (mountains, temperature or precipitation gradient\ldots). The model incorporates another parameter, $K$, which is a measure of the local population density. We study the model in the limit when first $\eps\to 0$ and then $K\to\infty$. In this asymptotic regime, we show that the rescaled position of the front as a function of time converges to the solution of an explicit ODE. We further discuss the relation with another popular model of population dynamics, the Fisher-KPP equation, which arises in the limit $K\to\infty$. Combined with known results on the Fisher-KPP equation, our results show in particular that the limits $\eps\to0$ and $K\to\infty$ do not commute in general. We conjecture that an interpolating regime appears when $\log K$ and $1/\eps$ are of the same order.
%Such models have seen considerable interest in the last decades in mathematics, physics and biology. The prototypical model of front propagation is the Fisher-KPP equation, a semi-linear parabolic partial differential equation. 
\end{abstract}

\tableofcontents

\pagebreak 

\section{Introduction}
\sectionmark{Introduction}
In this article, we are interested in the spatial propagation of a biological population in a heterogeneous environment, where the population lives on discrete sites or \emph{demes}. Formally, the population is a system of interacting particles on the integers $\mathbb{Z}$ evolving in discrete time. At each generation, particles duplicate at a certain space- and time-depending \jj{probability}, undergo a regulation step where $K$ particles at most survive at each site and jump (\emph{migrate}) according to a discretized Gaussian distribution.

\jj{In this first section,
	we introduce our model, state the main result of this article
	and compare it with previous results from the PDE literature.}

\subsection{Model and main result} \label{sec:simp:model}

We consider a particle system evolving on the rescaled lattice $\Delta x \cdot \mathbb Z$ at discrete time steps $0,\Delta t, 2\Delta t,\ldots$, where $\Delta t,\,\Delta x>0$ are small parameters. The system depends on the following parameters:
\begin{itemize}
	\item $\eps>0$ a small constant (with $1/\eps$ being the space- and time-scale of interest)
	\item $K>0$ a large constant  (the \emph{carrying capacity})%(the \emph{local population density})
	\item  $r:[0,\infty)\times \mathbb{R}\to (0,\infty) $ \jj{a function (the growth rate). }
\end{itemize}

\jj{The function $r$ will be assumed to satisfy certain regularity
	conditions, which will be described later in this section. A function satisfying
	these assumptions will be referred to as a \textit{good} growth rate function.}

\jj{Formally, our model is a Markov chain $(n_k)_{k\in\mathbb{N}}$ taking values
	in $\mathbb{N}_0^{\mathbb{Z}}$, where $n_k(i)$
	is interpreted as the number of particles on the
	site $i\Delta x$ at time $k\Delta t$. At each time step, the configuration
	$n_{k+1}$ is derived from $n_k$ through three consecutive steps:
	reproduction,  competition and  migration.
	These steps are defined in such a way that the process $n_k$ satisfies a
	\textit{monotonicity} condition: for two copies $n^1$ and $n^2$ of the system,
	starting from two intial conditions such that $n_0^1(i)\geq n_0^2(i)$ for all
	$i\in\mathbb{Z}$, the processes  $n^1$ and $n^2$ can be coupled in such a way that
	$n_k^1(i)\geq n_k^2(i)$ for all $k\in\mathbb{N}$ and $i\in\mathbb{Z}$.
	In this setting, we will need the following two definitions.}

\begin{definition}
	Let $\mu$ and $\nu$ be two probability distributions on $\mathbb R$. We say that $\nu$ stochastically dominates $\mu$ if $\nu([x,\infty))\geq \mu([x,\infty))$ for all $x\in \mathbb{R}$.
\end{definition}
\begin{definition}
	A family of probability distributions $(\nu_r)$ on $\mathbb R$ is
	increasing with respect to $r$ if for all $r_{1}<r_{2}$,
	$\nu_{r_{2}}$ stochastically dominates $\nu_{r_{1}}$.
\end{definition}

\jj{We are now ready to define our interacting particle system.
	Suppose we are given a good growth rate function $r$ and
	an increasing sequence of reproduction laws
	(i.e.~probability distributions on $\mathbb{N}$) denoted by
	$(\nu_r)$. Assume that  $\nu_r(0)=0$ for all $r>0$.
	Additionally, we assume that the initial condition
	satisfies $n_0(i) = 0$ for $i > 0$ and $n_0(0) \ge 1$, i.e.~the right-most
	particle is at the origin.  The configuration $(n_{k+1}(i))_{i\in\mathbb{Z}}$
	is obtained from $(n_{k}(i))_{i\in\mathbb{Z}}$ as follows:}

\begin{enumerate}
	\item \emph{Reproduction step.}
	Each particle living on the $i$-th site at generation $k$
	independently gives birth to a random number of children distributed according to
	$\nu_{r(\eps k\Delta t,\eps i\Delta x)}$.
	\item \emph{Competition step.}
	Only K particles per site survive to the next generation.
	In other words, the number of particles on site $i$ after the competition step is given by the truncated sum
	\begin{equation}\left(\sum_{m=1}^{n_k(i)}Y_m\right)\wedge K,\label{reprtronc}\end{equation}
	where $(Y_m)$ is a sequence of i.i.d. random variables of law $ \nu_{r(\eps k\Delta t,\eps i\Delta x)}$.
	%This step models the effect of food shortages when the population gets too large on one site. In this work, we will only consider local competition.
	
	\item \emph{Migration step.} A particle on the $i$-th site jumps to the site $i+j$ with probability $\mu(j)$, where the \emph{migration law} $\mu$ is a discretized normal distribution:
	\begin{equation}
		\mu(j)=\int_{(j-\frac 12)\Delta x}^{(j+\frac 12)\Delta x}\frac 1{\sqrt{2\pi\Delta t}}e^{-\frac{x^2}{2\Delta t}}\,dx.
		\label{displaw}
	\end{equation}
	All particles jump independently
	and \jj{simultaneously}.
\end{enumerate}
The resulting configuration is $n_{k+1}$.

We denote by $$X^*_k=\max\{i\in\mathbb{Z}:n_k(i)>0\}\cdot \Delta x$$ the position of the rightmost particle at generation $k$ in this system. Note that $X_0^* = 0$ by assumption.
The main goal of this article is to investigate the long-time behaviour of
the process $\left(\eps X_{k}^*\right)_{k\in\mathbb{Z}}$. More precisely, we
compare $\left(\eps X_{k}^*\right)_{k\ge0}$ with $(x(\eps k\Delta t))_{k\ge0}$,
where $x$ is the  solution of the Cauchy problem
\begin{equation}
	\begin{cases}
		x'(t) & =\sqrt{2r(t,x(t))} \\
		x(0)  & = 0.
	\end{cases}
	\tag{$\mathcal{C}$}
	\label{cauchypb}
\end{equation} The result is the following.
\begin{theorem}
	\label{main_theorem}
	\jj{Assume that $r$ is a good growth rate function}.
	Let $T>0$ and $\delta >0$. There exists $\Delta {t_\delta}>0$ and $C_\delta>0$ such that, if
	$
	\Delta t <\Delta {t_\delta}$ and $\Delta x< C_\delta \Delta t,
	$
	there exists $K_0$ such that, for all $K>K_0$
	\begin{equation*}
		\lim_{\eps\to 0}\mathbb{P}\left(\max_{k\in\left\llbracket0, \left\lfloor \frac{T}{\eps\Delta t}\right\rfloor \right\rrbracket}\left| \eps X^*_{k}-x(\eps k\Delta t)\right|\leq \delta\right)=1.%\geq 1-\delta .
	\end{equation*}
	\label{Thintro}\end{theorem}

\jj{We conclude this section by introducing our definition of a good
	growth rate function.
	This definition will ensure the existence and uniqueness of a solution
	to the Cauchy problem \eqref{cauchypb} defined on $[0,\infty)$.}

As we shall see, it will be sufficient to prove our intermediate results
for growth rate functions $r$ satisfying strong regularity conditions. In this case,
we will rely on the following assumptions.
\begin{definition}\jj{We say that $r$ is a \textit{smooth growth rate function}}
	if
	\begin{itemize}
		\item[(i)] there exists $0<\ir<\nr$ such that
		\begin{equation*}
			%\label{eq:rbounds}
			\forall (t,x)\in [0,\infty)\times \mathbb{R},
			\quad \ir \leq r(t,x) \leq \nr,
		\end{equation*}
		\item[(ii)] we have
		\begin{equation}
			\label{eq:ub_grad_r}
			r\in \mathcal{C}^1((0,\infty)\times \mathbb{R}), \quad
			\text{and} \quad
			\exists M>0: \quad \|\partial r/\partial t\|_\infty
			+\|\partial r/\partial x\|_\infty <M.
		\end{equation}
	\end{itemize}
	\jj{These two conditions imply that,
		for all $(t_1,x_1),(t_2,x_2)\in [0,\infty]\times \mathbb{R}$,
		\begin{equation}
			\label{eq:approx_r}
			|\sqrt{2r(t_1,x_1)}-\sqrt{2r(t_2,x_2)}|\leq L (|t_1-t_2|+|x_1-x_2|),
			\quad \text{with} \quad L:=\frac{M}{\sqrt{2\ir}},
		\end{equation}
		and that \eqref{cauchypb} has a unique global solution.}
	
\end{definition}
\begin{definition}
	\label{def:good_rate}
	\jj{We say that $r$ is a \textit{good growth rate function} if there exists $0<\ir<\nr$ and two sequences of smooth growth rate functions
		$(\theta_n)$ and $(\Theta_n)$ defined on $[0,\infty)\times \mathbb{R}$ such that the following holds.}
	\begin{itemize}
		\item[(i)] For all $n\in\mathbb{N}$ and for all $(t,x)\in [0,\infty)\times \mathbb{R},$
		\begin{equation*}
			\ir \leq \theta_n(t,x) \leq r(t,x) \leq \Theta_n(t,x) \leq \nr.
		\end{equation*}
		\item[(ii)] 
		%   Furthermore, we assume that the functions $(\theta_n)$ and $(\Theta_n)$ are
		%   smooth in the sense that there are in $\mathcal{C}^1((0,\infty)\times \mathbb{R})$
		%   and 
		%   \begin{equation}
			%       \label{eq:bound_grad_theta}
			%       \forall n\in \mathbb{N}, \; \exists \jj{M_n}>0:
			%       \|\partial \theta_n/\partial t\|_\infty
			%       +\|\partial \theta_n/\partial x\|_\infty
			%       + \|\partial \Theta_n/\partial t\|_\infty
			%       +\|\partial \Theta_n/\partial x\|_\infty
			%       <M_n.
			%   \end{equation}
		
		%This assumption implies that for any initial condition $x_0$,
		%and for all $n\in\mathbb{N}$, there exists a unique couple
		Let
		$(\underline{x}_n,\bar{x}_n)$ denote the unique couple 
		of global solutions to the Cauchy problems
		\begin{equation}
			\label{eq:cp_approx}
			\dot{\underline{x}}_n=\sqrt{2\theta_n(t,\underline{x}_n(t))}, \quad
			\dot{\bar{x}}_n=\sqrt{2\Theta_n(t,\bar{x}_n(t))},
			\quad \underline{x}_n(0)=\bar{x}_n(0)=x_0.
		\end{equation}
		We assume that for all $T>0$,
		\begin{equation*}
			\sup_{t\in[0,T]}|\underline{x}_n(t)-\bar{x}_n(t)|\xrightarrow{n\to\infty}0.
		\end{equation*}
	\end{itemize}
	
\end{definition}

\begin{ex}
	\jj{Smooth growth rate functions are good
		growth rate functions. Functions  that only depends on one coordinate,
		and that are piecewise $\mathcal{C}^1$ in this coordinate are good growth
		rate functions. This example will be further developed in the next section.}
	\gr{ The function $r(t,x)=1+1_{x\geq (\sqrt 2+2)t/2}$ is not a good growth rate function: for $x_0=0$, $\underline{x}_n(t)=\sqrt 2 t$, while $\bar x_n(t)=2t$, for any sequence of approximations $\theta_n$, $\Theta_n$. 
		If $\phi:\mathbb R\mapsto \mathbb R_+$ denotes a smooth function supported on $[-(1-\sqrt 2/2),1-\sqrt 2/2]$ and of integral $1$, $\tilde r(t,x):=\int_{\mathbb R} \frac{\phi((x-y)/t)}{t}r(t,y)\,dy$  is another example of a growth rate function that is not good (we also have $\underline{x}_n(t)=\sqrt 2 t$ and $\bar x_n(t)=2t$ for $x_0=0$), even though the only discontinuity point of $\tilde r$ is $(0,0)$.}
\end{ex}

\subsection{Discussion and comparison with deterministic models}\label{sec:det}

In Theorem~\ref{main_theorem}, we consider $K$ large but fixed and let $\eps\to 0$.
A more classical large population asymptotics of \emph{individual-based models}
such as this one consists in letting first $K\to \infty$, then $\eps\to 0$.
\jj{If we divide the population size by $K$ and let $K\to\infty$}  (and let $\Delta x,\Delta t$ tend to 0 in a controlled way), we obtain a (deterministic) PDE  (\cite{fournier2004microscopic}, \cite{Champagnat2007}). This PDE is a reaction-diffusion equation of Fisher-KPP type governing the density of individuals $u^\vep(t,x)$:
\begin{equation}
	\label{eq:FKPP}
	u^\eps_t(t,x) = \frac{1}{2} u^\eps_{xx}(t,x) + r(\eps t,\eps x)f(u^\eps(t,x)),
\end{equation}
with $f(u) = u\boldsymbol 1_{u\le 1}.$ The limiting behavior of $u^\eps$ as $\eps\to0$ has been widely investigated in the PDE literature \cite{barles1989wavefront,Evans:1989to,Berestycki:2015aa} but also with probabilistic arguments \cite{Freidlin1986}. Introducing the change of variables $(t,x)\mapsto (t/\eps,x/\eps)$ and the WKB ansatz
\begin{align}
	\label{eq:scaling}
	u^\eps(t,x) = e^{-v^\eps(t/\eps,x/\eps)/\eps},
\end{align}
and assuming that, in a certain sense, $u^\eps(t,x) \to \boldsymbol 1_{x=0}$ as $t\to 0$, it can be shown that the function $v^\eps$ converges, when $\eps\to0$, to the viscosity solution $v$ of the following Hamilton-Jacobi equation (or, more precisely, a variational inequality) \cite{Evans:1989to}:
\begin{align}
	\label{eq:HJ}
	\begin{cases}
		\min\left(v_t(t,x) + \frac{1}{2}(v_x(t,x))^2 + r(t,x),v(t,x)\right) = 0
		\\
		v(t,x) \to \infty\boldsymbol 1_{x\ne 0},\,t\to 0.
	\end{cases}
\end{align}
As a consequence, $u^\eps(t/\eps,x/\eps)$ converges to $1$ (resp. $0$) uniformly on compact subsets of int($I$) (resp. $I^c$), where
\[
I = \{(t,x)\in[0,\infty)\times \mathbb R:v(t,x)=0\}.
\]
In particular, if $x^\eps(t)$ denotes the position of the front (for example, $x^\eps(t) = \sup\{x:u^\eps(t,x)\ge 1/2\}$), then, for fixed $t\ge0$,
\[
x^\eps(t/\eps) \to x^{HJ}(t): = \sup\{x:v(t,x)=0\},\quad \text{as $\eps\to 0$.}
\]
This approach has been extensively employed so far to deal with different types of heterogeneous environments: periodic  \cite{shigesada1997biological,Shigesada:1986aa,Xin:1991aa}, random  \cite{Berestycki:2015aa,Nadin:2016aa}, but does not provide an explicit propagation speed, except in very specific situations \cite{Hamel:2011ue}.
\jj{However, it is known that the following bound holds.}
\begin{lemma}
	\label{lem:comp_HJ_ODE}
	\jj{ Let $x(\cdot)$ denote the solution of the Cauchy problem \eqref{cauchypb}. We always have
		\[
		x^{HJ}(t) \ge x(t)\qquad\text{ for all $t\ge0$.}
		\]}
\end{lemma}
\begin{proof}
	To see this, recall the variational representation of the function $v$ \cite{Evans:1989to}:
	\begin{align}
		v(t,x) = \sup_{\tau} \inf_{z} \left\{\int_0^{t\wedge \tau(z)} \frac{z'(s)^2}{2} - r(t-s,z(s))\,ds\mid z(0)=x,\,z(t)=0\right\}.
		\label{eq:variational_representation}
	\end{align}
	Here, the infimum is over all $z\in H^1_{\text{loc}}([0,\infty);\mathbb R)$ and the supremum is over all \emph{stopping times}\footnote{In fact, general theory of variational inequalities (see e.g. \cite[p.6]{Bensoussan1982}) implies that for given $(t,x)$, the optimal stopping time in \eqref{eq:variational_representation} is given by
		\(
		\tau_{t,x}(z) = \inf\{s\in[0,t]: v(t-s,z(s)) = 0\},
		\)
		but we don't make use of this fact.} $\tau$, i.e.~maps $\tau:H^1_{\text{loc}}([0,\infty);\mathbb R)\to[0,\infty)$ satisfying for all $z,\tilde z$ and all $s\ge0$:
	\[
	\text{if $z\equiv \tilde z$ on $[0,s]$ and $\tau(z)\le s$, then $\tau(\tilde z) = \tau(z)$.}
	\]
	In order to show that $x^{HJ}(t) \ge x(t),$ it suffices to show that $v(t,x(t)) = 0$ for all $t\ge0$. Fix $t\ge0$. Define $z(s) = x(t-s)$ for $s\in[0,t]$. Then $z(0) = x(t)$, $z(t) = 0$ and for all $s\in[0,t]$, $(z'(s))^2/2 = r(t-s,z(s))$. Hence, for every stopping time $\tau$, the integral in \eqref{eq:variational_representation} equals 0. This shows that $v(t,x(t)) \le 0$ and thus $v(t,x(t)) = 0$ by non-negativity of $v$.
\end{proof}
It is easy to construct examples where $x^{HJ}(t) > x(t)$ for some or all $t>0$. This is for example the case when $r(t,x) = r_0(x)$ for some (strictly) increasing function $r_0$. Indeed, in this case, it is easy to construct an affine function $z$ such that $z(0) > x(t)$ and $z(t) = 0$ and such that the integrand in \eqref{eq:variational_representation} is negative for all $s\in(0,t)$, whence $v(t,z(0)) = 0$ and $x^{HJ}(t) \ge z(0) > x(t)$. It is even possible to construct an example in which $x^{HJ}$ has jumps: if we consider a function $r$ such that, for some $h>0$, $r(x)=c_1>0$ for $x<h$ and $r(x)=c_2>2c_1$ if $x\geq h>0$, and an initial condition $\mathbb{1}_{(-\infty,0]}$, we observe a jump in the wavefront at time $T_0:=\frac{h}{c_2}\sqrt{2(c_2-c_1)}<\frac{h}{\sqrt{2c_1}}$ (see Example 3 in \cite{Freidlin1985}). On the other hand, when $r_0$ is non-increasing, then $x^{HJ}(t) = x(t)$ for all $t\ge0$, see \cite{Evans:1989to,Freidlin1985} for a detailed discussion and other sufficient conditions such that $x^{HJ}(t) = x(t)$ for all $t\ge0$. In this case, one says that the \emph{Huygens principle} is verified, in that the propagation of the front is described by a velocity field, see Freidlin \cite{Freidlin1986} for a discussion of this principle and its relation with the Hamilton-Jacobi limit, that he relates to geometric optics.

%
%Freidlin \cite{Freidlin1986} also questioned the link 
%between the displacement of the invasion front $x^{HJ}$ and the Huygens principle. Indeed, Equation (\ref{eq:FKPP}) often models the propagation of a disturbance in physics and the Huygens principle proposes an explanation of the wavefront propagation of light. According to this principle, it can described by a certain velocity fiels $v(x,\pm 1)$ and the question asked by Freidlin is the following: is it always possible to find such a velocity field that describes the propagation of the wavefront $x^{HJ}$? Actually, this representation fails in many cases, some of them being pointed out in \cite{Freidlin1986}. As for Theorem \ref{main_theorem}, note that the Huygens principle is satisfied and that the velocity field is explicit,  given by (\ref{cauchypb}).

It has been observed previously that the viscosity solution method may be unsatisfactory, from a biological standpoint,  in some situations \cite{Hamel:2010wf,jabin2012small}. This has been dubbed the ``tail problem'' \cite{jabin2012small}: artifacts may be generated in the deterministic model by the \emph{infinite speed of propagation} \cite{Hamel:2010wf} of the solutions of (\ref{eq:FKPP}), where meaningless, exponentially small ``populations'' are sent to favourable regions by diffusion before the invasion front $x(t)$, accelerating the speed of propagation and possibly causing  jumps in the position of the invasion front.
%This phenomenon points out a crucial difference between the reaction-diffusion model (\ref{eq:FKPP}) and the stochastic one. 
Some adjustments were suggested to ``cut the tails'' in the deterministic model. For instance, one can add a square root term with a survival threshold parameter in the F-KPP equation \cite{jabin2012small,mirrahimi2012singular}. Another correction suggested in \cite{Hamel:2010wf} consists in adding a strong Allee effect in Equation (\ref{eq:FKPP}). Namely, they set the growth rate $f$ to be negative at low densities, leading to a bistable reaction-diffusion equation. For such equations, the Huygens principle is verified, as shown by Freidlin \cite{Freidlin1986}.

In this article, we propose to come back to the microscopic, or individual-based population model and study it under a double limit, where we let first the space-time scale $1/\eps$, then the carrying capacity $K$ go to infinity.

The discrete nature of our model has the effect of a ``cutoff'' which prevents the solution from being exponentially small in $1/\eps$. In terms of the function $v$, which arises in the limit after a hyperbolic scaling, the cutoff prevents the function $v$ from taking finite positive values and thus formally ``pushes it up to $\infty$'' whenever it is (strictly) positive. The main conceptual advantage of this approach compared to the PDE approach is that our model naturally satisfies the Huygens principle, without the need of ad-hoc modifications.

In order to determine which of the two models, with or without cutoff, is a better model for a given biological population, one might consider our microscopic model in the limit when $K$ and $1/\eps$ go to infinity together. Indeed, we conjecture that it is possible to interpolate between the two double limits in $K$ and $\eps$, when $K$ and $1/\eps$ go to infinity in such a way that $\log K$ is of the same order as $1/\eps$. This relation is indeed suggested by the hyperbolic scaling \eqref{eq:scaling}. It also appears in the proof of Theorem~\ref{main_theorem}, see for example Proposition~\ref{thupperbound}. Precisely, setting $\log K = \kappa/\eps$ for a constant $\kappa > 0$, we believe that the renormalized log-density is described in the limit by a solution $v^\kappa$ to a variational inequality similar to \eqref{eq:HJ}, but with an additional constraint imposing that $v^\kappa(t,x)\ge \kappa$ implies $v^\kappa(t,x) = +\infty$, similarly to \cite{mirrahimi2012singular}. The variational representation of the function $v^\kappa$ would then be
\begin{align*}
	v^\kappa(t,x) = \sup_\tau \inf_z\Big\{\int_0^{t\wedge \tau(z)} & \frac{z'(s)^2}{2} - r(t-s,z(s))\,ds\mid z(0)=x,\,z(t)=0,                                          \\
	& \forall u\le\tau(z): \int_u^{t\wedge \tau(z)} \frac{z'(s)^2}{2} - r(t-s,z(s))\,ds < \kappa\Big\}.
\end{align*}
We leave the details for future work.

\subsection{Explicit example and simulations} \label{sec:comparison:HJ}

In this section, we illustrate our results on an explicit example.
% for which the asymptotic speed of propagation of the Fisher-KPP equation \eqref{eq:FKPP} is strictly larger than the speed of the ODE \eqref{cauchypb}. 
Let $r(t,x) = r(x)$, where $r$ is the $1$-periodic function such that
\begin{equation}
	\label{hyp:rperiodic}
	\begin{cases}
		r(x)=\mu^+ & \forall x\in[0,\frac{1}{2}) \\
		r(x)=\mu^- & \forall x\in[\frac{1}{2},1)
	\end{cases}
\end{equation} for some constants $\mu^+$ and $\mu^-$ satisfying $0<\mu^-<\mu^+$.
\jj{This function is a good growth rate function, and the solution $x$ to
	the corresponding Cauchy problem \eqref{cauchypb} is the continuous piecewise affine function
	satisfying $x(0) = 0$ and $x'(t) = \sqrt{2r(x(t))}$ for each $t$ at which $x$
	is differentiable.}

First, consider the Fisher-KPP equation \eqref{eq:FKPP},  however, in order to be able to state results from the literature, assume that non-linearity appearing in \eqref{eq:FKPP} has the form $f(u)=u(1-u)$ instead --- we believe that this should not have an impact on what follows. For $\eps>0$, denote by $c^*_\eps$ the speed of propagation, i.e.~the smallest number such that a pulsating travelling front with speed $c$ exists for all $c\geq c^*_\eps$ \cite{Berestycki:2005tt}. It has been shown \cite{Nadin:2010vd} that $c_\eps^*$ is nonincreasing with respect to $\eps$ and bounded. Therefore, it converges to some limit $c_0^*$ as $\eps\to 0$. An explicit formula for $c_0^*$ has been computed in \cite{Hamel:2011ue} with the viscosity solution method. Under assumption  (\ref{hyp:rperiodic}), this expression is even more explicit and given \cite{Hamel:2010wf} by
\begin{eqnarray*}
	c^{*}_0&=& 2\sqrt{2}\frac{(\mu^+)^2+(\mu^-)^2+(\mu^+-\mu^-)\sqrt{\Delta}}{(\mu^++\mu^-+2\sqrt{\Delta})^{3/2}}\\
	\Delta&=&(\mu^+)^2+(\mu^-)^2-\mu^-\mu^+.
\end{eqnarray*}
On the other hand, the limit speed of the ODE \eqref{cauchypb} is the harmonic mean between the two speeds $\sqrt{2\mu^+}$ and $\sqrt{2\mu^-}$:
\begin{equation*}
	c^{ODE}=2\frac{\sqrt{2\mu^+\mu^-}}{\sqrt{\mu^-}+\sqrt{\mu^+}}.
\end{equation*}
Therefore, as noted in \cite{Hamel:2010wf}, $c^{ODE}$ is strictly smaller than the quadratic mean $\sqrt{\mu^++\mu^-}$, which corresponds to the homogenization limit $\nc^*:=\lim_{\eps\to\infty} c_\eps^*$ \cite{Smaily:2009vv}. Summarizing, we have, using again the fact that $c_\vep^*$ is non-increasing with respect to $\vep$,
\[
c_0^*\ge \nc^*> c^{ODE}.
\]

\jj{We have simulated our particle system for
	\begin{equation*}
		\nu_r = (1-r\delta t )\delta_0+r\Delta t \delta_1,
	\end{equation*}
	and $r$ as  in \eqref{hyp:rperiodic}}, with   $\mu^+=3$ and $\mu^-=0.1$. In this case, $c_0^* = 1.901..$ and $c^{ODE} = 0.756..$ The results of the simulations are shown in Figure \ref{fig:chap2}. They illustrate the behaviour of the process under the limits $\eps\to0$ and $K\to\infty$. We observe that when $\eps$ is fixed and $K$ goes to infinity, the position of the rightmost particle in a simulation of the process approaches the front of the solution of the PDE. On the other hand, when $K$ is fixed and $\eps$ tends to $0$, it tends to the solution of the ODE, in line with Theorem~\ref{main_theorem}.
\begin{figure}
	\begin{center}
		\begin{minipage}[t]{0.5\textwidth}
			\includegraphics[trim=11.5 2 7 3, clip,width=\textwidth]{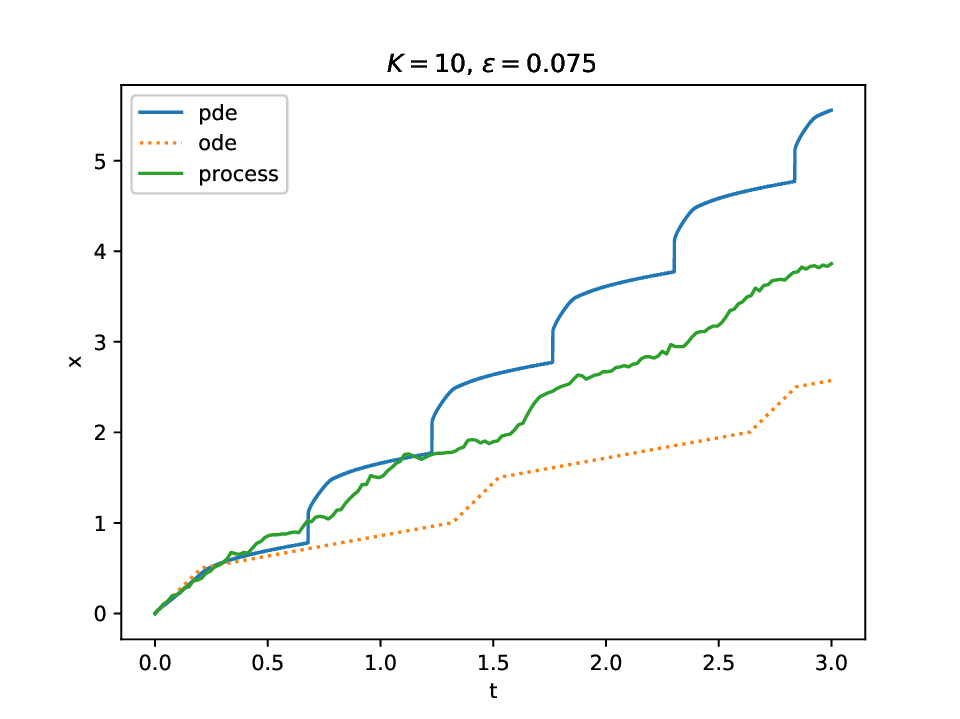}
		\end{minipage}% 
		\begin{minipage}[t]{0.5\textwidth}
			\includegraphics[trim=11.5 2 7 3, clip,width=\textwidth]{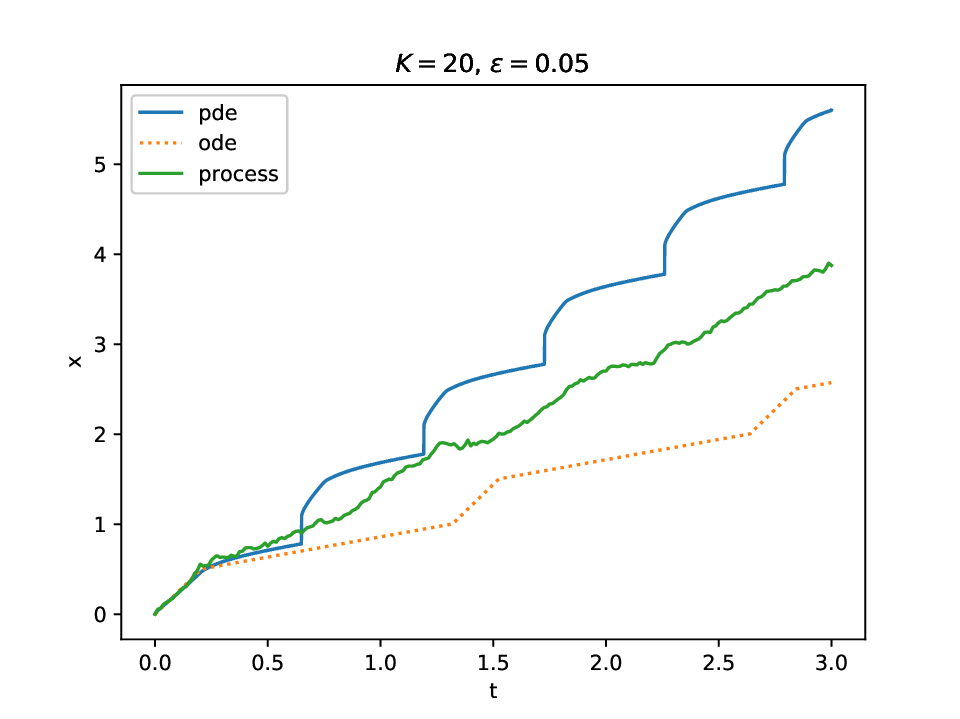}
		\end{minipage}
		
		\begin{minipage}[t]{0.5\textwidth}
			\includegraphics[trim=11.5 2 7 3, clip,width=\textwidth]{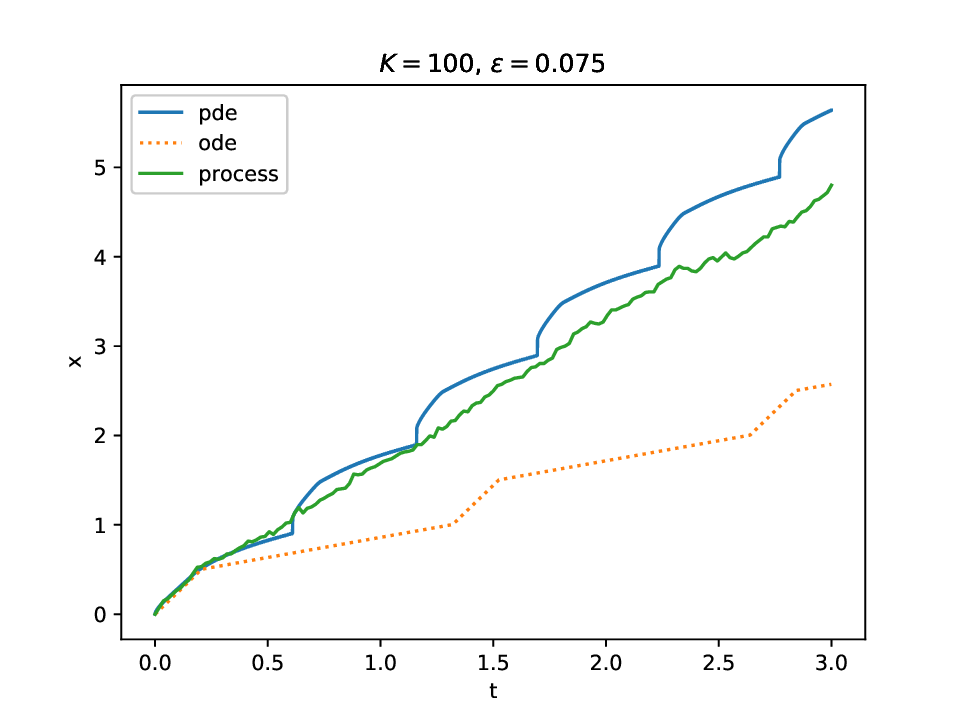}
		\end{minipage}%
		\begin{minipage}[t]{0.5\textwidth}
			\includegraphics[trim=11.5 2 7 3, clip,width=\textwidth]
			{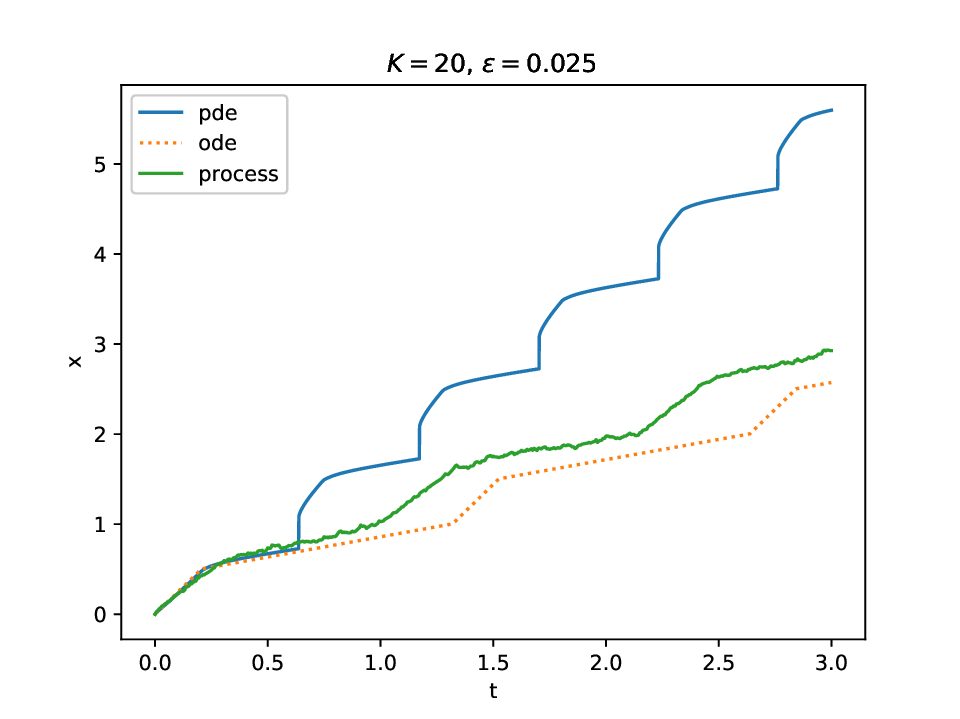}
		\end{minipage}
		
		\begin{minipage}[t]{0.5\textwidth}
			\includegraphics[trim=11.5 2 7 3, clip,width=\textwidth]
			{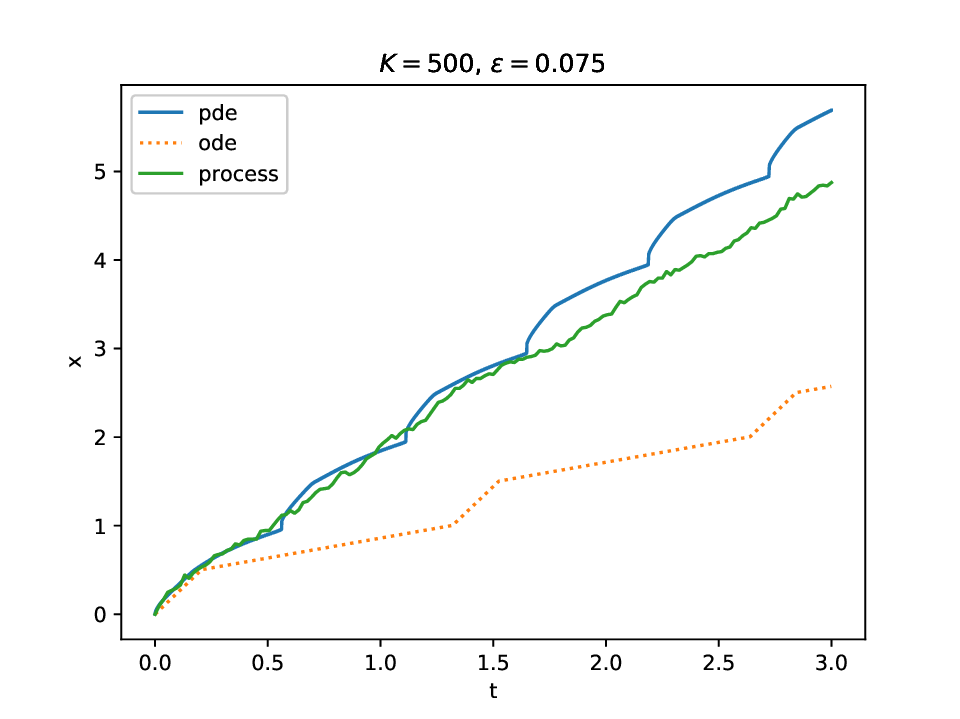}
		\end{minipage}%
		\begin{minipage}[t]{0.5\textwidth}
			\includegraphics[trim=11.5 2 7 3, clip,width=\textwidth]{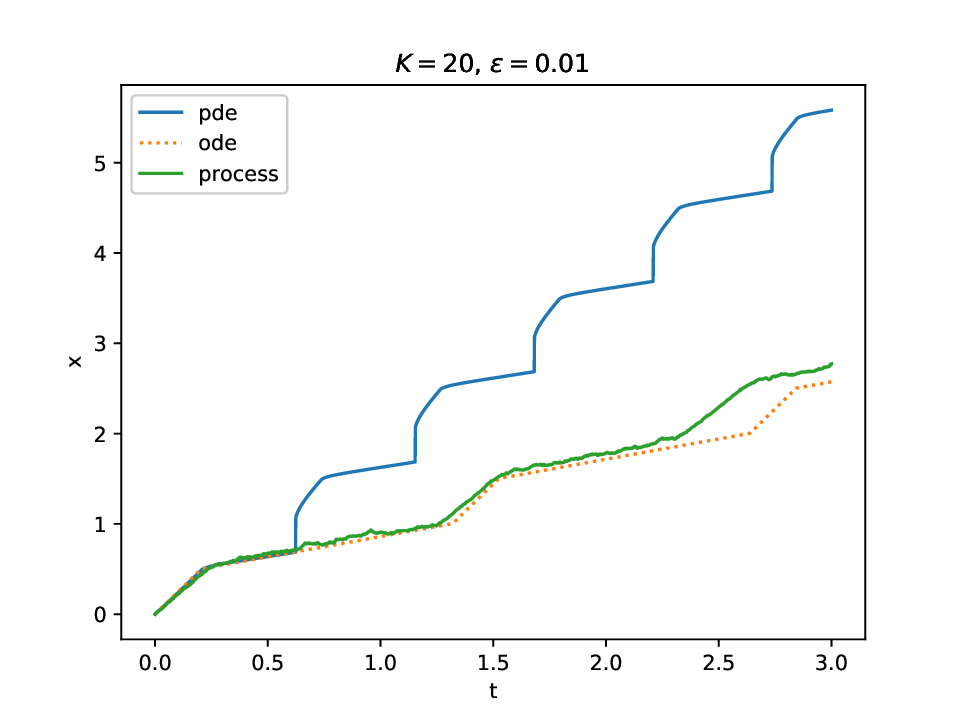}
		\end{minipage}
		\caption[The two double limits of the IBM (Chapter \ref{chap2})]{Rescaled position of the
			rightmost particle in simulations of the process defined in Section \ref{sec:simp:model}
			(green line) for different values of $(K,\eps)$. Left column: fixed $\eps$, increasing $K$,
			right column: fixed $K$, decreasing $\eps$. The growth rate $r$ is a $1$-periodic function
			of the form (\ref{hyp:rperiodic}) with $\mu^+=3$ and $\mu^-=0.1$, and the initial
			configuration is given by \jj{$n_k(i)=\mathbf{1}_{x_i\leq0}$ (which implies
				that $X_0^*=0$)}.
			The orange dotted line is the graph of the solution of the ODE (\ref{cauchypb}).
			The blue solid line is the position of the front $x(t)=\sup\{x\in\mathbb{R}:u(t,x)>\frac{1}{K}\}$ for $u$ solution to $u_t(t,x)=\frac{1}{2}u_{xx}(t,x)+r(\eps t,\eps x)u(t,x)(1-u(t,x))
			$, with initial condition $\chi(x)=\mathbf{1}_{x<0}$.
			\jj{Note that the orange dotted line lies below the blue one (see Lemma~\ref{lem:comp_HJ_ODE})}}
		\label{fig:chap2}
	\end{center}
\end{figure}

\subsection{Relation with other stochastic models}
The model we consider in this work is an example of a microscopic model for front propagation. Such models have seen considerable interest in the last two decades in mathematics, physics and biology. The prototypical model of front propagation is the Fisher-KPP equation, a semi-linear parabolic partial differential equation which admits so-called \emph{travelling waves}, i.e.~solutions which are stationary in shape and which travel at constant speed. Many microscopic models of front propagation (in homogeneous environments) can be seen as noisy versions of the Fisher-KPP equation, see e.g.~the reviews \cite{Panja2004,Kuehn2019}. A rich theory originating in the work of Brunet, Derrida and co-authors \cite{Brunet1997,Brunet2006,Brunet2006a} has put forward some universal asymptotic behavior when the local population density at equilibrium, $K$, goes to infinity. First, the speed of propagation of such systems admits a correction of the order $O((\log K)^{-2})$ compared to the limiting PDE. Second, the genealogy at the tip of the front is described by the Bolthausen--Sznitman coalescent over the time scale $(\log K)^3$, in stark contrast to mean-field models where the genealogy evolves over the much longer time scale $K$ and is described by Kingman's coalescent. These facts have been proven rigorously for several models \cite{Berard2010,Mueller2010,Berestycki2010,nbbm,Pain2015,Cortines2016a}.

Compared to the case of homogeneous environment, the model in heterogeneous environment considered in this paper has a different speed of propagation than its continuous limit, \emph{even in the limit of infinite population size}. A similar situation happens in homogeneous environment when the displacement is heavy-tailed. Such a microscopic model, with branching, competition and displacement with polynomial tails, was considered in \cite{Bezborodov2018}. For their model, the authors show the existence of a phase transition in the tail exponent of the displacement law: when the exponent is sufficiently large, the model grows linearly, whereas it grows superlinearly when the exponent is small. On the other hand, the continuous limit of the model, a certain integro-differential equation, always grows exponentially fast regardless of the exponent. This example, as well as the one considered in this paper, show that microscopic probabilistic models of front propagation or of spatial population dynamics can exhibit quite different qualitative behavior than their continuous limits. We believe this to be an exciting direction for future research.

Another body of literature is concerned with the behavior of locally regulated population models at equilibrium, i.e.~in the bulk. Basic questions like survival and ergodicity are often studied using two methods stemming from interacting particle systems: duality and/or comparison with simpler models such as directed percolation \cite{Etheridge2004,Hutzenthaler2007,Blath2007,Birkner2007,Berestycki2009a}. The genealogy of such systems is also of interest. Some related models from population genetics admit an explicit description of their genealogy in terms of coalescing, and sometimes branching random walks. Their behavior is therefore dimension-dependent, see e.g.~\cite{Barton2013} for a survey. For the one-dimensional model considered here, we expect the same to happen: the genealogy should be described by random walks coalescing when they meet at a rate proportional to $1/K$, where $K$ is the local population density at equilibrium. In particular, on the time-scale $K$, its scaling limit should be a system of Brownian motions which coalesce at a rate proportional to their intersection local time, whereas on a larger time-scale, corresponding to small population density, it should be described by the Brownian web. See \cite{Schertzer2016,Etheridge2017,Birkner:2019aa} for recent results on related models.

Finally, we point out that our model has been defined in such a way that it is a \emph{monotone} particle system. This property is crucial in order to compare the process to other, simpler processes. It is the analogue of the parabolic maximum principle for PDEs. Its absence causes significant technical difficulties, see for example \cite{brwnlc} which studies (homogeneous) branching random walks with non-local competition.

%%%%%%%%%%%%%%%%%%%%%%%%%%%%%%%%%%%%%%%%%%%%%%%%%%%%%%%%%%%%%%%%%%%

\section{General bounds on the speed of invasion}

\sectionmark{Definition and results}
\label{model}
\subsection{Generalization of the model}
\label{subdefmod}
\jj{
	In this section, we introduce a version of the model defined in
	the introduction and state intermediate results under weaker hypotheses on the dynamics of the particle system.}

As before, we consider a system of
interacting particles $\textbf{X}$ on the rescaled lattice $\Delta x\cdot\mathbb{Z}$,
evolving in discrete time $(t_{k})_{k\in\mathbb{N}}$, where $t_{k}=k\Delta t$.
The state of the system at time $t_{k}$ (or equivalently at generation $k$)
is described by its configuration $n_{k}:\mathbb{Z}\rightarrow \mathbb{N}$,
where the integer $n_{k}(i)$ counts the number of particles living on the
site $x_{i}=i\Delta x$. At each time step, the particles give birth to a
random number of children and die. After their birth, the offspring migrate
independently. As before, we denote by $X_k^*$ the position of the right-most
particle at time $t_k$.

The parameters of the model, $r$, $\eps$, $K$, are the same as in the introduction.
Furthermore, the migration step does not change: particles migrate according to
the discretized normal distribution $\mu$ defined in \eqref{displaw}.  However,
the reproduction and competition steps are generalized.
To this end, we suppose that we are given a family $(\nu_{r,n,K})$ of probability distributions on $\mathbb N$. Reproduction and competition are then contracted into a single step as follows: if a site $x_i$ is inhabited by $n$ particles at time $t_k$, these particles are replaced by a random number of offspring distributed according to $\nu_{r(\eps t_{k},\eps x_{i}),n_{k}(i),K}$, independently on all sites. Once the population is renewed on all sites, the particles migrate independently according to $\mu$. The resulting configuration is $n_{k+1}$.

\jj{The proof of Theorem~\ref{main_theorem} requires to establish an upper bound
	(Proposition~\ref{thupperbound})
	and a lower bound (Proposition~\ref{thlowerbound1}) on the position of the rightmost particle $X^*$. The upper bound
	holds under relatively weak assumptions, which we now introduce.}

\begin{hyp} \label{Assumption1}
	\jj{
		We assume that  $\nu_{r,n,K}$ is stochastically
		dominated by $\tilde{\nu}_{n,r,K}$, where $(\tilde{\nu}_{n,r,K})$ is  a family of reproduction laws
		satisfying the three following conditions:
		\begin{enumerate}
			\item 
			$(\tilde \nu_{r,n,K})$ is increasing with respect to $r$ and
			with respect to $n$.
			\item there exists a family of discrete probability distributions on $\mathbb{N}$,
			denoted by $(P_r)$,  such that $(P_{r})$ is increasing with respect
			to $r$,
			$\mathbb{E}[P_{r}]=1+r\Delta t$ and, for all $K>0$,  $P_r^{*n}$ stochastically dominates $\tilde{\nu}_{r,n,K}$.
			\item there exists a probability distribution on $\mathbb{N}$, denoted by
			$\bar{\nu}$, with finite expectation and such that $\bar{\nu}^{*K} $
			stochastically dominates $\tilde{\nu}_{r,n,K}$, for all $n\in\mathbb{N}$.
		\end{enumerate}
	}
\end{hyp}
\jj{In words, the first point ensures that the particle system is stochastically
	dominated by a monotone system (though it is not itself monotone),
	and the third point that all reproduction
	laws are uniformly bounded by a law that only depends  on the carrying capacity of the system. }

\jj{The proof of the lower bound requires stronger assumptions:}
\begin{hyp} \label{Assumption2}
	Let $(\nu_r)$ be a family of reproduction laws such that
	\begin{enumerate}
		\item $(\nu_{r})$ is continuous and increasing with respect to $r$,
		\item $\nu_r(0)=0$,
		\item $\sum k \nu_r(k)=1+r\Delta t$.
		
	\end{enumerate}
	\jj{Let $(Y_i)_{i\in \mathbb{N}}$ be a sequence of i.i.d. random variables of
		law $\nu_r$ and let $\tilde{\nu}_{n,r,K}$ denote the law of the random variable}
	\begin{equation*}
		\left(\sum_{i=1}^n Y_i
		\right)\wedge K.
	\end{equation*}
	\jj{ We assume that  $\nu_{r,n,K}$ stochastically
		dominates the $\tilde{\nu}_{n,r,K}$}.
\end{hyp}

\begin{rem}
	Note that under Assumption \ref{Assumption2} the process cannot die out.
	This is a simplifying assumption without which the proof of
	Proposition~\ref{thlowerbound1}
	would be more involved.
\end{rem}
\begin{rem}
	\jj{While neither Assumption~\ref{Assumption1} nor~\ref{Assumption2} requires
		the interacting particle system to be
		monotone, both imply that the system is bounded by monotone systems.}
	\jj{Note that Assumption~\ref{Assumption1} does not specify the precise form of
		the competition mechanism for the upper-bound system, in contrast to
		Assumption~\ref{Assumption2}.
		In addition, the process can go extinct under Assumption~\ref{Assumption1}, but not under
		Assumption~\ref{Assumption2}.
		
		As we will see, the proof of Proposition~\ref{thupperbound} and Proposition~\ref{thlowerbound1}
		rely on couplings with branching random walks.
		For the upper bound, these couplings only require Assumption~\ref{Assumption1}.
		The lower bound is more delicate and Assumption~\ref{Assumption2}
		proves particularly useful. Indeed, under Assumption~\ref{Assumption2},
		the system behaves like a branching system as long as the population size at
		each site remains below $K$.
	}
\end{rem}
\begin{rem}
	\jj{The system defined in Section~\ref{sec:simp:model} satisfies
		Assumption~\ref{Assumption1} and Assumption~\ref{Assumption2}.}
\end{rem}
\begin{ex}
	Suppose $\Delta t \le 1/\nr$. The family of distributions $(\nu_r)$ with
	$\nu_r=r\Delta t \delta_2+(1-r\Delta t)\delta_1$ satisfies Assumptions~\ref{Assumption2}.1,
	\ref{Assumption2}.2 and \ref{Assumption2}.3.
\end{ex}
\begin{ex}
	The family of Poisson distributions
	with parameters $$\lambda_{n,r,K} = n\left(1+r\Delta t \left(1-\frac{n}{K}\right)\right)_+$$
	forms a family of reproduction laws $(\nu_{r,n,K})$ which satisfies
	Assumption~\ref{Assumption1}
	but which does not satisfy Assumption \ref{Assumption2}.
	Consider the family of Poisson distributions with parameters
	$\tilde{\lambda}_{n,r,K}= (n(1+r\Delta t))\wedge
	(K(1+\nr\Delta t)^2/(4\rinf\Delta t)).$ One can easily check that this
	family of distribution satifies the three conditions of
	Assumption~\ref{Assumption1} (point 2 is satisfied with $P_r$ the Poisson
	distribution of parameter $1+r\Delta t$ and point 3 is satisfied with
	$\bar \nu$  the Poisson distribution with parameter
	$(1+\nr\Delta t)^2/(4\rinf\Delta t)$), and that it stochastically dominates
	$(\nu_{r,n,K})$.
	%Indeed, point 1 is satisfied with $P_r$ the Poisson distribution of parameter
	%$1+r\Delta t$. Moreover, a direct calculation shows that $\lambda_{n,r,K}
	%	\leq K(1+\nr\Delta t)^2/(4\rinf\Delta t)$, so that point 2 is satisfied
	%(we choose $\bar \nu$ to be the Poisson distribution with parameter $(1+\nr\Delta t)^2/(4\rinf\Delta t)$).
	On the other hand, this family does not satisfy any of the points of Assumption~\ref{Assumption2}.
	
	\jj{If $(P_r)$ is as in Assumption~\ref{Assumption1},
		the family of reproduction laws $(P_r^{*n})$ satisfies Assumption~\ref{Assumption2}
		but does not satisfy Assumption \ref{Assumption1}. }
\end{ex}

\subsection{Results}
% In Theorem~\ref{Thintro}, we have claimed a control on the propagation speed of the process $(n_k)_{k\in \mathbb{N}}$. In this section, we generalise this result: Theorem~\ref{thupperbound}, under Assumption~1, provides an explicit upper bound on the propagation of the population, while Theorem~\ref{thlowerbound1}, under Assumption~2, provides a lower bound. These theorems prove in turn Theorem~\ref{Thintro}, since the processes considered then satisfy both Assumption~1 and Assumption~2. %   The proof of Theorem \ref{Thintro} is divided in two intermediate results stated below, which respectively correspond to an upper and a lower bound on the process $(X_k)$.  

\begin{proposition}[Upper bound on the propagation speed] \label{thupperbound}
	\jj{
		Assume that $r$ is a good growth rate function and that Assumption~\ref{Assumption1} holds.
		Assume that $X_0^*=0$.
		%Let $\Delta t\leq \left(\nr^{-1}\wedge \frac{1}{2}\right)$ and
		%$\Delta x\leq \frac 1 5\sqrt{2(\log 2) \ir}\Delta t$.
		Let $T>0$ and $\delta>0$.
		There exist $\Delta t_\delta>0$, $C_\delta>0$ such that, if
		\begin{equation*}
			%\label{Assumptiondtdx:1}
			\Delta t< \Delta t_\delta, \quad \Delta x<C_\delta \Delta t\tag{H},
		\end{equation*}
		the following holds:
		there exists $\alpha>0$ and $\eps_0>0$ such that,	}
	\begin{equation}
		\forall \eps<\eps_0, \quad \forall K\geq 1, \quad
		\mathbb{P}\left(\exists k\in \left\llbracket0,
		\left\lfloor \frac{T}{\eps \Delta t}\right\rfloor \right\rrbracket :
		\eps X^*_k>x(k\eps\Delta t)
		+\delta\right)
		\leq Ke^{-\frac{\alpha}{\eps}}.
	\end{equation}
\end{proposition}
\begin{proposition}[Lower bound on the propagation speed]\label{thlowerbound1}
	Assume that $r$ is a good growth rate function and that Assumption~\ref{Assumption2} holds. Assume that $X_0^*=0$.
	Let $\delta >0$ and $T>0$. There exist $\Delta t_\delta>0$, $C_\delta>0$ such that, if
	\begin{equation*}
		\Delta t< \Delta t_\delta, \quad \Delta x<C_\delta \Delta t\tag{H}, %\label{Assumptiondtdx:1}
	\end{equation*}
	the following holds: there exists $K_0>0$ and $\eps_0>0$ such that,
	\jj{\begin{equation*}
			\forall K>K_0, \quad 	\forall \eps<\eps_0, \quad
			\mathbb{P}\left(\exists k \in\left \llbracket 0,  \left\lfloor \frac{T}{\eps\Delta t}\right\rfloor \right\rrbracket : \eps X^*_k< x(k\eps \Delta t)-\delta \right)\leq \sqrt{\vep}.
	\end{equation*}}
	% and  $x$ the unique global solution of  (\ref{cauchypb}) on $[0,T]$.
\end{proposition}

\begin{rem} \label{rem:deltax}
	\jj{
		We believe that Proposition~\ref{thlowerbound1} holds in fact for every $K\ge 1$. Indeed, choosing $\Delta x$ small has a similar effect than choosing $K$ large in that it increases the density of particles in a unit interval at equilibrium (which one expects to be of the order of $K/\Delta x$). We believe that one could extend the proof of Proposition~\ref{thlowerbound1}  to this case, at the expense of more elaborate arguments, see Remark~\ref{rem:K_Deltax} below.
	}
	
\end{rem}

\begin{rem}
	Note that in both theorems above, $\vep_0$ is chosen small,
	depending on all the other parameters. This indeed corresponds to the
	limit where we first let $\vep\to0$, then $K\to\infty$ (or $\Delta x\to 0$,
	see Remark~\ref{rem:deltax}), as mentioned in Section~\ref{sec:det}.
	
\end{rem}

\begin{rem}
	\jj{Our results depend on the initial configuration only through the position of the rightmost particle.}
\end{rem}
%%%%%%%%%%%%%%%%%%%%%%%%%%%%%%%%%%%%%%%%%%%%%%%%%%%%%%%%%%%%%%%%%%%

\subsection{Notation}\label{sec:not}
We introduce some notation that will be used throughout the article.
For each particle $u$ living in the process $\mathbf{X}$, we denote by $X_u$ its position. The process $\mathbf{X}$ will be compared to several branching random walks (BRW). Likewise, for a branching random walk $\mathbf{\Xi}$, $\Xi_u$ refers to the position of the particle $u$. In both cases, we denote by $|u|$ the generation of $u$.
We further define the following constants:
\begin{equation}
	\label{def:constant}
	\gamma=\log(2) \quad \textnormal{ and }\quad C_0=16\gamma^{-\frac{1}{2}}
	.
\end{equation}
\subsection{Structure of the proof}
\jj{
	In Section \ref{coupsect}, we state a general coupling lemma allowing us to compare
	systems with different reproduction mechanisms. In particular, it allows us to
	compare the interacting particle system with branching random walks.
	
	Section~\ref{partub} contains the proof of Proposition~\ref{thupperbound}
	(upper bound). The proof uses a Trotter-Kato-type scheme and local comparisons
	with branching random walks.
	The proof of Proposition~\ref{thlowerbound1} (lower bound) is given in
	Section~\ref{partlow} and uses a martingale argument together with first and
	second moment estimates.
	
	Appendix~\ref{SecBRW} recalls  known results on branching random walks and
	provides explicit estimates on the rate functions of the branching random walks
	used in this article. Finally, Appendix~\ref{sec:EDO} recalls known results
	on the Euler scheme for the solution $x$ of \eqref{cauchypb}.
}

%%%%%%%%%%%%%%%%%%%%%%%%%%%%%%%%%%%%%%%%%%%%%%%%%%%%%%%%%%%%%%%%%%%

\section{A coupling lemma}
\label{coupsect}
Let $S^1$ and $S^2$ be two systems of interacting particles on $\mathbb{Z}$ whose configurations, $(n_{k}^{1})$ and $(n_{k}^{2})$, evolve as follows. At time $t_{k}$, the particles of $S^1$ (resp. $S^2$) living on $x_{i}$, are replaced by a random number of offspring distributed according to $(p_{l}^{1}(n^{1}_{k}(i),x_{i},t_{k}))_{l\in\mathbb{N}}$ (resp. $(p_{l}^{2}(n^{2}_{k}(i),x_{i},t_{k}))_{l\in\mathbb{N}}$). Once the population is renewed on each site, the particles migrate independently according to $\mu$ in both processes. Furthermore, let $\tau$ be a stopping time for the process $S^2$ (which may be infinite). We say that $S^1$ \emph{dominates} $S^2$ until time $\tau$ if
\begin{equation*}
	\mathbb{P}(\forall i\in \mathbb{Z}, \forall k\le \tau: n_k^{1}(i)\geq n_k^{2}(i))=1. \label{domstoch}
\end{equation*}
If $\tau = +\infty$ almost surely, we simply say that $S^1$ dominates $S^2$.

The following lemma establishes a coupling $(\tilde{S}^1,\tilde{S}^2)$ of $S^1$ and $S^2$ such $\tilde{S}^1$ dominates $\tilde{S}^2$, provided that the reproduction laws $p^1$ and $p^2$ and the initial conditions meet certain conditions.

\begin{lemma}
	\label{lem:coupl}
	Assume that
	\begin{enumerate}
		\item The initial configurations  satisfy $n^1_0(i)\geq n^2_0(i),$ for all $ i\in \mathbb{Z}$.
		\item \jj{The system $S^1$ is monotone:
			for all $(m,n)\in\mathbb{N}^{2}$ such that $n\geq m,$
			\begin{equation}
				\sum_{q\geq l}p_q^1(n,t_k,x_i)\geq
				\sum_{q\geq l}p_q^1(m,t_k,x_i),
				\quad \forall l\in \mathbb{N}.
				\label{Assumption3}
		\end{equation}}
		\item Almost surely with respect to the process $S^2$, for every $k<\tau$, $i\in \mathbb{N}$,
		\begin{equation}
			\label{Assumption4}
			\quad  \sum_{q\geq l}p_q^1(n_k^2(i),t_k,x_i)\geq
			\sum_{q\geq l}p_q^2(n_k^2(i),t_k,x_i),  \quad
			\forall l\in\mathbb N,\; \forall i\in \mathbb{Z}.
		\end{equation}
	\end{enumerate}
	Then, there exists two processes $\tilde{S}^1$ and $\tilde{S}^2$,
	distributed as $S^1$ and $S^2$, such that $\tilde{S}^1$ dominates
	$\tilde{S}^2$ until time $\tau$.
\end{lemma}

\begin{proof}
	We first assume that $\tau= +\infty$. We construct a probability space
	supporting two processes $\tilde{S}^1$ and $\tilde{S}^2$, distributed as
	$S^1$ and $S^2$, such that $\tilde{S}^1$ dominates $\tilde{S}^2$.
	We first consider a set of particles organised according to $n^1_0$. On each
	site $x_i$,  $n^2_0(i)$ of these particles are coloured blue. The remaining
	individuals are coloured red. The initial population of the  process
	$\tilde{S}^2$ (resp.~$\tilde{S}^1$) is defined as the set of blue
	(resp.~red and blue) particles.
	
	We then construct the first generation ($k=1$) as follows. Consider a site
	$x_i$ such that $n_0^1(i)=\tilde{n}^{1}_0(i)\neq 0$: this site is inhabited
	by $k_1$ red particles, $k_2$ blue particles such that $k_{1}+k_{2}>0$. Draw
	a uniform random variable $\mathcal{U}$ on $[0,1]$ and consider
	$l_{1}(\omega)$ and $l_{2}(\omega)$, two integers defined by
	\begin{equation*}
		l_1(\omega)=\max\left\{n\in \mathbb{N}: \mathcal{U}(\omega)\geq
		\sum_{q=1}^{n-1}p_q^1(k_1+k_2,t_1,x_i)\right\} \,, \end{equation*}
	and likewise,
	\begin{equation*}
		l_2(\omega)=\max\left\{n\in \mathbb{N}: \mathcal{U}(\omega)\geq
		\sum_{q=1}^{n-1}p_q^2(k_2,t_1,x_i)\right\}.
	\end{equation*}
	\jj{By definition of $l_1$ and $l_2$, if $l_2\geq l_1-1$,
		Equations  (\ref{Assumption4}) and (\ref{Assumption3}) imply that}
	\jj{ \begin{multline}
			\label{eq:proof_coupling}
			\mathcal{U}(\omega)>\sum_{q=1}^{l_2}p^2_q(k_2,t_1,x_i)\geq
			\sum_{q=1}^{l_1-1}p^2_q(k_2,t_1,x_i)
			\geq
			\sum_{q=1}^{l_1-1}p^1_q(k_2,t_1,x_i)\\
			\geq  \sum_{q=1}^{l_1-1}p^1_q(k_1+k_2,t_1,x_i)
			\geq\mathcal{U}(\omega).
	\end{multline}}
	\jj{Thus, we deduce that $l_1(\omega)\geq l_2(\omega)$.
		We then generate $l_1(\omega)$ individuals on $x_i$ and $l_2(\omega)$ of
		them are coloured blue. The remaining ones are painted red.} We repeat
	this construction until the population is renewed on each non-empty
	site $x_i$. Then, all the particles (red and blue ones) migrate independently
	according to $\mu$. After the migration phase, the first generation of
	$\tilde{S}^1$ (resp. $\tilde{S}^2$)  is  the set of blue (resp. red and blue) particles.
	The following generations are constructed similarly by induction on $k$.
	
	If $\tau$ is an arbitrary stopping time, since it is a measurable function
	of the process $S^2$, it can be transferred to the probability space
	constructed above, to become a stopping time for the process $\tilde S^2$.
	The above chain of inequalities then still hold for every $k < \tau$ and
	the statement follows.
\end{proof}
\begin{rem} 
	\label{rem:coupling_lemma}
	The same result holds when the roles of $S^1$ and $S^2$ are swapped.
	Indeed, a similar coupling can be constructed if $\tau$ is a stopping time
	for the process $S^1$, $S^2$ is monotone and \eqref{Assumption4} holds
	for any configuration of the system $S^1$.
	In this case, the proof goes along the same lines,
	with \eqref{eq:proof_coupling} replaced by the following inequalities:
	\jj{ \begin{multline*}
			\mathcal{U}(\omega)>\sum_{q=1}^{l_2}p^2_q(k_2,t_1,x_i)\geq
			\sum_{q=1}^{l_2}p^2_q(k_1+k_2,t_1,x_i)
			\geq
			\sum_{q=1}^{l_2}p^1_q(k_1+k_2,t_1,x_i)\\
			\geq  \sum_{q=1}^{l_1-1}p^1_q(k_1+k_2,t_1,x_i)
			\geq\mathcal{U}(\omega).
	\end{multline*}}
\end{rem}

%%%%%%%%%%%%%%%%%%%%%%%%%%%%%%%%%%%%%%%%%%%%%%%%%%%%%%%%%%%%%%%%%%%

\section{Proof of Proposition \ref{thupperbound}: Upper bound on the propagation speed}
\sectionmark{Upper bound on the invasion speed}
\label{partub}

In this section, we give an upper bound on the invasion speed of the process $\textbf{X}$  under Assumption \ref{Assumption1}. The idea of the proof of Proposition \ref{thupperbound} is to first establish a coupling between $\pr$ and a process without competition. The absence of competition in this process then allows to compare it with several branching random walks, for which we can easily control the position of their rightmost particles (See Section \ref{sec:BRW}).

\subsection{An estimate on the branching random walk}\label{sec:BRW}
Let $X_1$ be a random variable of law $\mu$. We define the function $\Lambda$ by
\begin{equation}
	\mathbb{E}\left[e^{\lambda X_1}\right]=e^{\Lambda(\lambda)}, \quad \forall \lambda \in \mathbb{R},
\end{equation}
and denote by $I$ its convex conjugate:
\begin{equation}
	I(y)=\sup_{\lambda \in \mathbb{R}} \left(  \lambda y -\Lambda (\lambda)\right), \quad \forall y\in \mathbb{R}.
	\label{convconju}
\end{equation}
Note that $\mu$ from \eqref{displaw} has super-exponential tails and that its support is unbounded both to the right and to the left, which implies that both $\Lambda$ and $I$ are finite and strictly convex on $\mathbb{R}$. Furthermore, $I$ has a minimum at $\mathbb{E}[X_1]=0$ (and $I(0)=0$) so
%$I(0)=I(\mathbb{E}[X_1])=0$ 
that $I$ is decreasing on $(-\infty,0)$ and increasing on $(0,\infty)$, as a consequence of strict convexity. We also define
\begin{equation}\label{eq:Lambda0}
	\Lambda_{0 }(\lambda)  :=  \log\int _\mathbb{R} \frac{1}{\sqrt{2\pi \Delta t}}e^{\frac{-x^2}{2 \Delta t}+\lambda x}dx=\frac{ \Delta t}{2} \lambda ^2,  \quad \forall \lambda\in \mathbb{R}
\end{equation}
and remark that  $I_0(y):=\sup_{\lambda \in \mathbb{R}}\left( \lambda y -\Lambda_0(\lambda)\right)$ is given by
\begin{equation}
	I_0(y)=\left(\frac y{\Delta t}\right) y - \Lambda_0\left(\frac y{\Delta t}\right)=\frac{y^2}{2\Delta t}.
	\label{io}
\end{equation}
Thus, for all $m>1$, the equation  $I_0(c)=\log(m)$ has a unique positive solution,
\begin{equation}
	c_0=\sqrt{2 \Delta t\,\log m}.
	\label{def:c0}
\end{equation}
Since $I$ increases on $(0,\infty)$, the equation $I(c)=\log(m)$ also has a unique solution $c\in(0,\infty)$. In Appendix \ref{secconvconj}, we state several results (Lemma \ref{lem51} to \ref{lem:6}) on the regularity of $I$ and $c$. These results lead to a first rough estimate on $c$.
\begin{lemma} \label{lem:estc1}
	Let $\Delta t<\nr^{-1}$ and $\Delta x<\frac{1}{16}\sqrt{2\gamma \ir }\Delta t$. Let $\br\in[\underline{r},\nr]$ and $\bc$ be the unique positive solution of $I(\bc)=\log(1+\br\Delta t).$ Then, \begin{equation}
		\frac{1}{2}\sqrt{2\gamma\rinf}\Delta t<\bc<2\sqrt{2\nr}\Delta t. \label{est:c1}
	\end{equation}
	\label{lem:ccompact}
\end{lemma}
\begin{proof}
	By concavity of the logarithm function,
	\begin{equation}
		(1-\br\Delta t) \log(1)+\br\Delta t \log(2)=\gamma\br\Delta t\leq \log(1+\br\Delta t). \label{log:concav}
	\end{equation}
	This implies that $\Delta x<\frac{1}{16}\sqrt{2\Delta t\log(1+\ir\Delta t)}$ and, according to Lemma \ref{minc} and Equation (\ref{io}), that
	%\Delta x <\frac{1}{16}\sqrt{2\gamma\ir}\Delta t,$ $\Delta x$ satisfies the assumption of Lemma \ref{minc} and thus,
	\begin{equation}
		\frac{1}{2}\sqrt{2\Delta t\log(1+\br\Delta t)}<\bc< 2
		\sqrt{2\Delta t\log(1+\br\Delta t)}.\label{c:19}
	\end{equation}
	Finally, combining (\ref{log:concav}) and (\ref{c:19}), we get that
	\begin{equation}
		\frac{1}{2}\sqrt{2\gamma\rinf}\Delta t\leq \frac{1}{2}\sqrt{2\gamma\br}\Delta t \leq \bc\leqslant2\sqrt{2\br}\Delta t\leqslant2\sqrt{2\nr}\Delta t.
		\label{c:20}
	\end{equation}
\end{proof}

\begin{lemma} \label{est:c21}
	Under the same assumptions as in Lemma~\ref{lem:estc1}, \begin{equation*}
		\left|\bc-\sqrt{2\Delta t\log(1+\br\Delta t)}\right|\leq a\Delta x,
	\end{equation*}
	with $a=16 \gamma ^{-\frac{1}{2}}\left(\frac{\br}{\ir}\right)^{\frac{1}{2}}\leq 16 \gamma ^{-\frac{1}{2}}\left(\frac{\nr}{\ir}\right)^{\frac{1}{2}}.$
\end{lemma}
\begin{proof}
	According to Lemma \ref{lem:ccompact}, $\bc$ is located in a compact interval that does not depend on $\br$. Let $\bc_0=\sqrt{2\Delta t\log(1+\br\Delta t)}$ and remark that the inequality (\ref{est:c1})  also holds when $\bc$ is replaced by $\bc_0$.   Then, since $\Delta x<\frac{1}{4}\left(\frac{1}{2}\sqrt{2\gamma\ir}\Delta t\right),$  Lemma \ref{lem:6} applied with $\underline{y}=\frac{1}{2}\sqrt{2\gamma\rinf}\Delta t$ implies that
	\begin{equation}
		\frac{1}{8}\sqrt{2\gamma\rinf}|\bc-\bc_0|\leq|I(\bc)-I(\bc_0)|\label{est:c2}.
	\end{equation}
	In addition, note that $I(\bc)=I_0(\bc_0)=\log(1+\br\Delta t)$ (see Equation (\ref{io})), so that
	\begin{equation}
		|I(\bc)-I(\bc_0)|=|I_0(\bc_0)-I(\bc_0)|\leq 2\bc_0\frac{\Delta x}{\Delta t}\leq 2\sqrt{2\br}\Delta x,\label{est:c3}
	\end{equation}
	according to Lemma \ref{lem51}. Then, combining Equations (\ref{est:c2}) and (\ref{est:c3}), we get that
	\begin{equation*}
		|\bc-\bc_0|\leq a \Delta x.
	\end{equation*}
\end{proof}
Let us now consider a branching random walk of reproduction law $P_{\br}$, for some $\br\in[\underline{r},\nr]$, and migration law $\mu$. In Lemma  \ref{lem:speedmab}, we give an estimate on the speed of propagation of this BRW. %Note that this bound is not given uniformly in $\br$, but this bound can be easily obtained from Equation (\ref{speed:mab}), thanks to Lemma \ref{est:c1} (See Corollary \ref{lem:speedmab}).  
For further details about the BRW, we refer to Appendix \ref{SecBRW}.
\begin{lemma}
	\label{lem:speedmab} Suppose the same assumptions as in Lemma~\ref{lem:estc1} hold. Consider a branching random walk of reproduction law $P_{\br}$ and displacement law $\mu$, starting with a single particle at $0$, and denote by $M_n$ the position of its rightmost particle at generation $n$. Then, for all $\eta>0$ and $A\geq 0$,
	\begin{equation*}
		\mathbb{P}(\exists n\in \mathbb{N}: M_n>(1+\eta)n\bc+A)\leq h(\eta)e^{-\frac{\sqrt{2\gamma\ir}}{8}A},
	\end{equation*}
	with $h$ defined by
	\begin{equation}\label{def:h1}
		h(\eta)=\frac{e^{-\frac{\gamma\ir\Delta t}{8}\eta}}{1-e^{-\frac{\gamma\ir\Delta t}{8}\eta}},\quad \forall \eta>0.
	\end{equation}
\end{lemma}
\begin{proof}
	Let $\eta>0$, $A\geq 0$. It will be enough to prove that
	\begin{equation}
		\mathbb{P}(\exists n\in \mathbb{N}: M_n>(1+\eta)n\bc+A)\leq g(\eta)  e^{-\frac{A\bc}{4\Delta t}}, \label{speed:mab}
	\end{equation}
	where $\bc$ refers to the unique positive solution of $I(\bc)=\log(1+\br\Delta t)$ and
	\begin{equation}
		g(\eta)=\frac{e^{-\frac{\bc^2}{4\Delta t}\eta}}{1-e^{-\frac{\bc^2}{4\Delta t}\eta}}. \label{def:g}
	\end{equation}
	Indeed, according to Lemma \ref{lem:estc1}, $\bc>\frac{1}{2}\sqrt{2\gamma \ir} \Delta t$, so that
	\begin{equation*}
		\frac{\bc^2}{4\Delta t}\geq \frac{1}{8}\gamma \ir \Delta t, \quad \textnormal{and, }\quad \frac{\bc}{4\Delta t}\geq \frac{1}{8}\sqrt{2\gamma\ir}.
	\end{equation*}
	
	We now prove \eqref{speed:mab}. For a particle $v$ living in the BRW, we denote by $\Xi_v$ its position and define  $Z_n=\sum_{|v|=n}\mathbb{1}_{\Xi_v>(1+\eta)n\bc+A}$. Markov's inequality implies that
	\begin{equation}
		\mathbb{P}(M_n>(1+\eta)n\bc+A)=\mathbb{P}(Z_n\geq 1)\leq \mathbb{E}[Z_n], \label{markovZ}
	\end{equation}
	and thanks to the many-to-one lemma (see Lemma \ref{lem:manyto1}), we know that \begin{equation}
		\mathbb{E}[Z_n]=(1+\br\Delta t)^n\mathbb{P}(\Xi_{v}>(1+\eta)n\bc+A),\label{manyto2Z}
	\end{equation}
	for any particle $v$ of the $n$-th generation. Besides, by Chernoff's bound,
	\begin{equation}
		\mathbb{P}(\Xi_{v}>(1+\eta)n\bc+A)\leq e^{n( \Lambda(\theta)-\theta ((1+\eta)\bc+A/n) )}, \quad \forall \theta \geq 0. \label{chernoff:I}
		%\mathbb{P}\left(\frac{X_v}{n}\in\right)    }\mathbb{E}\left[e^{\theta \Xi_v}\right]=e^{n( \Lambda(\theta)-\theta ((1+\eta)c+A/n) )}\leq e^{-nI((1+\eta)c+A/n))}.
\end{equation}
Remark that for $\theta<0$,
$$\theta ((1+\eta)\bc+A/n)-\Lambda((1+\eta)\bc+A/n)\leq -\Lambda((1+\eta)\bc+A/n)\leq I(0)=0.$$
Yet, $I((1+\eta)\bc+A/n)\geq 0,$ so that $I((1+\eta)\bc+A/n))=\sup_{\theta \geq 0}\theta((1+\eta)\bc+A/n)-\Lambda((1+\eta)\bc+A/n),$ and Equation (\ref{chernoff:I}) gives that
\begin{equation}
	\mathbb{P}(\Xi_{v}>(1+\eta)n\bc+A)\leq e^{-nI((1+\eta)\bc+A/n)}. \label{cramer}
\end{equation}
Moreover, thanks to Lemma \ref{lem:ccompact}, we know that $\Delta x\leq\frac{\bc}{8}$, therefore, according to Lemma \ref{lem:convexity},
\begin{equation*}
	I(\bc)+\frac{\bc^2}{4\Delta t}\left(\eta +\frac{A}{n\bc}\right)\leq I((1+\eta +A/(n\bc))\bc),\quad \forall n\in \mathbb{N}.
\end{equation*}
Thus, combining (\ref{markovZ}), (\ref{manyto2Z}) and (\ref{cramer}), we get that
\begin{equation*}
	\mathbb{E}[Z_n]\leq e^{-\frac{n\bc^2}{4\Delta t}\left(\eta+\frac{A}{n\bc}\right)},
\end{equation*}
since $I(\bc)=\log(1+\underline{r}\Delta t)$. Finally, by a union bound,
\begin{equation*}
	\mathbb{P}(\exists n\in \mathbb{N}: M_n>(1+\eta)n\bc+A)\leq \sum_{n=1}^\infty \mathbb{E}[Z_n]\leq e^{-\frac{A\bc}{4\Delta t}}\sum_{n=1}^\infty e^{-\frac{n\bc^2}{4\Delta t}\eta}=\frac{e^{-\frac{A\bc}{4\Delta t}-\frac{\bc^2}{4\Delta t}\eta}}{1-e^{-\frac{\bc^2}{4\Delta t}\eta}}.
\end{equation*}
%where we can choose $C_\eta \geq(e^{-\gamma\ir\eta/4}-1)^{-1}$ thanks to Lemma \ref{lem:ccompact}.
This proves \eqref{speed:mab} and finishes the proof of the lemma.
\end{proof}

\subsection{Invasion speed estimate: small time steps}
\label{sec:tsptvep}
In this subsection, we bound the displacement of the rightmost particle in $\pr$
after $\lfloor \eps^{-1}\rfloor$ generations, i.e.~after time $\Delta t\cdot\lfloor \eps^{-1}\rfloor$.

\begin{proposition}
\label{pr:upperbound}
Assume that Assumption~\ref{Assumption1} holds.
\jj{Suppose that $r$ is a smooth growth rate function.}
Let  $\Delta t< \nr^{-1}$ and $\Delta x<\frac 1 5\sqrt{2\gamma\ir}\Delta t$.
There exist two positive constants $\alpha$ and $\eps_0$ such that, for all $K>0$, $\eps<\eps_0$, $(k_0,i_0)\in \mathbb{N}\times\mathbb{Z}$  and $k_1=k_0+\lfloor \eps^{-1}\rfloor$,
\begin{multline}
	\mathbb{P}\left(\exists k\in\llbracket k_0,k_1\rrbracket: \eps X^*_k> \eps x_{i_0}+\sqrt{2r(\eps t_{k_0},\eps x_{i_0})}\eps\Delta t(k-k_0) +A(\Delta x+\Delta t^2)\big|X^*_{k_0}\leq x_{i_0}\right)\\\leq Ke^{-\frac{\alpha}{\eps}} ,
\end{multline}
for some constant $A>0$ that only depends on $\nr$, $\ir$ \jj{and $M$ (see \eqref{eq:ub_grad_r})}.
\end{proposition}

\begin{rem}
\jj{Note that Proposition~\ref{pr:upperbound} requires stronger regularity assumptions
	on $r$ than
	Proposition~\ref{thupperbound}.
	We will show in Section~\ref{sec:small_steps_ub} that Theorem~\ref{thupperbound} can be derived from
	Proposition~\ref{pr:upperbound} using an approximation argument.
}
\end{rem}
\begin{rem}
\label{rem:simplification_ub}
\jj{
	It is clear from Assumption~\ref{Assumption1} and Lemma~\ref{lem:coupl} that it suffices
	to prove the result for the particle system with reproduction
	law $\tilde{\nu}_{r,n,K}$.}
	\end{rem}
	\begin{proof}
Let $(k_0,i_0)\in\mathbb{Z}\times \mathbb{N}$. Throughout the proof, we will
assume that we start the process at generation $k_0$ with a deterministic
initial condition $n_{k_0}$ such that $X^*_{k_0} \le x_{i_0}$. The estimates
we obtain will not depend on this initial condition. Rewriting the
statement slightly, it will therefore be enough to show the following:
\begin{equation}
	\label{eq:proof_upperbound_toshow}
	\mathbb{P}\left(\exists k\in\llbracket k_0+1,k_1\rrbracket: X^*_k> a_k\right)\leq Ke^{-\frac{\alpha}{\eps}} ,
\end{equation}
where $a_k = x_{i_0}+\sqrt{2r(\eps t_{k_0},\eps x_{i_0})}\Delta t(k-k_0)
+A(\Delta x+\Delta t^2)/\eps$ and $\alpha,\eps_0,A$, to be defined later,
are as in the statement of the proposition.

The proof is divided into two steps. In the first step, we let the process
run for one time step, after which the expected local density of the process
can be bounded by a constant multiple of $K$, thanks to Assumption~\ref{Assumption1}.
In the second step, we control the displacement of the rightmost particle in $\pr$
between generations $k_0+1$ and $k_1$, thanks to several couplings with
processes without competition and distinguishing the particles according
to the position of their ancestor at generation $k_0+1$.

\paragraph{Step 1: Control of the population at generation $k_0+1$.}
In this step, we control the number of particles on each site in the process $\textbf{X}$ after one generation. Recall that $n_k$ denotes the configuration of the process $\pr$ at generation $k$. We denote by $N_i$ the number of individuals born on the site $x_i$ during the first reproduction phase. In addition, for $\ell\in\llbracket 1,N_i\rrbracket,$ we denote by $U^\ell_i$ the displacement of the $\ell$-th particle born on the site $x_i$ during the first reproduction phase. Therefore, we have for every $i\in\mathbb Z$
\begin{equation}
	n_{k_0+1}(i)=\sum_{j\leq i_0}\sum_{\ell=1}^{N_j}\mathbb{1}_{U^\ell_j=i-j}. \label{def:Zi}
\end{equation}
Recall that $(U^\ell_i)$ is a sequence of i.i.d. random variables of law $\mu$ and that the $(N_i)$ are stochastically dominated by the sum of $K$ i.i.d. random variables of law $\bar{\nu}$ of finite expectation $m$, by Assumption~\ref{Assumption1}. Thus, we have that
\begin{equation}
	\mathbb{E}[n_{k_0+1}(i)]\leq mK \sum_{j\leq i_0}\mathbb{P}(U=i-j),\label{EZ:1}
\end{equation} where $U$ is a random variable of law $\mu$. In particular, we get
\begin{equation}
	\mathbb{E}[n_{k_0+1}(i)]\leq mK\label{EZ:2}.
\end{equation}

We will also need a bound on the position of the maximal particle at generation $k_0+1$. Let $k\in \mathbb{N}_0$ and $i\in\mathbb Z$ such that $x_i-x_{i_0}\geq \frac{1}{2\sqrt{\eps}}+(k+1)\Delta x$. Then, we have that
\begin{eqnarray*}
	\sum_{j\leq i_0}\mathbb{P}(U=i-j)&=& \int_{(i-i_0-\frac{1}{2})\Delta x}^\infty\frac{1}{\sqrt{2\pi \Delta t}}e^{-\frac{x^2}{2\Delta t}}dx\leq \int_{\frac{1}{2\sqrt{\eps}}+k\Delta x}^{\infty}\frac{1}{\sqrt{2\pi \Delta t}}e^{-\frac{x^2}{2\Delta t}}dx\\
	&\leq& e^{-\frac{1}{2\Delta t} \left(\frac{1}{2\sqrt{\eps}}+k\Delta x\right)^2}\leq e^{-\frac
		{1}{8\eps  \Delta t}}e^{-\frac{k\Delta x}{2\sqrt{\eps}\Delta t}},
\end{eqnarray*}
where we have used the Gaussian tail estimate $\mathbb P(Z\ge x) \le e^{-x^2/2}$ for a standard Gaussian r.v.~$Z$. Using Equation \eqref{EZ:1}, we have by Markov's inequality and a union bound,
\begin{align*}
	\mathbb{P}\left(X^*_{k_0+1} >  x_{i_0}+\frac{1}{2\sqrt{\eps}} + \Delta x\right)\leq mK e^{-\frac
		{1}{8\eps  \Delta t}}\sum_{k=0}^\infty e^{-\frac{k\Delta x}{2\sqrt{\eps}\Delta t}},
\end{align*}
and therefore, as long as $\eps$ is small enough,
\begin{equation}
	\label{EZ:1.2}
	\mathbb{P}\left(X^*_{k_0+1} > x_{i_0}+\frac{1}{\sqrt \eps} \right) \leq 2mK e^{-\frac
		{1}{8\eps  \Delta t}}.
\end{equation}

\paragraph{Step 2: Between generations $k_0+1$ and $k_1$.}

As mentioned above, between generations $k_0+1$ and $k_1$, we control the process $\pr$ by another process without competition between particles. More precisely, we denote by $\pr^1$ the process defined as $\pr$, but where for every $k\in\llbracket k_0+1,k_1\rrbracket$ the reproduction law on site $x_i$ at time $t_k$ is given by the probability distribution $P_{r(\eps t_k, \eps x_i)}^{*n_k(i)}$ instead of $\nu_{r(\eps t_k, \eps x_i),n,K}$ (the migration law is still $\mu$). The position of its maximum at generation $k$ is analogously denoted by $X^{1*}_k$. By the first part of Assumption~\ref{Assumption1} and Lemma~\ref{lem:coupl}, we can couple $\pr$ and $\pr^1$ such that $\pr^1$ dominates $\pr$. Hence, in what follows, it will be enough to prove \eqref{eq:proof_upperbound_toshow} with $\pr^1$ instead of $\pr$. The advantage of working with $\pr^1$ instead of $\pr$ is the fact that $\pr^1$ satisfies the branching property, i.e.~the descendants of different individuals from the same generation evolve independently.

We first make use of the estimates from Step 1. Conditioning on the process at generation $k_0+1$, and using a union bound over the particles from that generation, with the notation $\mathbb{P}_{(\delta_i,k_0+1)}$ to mean that the process starts with one particle at site $x_i$ at generation $k_0+1$, we get for sufficiently small $\eps$,
\begin{align}
	& \mathbb{P}\left(\exists k\in \llbracket k_0+1,k_1\rrbracket :X_k^{1*}> a_k\right)\nonumber                                                                                                          \\
	& \leq \sum_{i\in\mathbb Z: x_i\le x_{i_0}+\frac{1}{\sqrt \eps}}\mathbb{E}[n_{k_0+1}(i)]\mathbb{P}_{(\delta_i,k_0+1)}\left(\exists k\in \llbracket k_0+1,k_1\rrbracket :X_k^{1*}> a_k\right)\nonumber \\
	& \qquad \qquad \qquad + \mathbb{P}\left(X_{k_0+1}^* > x_{i_0}+\frac{1}{\sqrt \eps}\right)\nonumber                                                                                                   \\
	& \leq mK\sum_{i\in\mathbb Z: x_i\le x_{i_0}+\frac{1}{\sqrt \eps}}\mathbb{P}_{(\delta_i,k_0+1)}\left(\exists k\in \llbracket k_0+1,k_1\rrbracket :X_k^{1*}> a_k\right)+2mKe^{-\frac
		{1}{8\eps  \Delta t}}.\label{eq:time_after_time}
\end{align}
Here, we used \eqref{EZ:2} and \eqref{EZ:1.2} from Step 1 in the last line.

In what follows, we bound the probability appearing on the RHS of \eqref{eq:time_after_time} for various values of $x_i$. The bound will depend on whether $x_i\ge x_{i_0}-R/\eps$ or not, where
% We first delineate an area in which the value of the function $r$ controls the propagation speed of the process: 
\begin{equation}
	R\coloneqq 2\sqrt{2\nr}\Delta t. \label{def:R}
\end{equation}
We need a few more definitions. Define
\begin{equation}
	\br=\max\left\{r(\eps t,\eps x);  \;t_{k_0}\leq t\leq t_{k_1}, \; |x- x_{i_0}|\leq \frac{2R}{\eps}\right\}.\label{def:rb}
\end{equation}
and $\bc$ the unique positive solution of
\begin{equation}
	I(\bc)=\log(1+\br\Delta t).\label{def:barc}
\end{equation}
\jj{We see from \eqref{eq:approx_r} that
	\begin{equation}
		|\sqrt{2\br}-\sqrt{2r(\eps t_{k_0},\eps x_{i_0})}|\leq
		L(\eps(t_{k_1}-t_{k_0})+2R)\leq L (1+4\sqrt{2\nr})\Delta t.
		%\label{eq:approx_r}
		\label{eq:approx_r_bis}
\end{equation}}
Denote by $\tilde{r}$ the function
\begin{equation}\label{def:rtilde}
	\tilde{r}(t,x)= \begin{cases}
		\br & \text{if}\; |x-x_{i_0}|\le \frac{2R}{\eps} \\
		\nr & \text{if} \;|x-x_{i_0}| > \frac{2R}{\eps}.
	\end{cases}
\end{equation}
Now introduce three more processes $\pr^2$, $\pr^3$ and $\pr^4$. These processes are defined as $\pr^1$, except that their reproduction law on site $x_i$ at time $t_k$, $k\in\llbracket k_0+1,k_1\rrbracket$, is given by $P_{\tilde{r}(\eps t_k, \eps x_i)}^{*n_k(i)}$, $P_{\br}^{*n_k(i)}$ and $P_{\nr}^{*n_k(i)}$, respectively. In other words, $\pr^3$ and $\pr^4$ are BRW with reproduction laws $P_{\br}$ and $P_{\nr}$, respectively.

From the definition of $\tilde{r}$, we immediately get that $\tilde{r} \ge r$. Therefore, according to Lemma \ref{lem:coupl} and Assumption \ref{Assumption1}, there exists a coupling between $\textbf{X}^1$ and $\textbf{X}^2$ so that $\textbf{X}^2$ dominates $\textbf{X}^1.$ Similarly, there exists a coupling between $\textbf{X}^2$ and $\textbf{X}^4$ so that $\textbf{X}^4$ dominates $\textbf{X}^2.$ In order to construct a coupling between $\pr^2$ and $\pr^3$, define the stopping time $\tau$ as the first time at which a particle from the process $\pr^2$ exits the interval $\left[x_{i_0}-\frac{2R}{\eps},x_{i_0}+\frac{2R}{\eps}\right]$ before time $t_{k_1}$. By the definition of $\tilde{r}$, there exists then a coupling between $\pr^2$ and $\pr^3$, such that $\pr^3$ dominates $\pr^2$ until the time $\tau$.

Let $i\in\mathbb Z$. As a consequence of the previous couplings, we have the following two bounds for the probability appearing on the RHS of \eqref{eq:time_after_time}. First,
\begin{align}
	\mathbb{P}_{(\delta_i,k_0+1)}\left(\exists k\in \llbracket k_0+1,k_1\rrbracket :X_k^{1*}> a_k\right) \le \mathbb{P}_{(\delta_i,k_0+1)}\left(\exists k\in \llbracket k_0+1,k_1\rrbracket :X_k^{4*}> a_k\right).
	\label{eq:salade_de_fruits1}
\end{align}
This bound will be used for $x_i \le x_{i_0}-R/\eps$.
Second, denoting by $\bar X_k^{3*}$ the position of the \emph{minimal} particle in the process $\pr^3$ at generation $k$, we have
\begin{align}
	\mathbb{P}_{(\delta_i,k_0+1)} & \left(\exists k\in \llbracket k_0+1,k_1\rrbracket :X_k^{1*}> a_k\right) \le \mathbb{P}_{(\delta_i,k_0+1)}\left(\exists k\in \llbracket k_0+1,k_1\rrbracket :X_k^{3*}> a_k\right)\nonumber \\
	& + \mathbb{P}_{(\delta_i,k_0+1)}\left(\exists k\in \llbracket k_0+1,k_1\rrbracket :X_k^{3*}> x_{i_0} + 2R/\eps\text{ or }\bar X_k^{3*}< x_{i_0} - 2R/\eps\right)\nonumber                  \\
	& \eqqcolon T_1+T_2.
	\label{eq:salade_de_fruits2}
\end{align}
This bound will be used for $i$ such that $x_i \in (x_{i_0}-R/\eps,x_{i_0}+1/\sqrt\eps]$.

\paragraph{Step 2a: particles close to the maximum.} Let $i\in\mathbb Z$ such that $x_i \in (x_{i_0}-R/\eps,x_{i_0}+1/\sqrt\eps]$. We  bound the RHS of \eqref{eq:salade_de_fruits2}. Assume $\eps$ is small enough so that $1/\sqrt\eps\le R/\eps$. Using first the assumption on $x_i$ and then the symmetry and translational invariance of $\pr^3$, we have
\begin{align*}
	T_2 & \le \mathbb{P}_{(\delta_i,k_0+1)}\left(\exists k\in \llbracket k_0+1,k_1\rrbracket :X_k^{3*}> x_i + R/\eps\text{ or }\added[id=j]{\bar X_k^{3*}}< x_i - R/\eps\right) \\
	& \le  2\mathbb{P}_{(\delta_0,0)}\left(\exists k\in \llbracket 0,k_1-(k_0+1)\rrbracket :X_k^{3*}> R/\eps\right).
\end{align*}
We can now apply Lemma \ref{lem:speedmab} with $A=\frac{R}{4\eps}$ and $\eta=\frac{1}{4}$. Indeed, since $\Delta x<\frac 1 5\sqrt{2\gamma \ir}\Delta t$, Lemma \ref{est:c21} applied to $\bc$ (see Equation \eqref{def:barc}) gives that
$$\bc\leq \sqrt{2\Delta t\log(1+\br\Delta t)}+a\Delta x\leq \sqrt{2\br}\Delta t +\frac{1}{5}\sqrt{2\br}\Delta t\leq \frac{6}{5}\sqrt{2\br}\Delta t \le \frac35 R$$ so that, for $k\leq k_1$
\begin{eqnarray*}
	(1+\eta)(k-k_0-1)\bc+A&\leq& \frac5{4\eps}\bc+A \leq \frac{3R}{4\eps} + \frac R{4\eps}= \frac{R}{\eps}.
\end{eqnarray*}
Lemma \ref{lem:speedmab} now gives that
\begin{equation}
	T_2 \leq 2h(1/4)e^{-\frac{\sqrt{2\gamma\ir}}{32}\frac{R}{\eps}}, \label{eq:420}
\end{equation}
with $h$ as in the statement of Lemma \ref{lem:speedmab}.

We now bound the term $T_1$ on the RHS of \eqref{eq:salade_de_fruits2}. Let $i\in\mathbb Z$ such that $x_i \in (x_{i_0}-R/\eps,x_{i_0}+1/\sqrt\eps]$. We then have by the definitions of $(a_k)$ and $\br$
\begin{equation}
	\label{eq:rayman}
	T_1 \le  \mathbb{P}_{(\delta_i,k_0+1)}\left(\exists k\in \llbracket k_0+1,k_1\rrbracket :X_k^{3*}> x_{i_0} + \sqrt{2\br}\Delta t(k-k_0)+\frac A\eps(\Delta x +\Delta t^2)\right).
\end{equation}
Now, according to Lemma \ref{lem:speedmab}, we have, for some $C>0$ not depending on $\eps$,
\begin{equation}
	\mathbb{P}\left(\exists k\in \llbracket k_0+1,k_1\rrbracket :X^{3*}_k>x_{i}+(1+\Delta t)(k-k_0-1)\bc+\frac{\Delta t^2}{\eps}\right)\leq Ce^{-\frac{\sqrt{2\gamma \ir}}{8}\frac{\Delta t^2}{\eps}}.
	\label{eq:421}
\end{equation}
Besides, recall from Lemma \ref{est:c21} that
\begin{equation}
	\bc\leq \sqrt{2\br}\Delta t+a\Delta x,\label{eq:422}
\end{equation}
with some $a\leq 16 \gamma ^{-\frac{1}{2}}\left(\frac{\nr}{\ir}\right)^{\frac{1}{2}}.$
Therefore, combining \eqref{eq:421} and \eqref{eq:422} and using that $x_i\le x_{i_0}+\frac{1}{\sqrt \eps}$ and $\Delta t \le \nr^{-1}$ and $k-k_0\leq \frac{1}{\eps}$ for $k\leq k_1$, we have
\begin{align}
	& \mathbb{P}\left(\exists k\in \llbracket k_0+1,k_1\rrbracket :\right.                                                                                               \\
	& \left.\; \; X^{3*}_k>x_{i_0}+\sqrt{2\br}\Delta t(k-k_0)+ \frac{\sqrt \eps+\sqrt{2\nr}\eps + (1+\sqrt{2\nr})\Delta t^2+a(1+\nr^{-1})\Delta x}{\eps}\right)\nonumber \\
	& \;\leq Ce^{-\frac{\sqrt{2\gamma \ir}}{8}\frac{\Delta t^2}{\eps}}.
	\label{eq:423bis}
\end{align}
Combining \eqref{eq:rayman}, \eqref{eq:423bis} and \eqref{eq:approx_r_bis}, it
follows that for  $$A = 1+\max(1 +\sqrt{2\nr}+ L(1+4\sqrt{2\nr}),a(1+\nr^{-1})),$$ we get for $\eps$ small enough,
\begin{equation}
	T_1 \le Ce^{-\frac{\sqrt{2\gamma \ir}}{8}\frac{\Delta t^2}{\eps}}.
	\label{eq:423}
\end{equation}
Combining \eqref{eq:salade_de_fruits2}, \eqref{eq:420} and \eqref{eq:423}, we now get, using again that $1/\sqrt\eps \le R/\eps$,
\begin{multline}
	\sum_{i\in\mathbb Z: x_i\in(x_{i_0}-R/\eps, x_{i_0}+\frac{1}{\sqrt \eps}]}\mathbb{P}_{(\delta_i,k_0+1)}\left(\exists k\in \llbracket k_0+1,k_1\rrbracket :X_k^{1*}> a_k\right) \\
	\le \frac{2R}{\Delta x \eps} \left(2h(1/4)e^{-\frac{\sqrt{2\gamma\ir}}{32}\frac{R}{\eps}}+ Ce^{-\frac{\sqrt{2\gamma \ir}}{8}\frac{\Delta t^2}{\eps}}\right) \le e^{-\alpha_1/\eps},
	\label{eq:step2a}
\end{multline}
for some $\alpha_1>0$ and for $\eps$ sufficiently small, and with $A$ as above.

\paragraph{Step 2b: particles far away from the maximum.} Let $i\in\mathbb Z$ such that $x_i \le x_{i_0}-R/\eps$. We bound the RHS of \eqref{eq:salade_de_fruits1}. We have for every $A\ge0$, using that $a_k \ge x_{i_0}$ for every $k\in\llbracket k_0+1,k_1\rrbracket$,
\begin{align}
	\mathbb{P}_{(\delta_i,k_0+1)}\left(\exists k\in \llbracket k_0+1,k_1\rrbracket :X_k^{4*}> a_k\right)
	& \le \mathbb{P}_{(\delta_i,k_0+1)}\left(\exists k\in \llbracket k_0+1,k_1\rrbracket :X_k^{4*}> x_{i_0}\right)\nonumber \\
	& \le \mathbb{P}_{(\delta_0,0)}\left(\exists k\le \lfloor \eps^{-1}\rfloor :X_k^{4*}> x_{i_0}-x_i\right).
\end{align}
Denote by $\nc$ the unique positive solution of $I(\nc)=\log(1+\nr \Delta t)$.
Following the same calculations as in Step 2a, we have $\nc \le \frac{3}{5}R$. We then use Lemma~\ref{lem:speedmab} with $\eta = 1/4$ and $A = x_{i_0}-x_i - \frac{3R}{4\eps} $ to see that
\begin{align}
	& \mathbb{P}_{(\delta_0,0)}\left(\exists k\le \lfloor \eps^{-1}\rfloor :X_k^{4*}> x_{i_0}-x_i\right)                                                             \\
	& \le \mathbb{P}_{(\delta_0,0)}\left(\exists k\le \lfloor \eps^{-1}\rfloor :X_k^{4*}> \frac{(1+\eta)\nc}{\eps} + x_{i_0}-x_i - \frac{3R}{4\eps} \right)\nonumber \\
	& \le h(1/4)\exp\left(-\frac{\sqrt{2\gamma \ir}}8 \left(x_{i_0}-x_i - \frac{3R}{4\eps}\right)\right).
	\label{eq:doucement}
\end{align}
Combining \eqref{eq:salade_de_fruits1} and \eqref{eq:doucement}, we now get
\begin{multline}
	\sum_{i\in\mathbb Z: x_i\le x_{i_0}-R/\eps}\mathbb{P}_{(\delta_i,k_0+1)}\left(\exists k\in \llbracket k_0+1,k_1\rrbracket :X_k^{1*}> a_k\right) \\
	\le \sum_{j\ge 0}h(1/4) \exp\left(-\frac{\sqrt{2\gamma \ir}}8 \left(\frac{R}{4\eps} + (\Delta x)j\right)\right) \le \exp(-\alpha_2/\eps),
	\label{eq:step2b}
\end{multline}
for some $\alpha_2>0$ and for $\eps$ sufficiently small.

Combining \eqref{eq:time_after_time}, \eqref{eq:step2a} and \eqref{eq:step2b},
and using the fact that $\pr^1$ dominates $\pr$, we obtain
\eqref{eq:proof_upperbound_toshow} for some $\alpha>0$ and for $\eps$
sufficiently small, with $A$ as above, depending only on $\nr$, $\ir$
\jj{and $M$} (see \eqref{eq:ub_grad_r}). This concludes the proof of the proposition.
\end{proof}

\subsection{Comparison with the solution of (\ref{cauchypb}): proof of Proposition \ref{thupperbound}}
\label{sec:small_steps_ub}

\jj{\textit{Step 1.} Let us first assume that the function $r$ is as a smooth
growth rate function. }
Let  $T>0$,  $N=\lfloor (T+1)/\Delta t\rfloor$ and consider the sequence $(s_i)_{i=0}^N$ defined as
\begin{equation*}
%s_0=0<s_1=\eps \lfloor \eps^{-1}\rfloor \Delta t<...<s_N=N \eps \lfloor \eps^{-1}\rfloor \Delta t.
s_i=i\eps \lfloor \eps^{-1}\rfloor \Delta t, \quad \forall i\in\llbracket 0,N\rrbracket.
\end{equation*}
We also denote by $A_\eps=\frac{A}{\eps \lfloor \eps^{-1}\rfloor}$,  where $A>0$ is the constant from Proposition \ref{pr:upperbound}, and
consider the sequence $(\tilde{y}_j)_{i=1...N}$ such that
\begin{equation}
\begin{cases}
	\tilde{y}_0     & =0                                                                                                         \\
	\tilde{y}_{j+1} & =\tilde{y}_j+\left(\sqrt{2r(s_j,\tilde{y}_j)}+A_\eps\left(\frac{\Delta x}{\Delta t}+\Delta t\right)\right)
	\iep\eps\Delta t,\quad  \forall j\in\llbracket1,N-1\rrbracket.
\end{cases}
\end{equation}
For $j\in \mathbb{N}$, we define $ \varphi(j)=j\iep$ and consider the following function:
\begin{equation}
f(t)=\tilde{y}_j+(\tilde{y}_{j+1}-\tilde{y}_j)\frac{t/(\eps\Delta t) -\varphi(j)}{\varphi(j+1)-\varphi(j)},\quad \textnormal{if} \; t\in[\eps\varphi(j)\Delta t, \eps\varphi(j+1)\Delta t].
\end{equation}

\noindent First, recall that $X_0^*=0$ and note that the following three events
coincide:
\begin{align*}
\mathcal{B}_0: & =\left\{\exists k\in \llbracket\varphi(0),\varphi(1)\rrbracket: \eps X^*_k>\tilde{y}_0+(\tilde{y}_1-\tilde{y}_0)\frac{k-\varphi(0)}{\varphi(1)-\varphi(0)}\right\}                                                     \\
& = \left\{\exists k\in\llbracket\varphi(0),\varphi(1)\rrbracket : \eps X_k^* >f(k\eps\Delta t)\right\}                                                                                                                  \\
& =\left\{\exists k\in\llbracket\varphi(0),\varphi(1) \rrbracket : \eps X_k^*> \eps X_0^*+\left(\sqrt{2r(0,\eps X_0^*)}+A_\eps\left(\frac{\Delta x}{\Delta t}+\Delta t\right)\right)\eps\Delta t (k-\varphi(0))\right\}.
\end{align*}
Then, for all $j\in \llbracket 0,N-1\rrbracket $, we define $$\mathcal{B}_j=\left\{\exists k\in\llbracket\varphi(j),\varphi(j+1)\rrbracket : \eps X_k^* >f(k\eps\Delta t)\right\}.$$ According to Proposition \ref{pr:upperbound}, there exists $\alpha$ and $\eps_0$, that does not depend on $\tilde{y}_j$ nor on $s_j$, such that, if $\eps<\eps_0$, $K>0$,
\begin{equation*}
\mathbb P(\mathcal{B}_j\cap \mathcal{B}_0^c\cap \cdots \cap \mathcal{B}_{j-1}^c) \le
\mathbb{P}(\mathcal{B}_j \cap \{X_{\varphi(j)}^*\leq \tilde{y}_j\})\leq Ke^{-\frac{\alpha}{\eps}}, \quad \forall j\in \llbracket 0, N-1\rrbracket.
\end{equation*}
Hence, we have
\begin{equation}
\mathbb{P}\left(\cup_{j=0}^{N-1}\mathcal{B}_j\right) = \sum_{j=0}^{N-1}
\mathbb P(\mathcal{B}_j\cap \mathcal{B}_0^c\cap \cdots \cap \mathcal{B}_{j-1}^c) \leq NKe^{-\frac{\alpha}{\eps}}.
\label{unionboundB}
\end{equation}
Then, let us consider the solution $\tilde{x}$ of
\begin{equation*}
\begin{cases}
	\dot{\tilde{x}}(t) & =\sqrt{2r(t,\tilde{x}(t))}+A_\eps\left(\frac{\Delta x}{\Delta t}+\Delta t\right) \\
	\tilde{x}(0)       & =0.
\end{cases}
\end{equation*}
\jj{We know from standard results on the Euler method}
(see  Equation (\ref{euler:s}), Appendix \ref{sec:EDO}) that
\begin{equation*}
\max_{j\in\llbracket 0,N-1\rrbracket}|\tilde{x}(s_j)-\tilde{y}_j|\leq \frac{1}{2}e^{L(T+1)}\eps\lfloor \eps^{-1}\rfloor\Delta t\leq \frac{1}{2}e^{L(T+1)}\Delta t,
\end{equation*}
\jj{where $L$ is as in \eqref{eq:approx_r}.}
Thus, \added[id=j]{using this equation and the mean value theorem}, we get that, for all $j\in \llbracket 0,N-1\rrbracket$ and $t\in[ s_j,s_{j+1}],$ we have
\begin{eqnarray*}
|\tilde{x}(t)-f(t)|&\leq &
|\tilde{x}(t)-\tilde{x}(s_j)|+|\tilde{x}(s_j)-\tilde{y}_j|+|f(t)-\tilde{y}_j|\\
&\leq & 2\left(\sqrt{2\nr}+A_\eps\left(\frac{\Delta x}{\Delta t}+\Delta t\right)\right)\Delta t+ \frac{1}{2}e^{L(T+1)}\Delta t.
\end{eqnarray*}
Let us now compare $\tilde{x}$ with the solution $x$ of \eqref{cauchypb} on $[0,T]$. According to Lemma \ref{lem:stab:EDO}, we have
\begin{equation*}
\max_{t\in [0,T]}|x(t)-\tilde{x}(t)|\leq A_\eps\left(\frac{\Delta x}{\Delta t}+\Delta t\right)
(T+1)e^{LT}.
\end{equation*}
Thus, there exists a constant $B>0$ that only depends \jj{on $\ir$, $\nr$, $M$
(see \eqref{eq:ub_grad_r})} and $T$ such that, for all $\eps<1$
\begin{equation}
\sup_{t\in[0,T]}|x(t)-f(t)|\leq B\left(\Delta t + \frac{\Delta x}{\Delta t}\right).
\label{ub:edof}
\end{equation}
Finally, remarking that for $\Delta t <\frac{1}{2}$ and \added[id=j]{ $\varepsilon< \left(\vep_0\wedge \frac{2}{T+1}\right)$}
\begin{align*}
\varphi(N)=\left\lfloor\frac{T+1}{\Delta t}\right\rfloor\lfloor\varepsilon^{-1}\rfloor & >\left(\frac{T+1}{\Delta t}-1\right)\left(\varepsilon^{-1}-1\right)                                                                \\
& >\frac{T+1-\Delta t -\vep(T+1)}{\vep\Delta t}\geq \frac{T}{\vep \Delta t}\geq \left \lfloor \frac{T}{\vep \Delta t} \right\rfloor,\end{align*}
and combining  Equations (\ref{unionboundB}) and (\ref{ub:edof}), we get that for
all $K\geq 1$ and $\eps<\left(\vep_0\wedge \frac{T+1}{2}\right)$
\begin{eqnarray*}
&&\mathbb{P}\left(\exists k \in \left\llbracket 0,\lfloor T/(\eps\Delta t) \rfloor\right\rrbracket : \eps X_k^*>x(k\eps\Delta t)+B\left(\Delta t+\frac{\Delta x}{\Delta t}\right)\right)\\
&&\leq \mathbb{P}\left(\exists k \in \left\llbracket 0,\varphi(N)\right\rrbracket : \eps X_k^*>x(k\eps\Delta t)+B\left(\Delta t+\frac{\Delta x}{\Delta t}\right)\right)\\
&&\leq \frac{T+1}{\Delta t} Ke^{-\frac{\alpha}{\eps}}.
\end{eqnarray*}
This proves Proposition \ref{thupperbound} when $r$ is a
\jj{smooth growth rate functions},
possibly with a different value of $\alpha$.

\jj{\textit{Step 2.} Suppose that $r$ is a good growth rate function.}

Let $\delta>0$. Pick $n$ large enough so that
\begin{equation}
\label{approx:ub_smooth}
\sup_{t\in[0,T]}|x(t)-\bar{x}_n(t)|<\frac{\delta}{2},
\end{equation}
where $\bar{x}_n$ refers to the solution of the Cauchy problem \eqref{cauchypb}
with smooth growth rate function $\Theta_n$ (see \ref{eq:cp_approx}).

\jj{Note that $\Theta_n$ satisfies the assumptions of Proposition~\ref{pr:upperbound}.
Moreover, it follows from Assumption~\ref{Assumption1} and
Lemma~\ref{lem:coupl} that the interacting particle system
with reproduction laws $(\tilde{\nu}_{\Theta_n,n,K})$
stochastically dominates the system with reproduction laws $({\nu}_{r\textbf{},n,K})$.
Recalling from Remark~\ref{rem:simplification_ub} that
Proposition~\ref{pr:upperbound} was actually proved for the interacting particle
system with reproduction laws $(\tilde{\nu}_{r,n,K})$,
Step 1 shows that there exists $\vep_0\equiv \vep_0(n,\Delta t, \Delta x),
\alpha\equiv\alpha(n,\Delta t, \Delta x)$
and $B\equiv B(n)$ such that
\begin{multline*}
	\forall \vep < \vep_0, \quad  \forall K\geq 1, \;\\
	\mathbb{P}\left(\exists k \in \left\llbracket 0,\lfloor T/(\eps\Delta t)
	\rfloor\right\rrbracket : \eps X_k^*>\underline{x}_n(k\eps\Delta t)
	+B\left(\Delta t+\frac{\Delta x}{\Delta t}\right)\right)
	\leq  Ke^{-\frac{\alpha}{\eps}}.
\end{multline*}
We then fix $\Delta t$ and $\Delta x$ such that
\begin{equation*}
	B\left(\Delta t+\frac{\Delta x}{\Delta t}\right)\leq \frac{\delta}{2}.
\end{equation*}
and combine this bound with \eqref{approx:ub_smooth} to get the result.}

%%%%%%%%%%%%%%%%%%%%%%%%%%%%%%%%%%%%%%%%%%%%%%%%%%%%%%%%%%%%%%%%%%

\section{Proof of Proposition \ref{thlowerbound1}: Lower bound on the propagation speed}\label{partlow}
\sectionmark{Lower bound on the invasion speed}

In this part, we establish a lower bound on the propagation speed of the process
$\textbf{X}$ under Assumption~\ref{Assumption2}.
\jj{We will use the same series of reductions as in the proof of 
Proposition~\ref{thupperbound}: (i) we will derive
the lower bound on the  invasion speed for the interacting particle system with reproduction laws $(\tilde{\nu}_{r,n,K})$
(instead of $({\nu}_{r,n,K})$). The desired lower bound will then  follow from Remark~\ref{rem:coupling_lemma}. (ii)
we will assume that the good growth rate function $r$
satisfies additional regularity assumptions and then conclude using an
approximation argument}

The idea of the proof of Proposition \ref{thlowerbound1}
is to construct a minimising process $\textbf{X}^0$ in which the effect of local competition is negligible (Section \ref{construction:X0}) so that it can be compared to a BRW (Section \ref{sec:ub:brw}), and then to the solution of the ODE \eqref{cauchypb} (Section \ref{sec:ub:ode}). In contrast to Section \ref{partub}, we can no longer compare the process $\textbf{X}$ with several BRW over $\lfloor \eps ^{-1}\rfloor$ generations.
That is why, we will consider smaller time intervals, of order $\log(K)\ll\lfloor\eps^{-1}\rfloor$, during which the population size does not grow too much. Note that the length of the time steps considered in Section \ref{sec:tsptvep} was only constrained by the scale of heterogeneity of the function $r$ and not by the carrying capacity of the environment.

\jj{In Section \ref{partlow}, we denote by $\nc$ the unique positive solution of}
\begin{equation}
I(\nc)=\log(1+\nr \Delta t).
\label{def:c5}
\end{equation}

%%%%%%%%%%%%%%%%%%%%%%%%%%%%%

\subsection{The rebooted process $\textbf{X}^0$} \label{construction:X0}
As explained above, this subsection is aimed at constructing a minimising process $\pr^0$ in which we can ignore the effect of local competition. By minimising process we mean  a process that can be coupled with $\pr$ in such a way that it is dominated by $\pr$ in the sense of Section~\ref{coupsect}. We recall that $X^{0*}_k$ denotes the position of the rightmost particle in the process $\pr^0$ at generation $k$.

The idea of the following construction is to "reboot" the process $\pr$ before its population size gets too large. Let $(\varphi(k))_{k\in\mathbb{N}}$ be a sequence of rebooting times $i.e.$ an increasing sequence of integers. The process  $\pr^0$ starts with a single particle at $X^*_0$ and has the same reproduction and migration laws as $\pr$. At generation $\varphi(1)$, all the particles in $\pr^0$ are killed, except one, located at $X^{0*}_{\varphi(1)}$. The process $\pr^0$ then evolves as $\pr$ until the following rebooting time $\varphi(2)$. Similarly, $\pr^0$ is rebooted at each generation $\varphi(k)$ and is distributed as $\pr$ between generations $\varphi(k)$ and $\varphi(k+1).$

The goal of the following lemma is to show that for $K$ large enough, the population size of $\pr^0$ does not exceed $K$ with high probability for
\begin{equation}
\varphi(k)=k\lfloor \log(K)\rfloor.\label{def:phi2}
\end{equation}

\begin{lemma} \label{lem:reboot}Let $\Delta t<\nr^{-1}$ and $\Delta x>0$. Let $K>0$. Consider a branching random walk of reproduction law $\nu_{\nr}$ and displacement law $\mu$, starting with a single particle at $0$. Let $\tau_0$ be the first generation during which the population size of the process exceeds $K$.
Then,
\begin{equation*}
	\mathbb{P}\left(\tau_0\leq \lfloor \log(K)\rfloor\right)\leq K^{\log(1+\nr \Delta t)-1}.
\end{equation*}
\label{lem:tauK}
\end{lemma}
\begin{proof}Let $N_k$ be the number of individuals alive during generation $k$ in the BRW. Recall that $(N_k)$ is a Galton-Watson process of reproduction law $\nu_{\nr}$. According to Assumption \ref{Assumption2}, the expectation of the reproduction law is equal to $1+\nr \Delta t$. Thanks to basic results on Galton-Watson processes, we know that $((1+\nr \Delta t)^{-k}N_k)_{k\geqslant0}$ is a positive martingale of mean one. Thus, Doob's inequality implies that
\begin{eqnarray*}
	\mathbb{P}(\tau_0\leq \lfloor \log(K)\rfloor)&=& \mathbb{P}\left(\max_{l\leq \lfloor \log(K)\rfloor}N_l\geq K\right) \\
	%&=&\mathbb{P}\left(\max_{l\leq \lfloor \log(K)\rfloor}\frac{N_l}{(1+\nr\Delta t)^l}\geq \frac{K}{(1+\nr\Delta t)^l}\right)\\
	&\leq& \mathbb{P}\left(\max_{l\leq \lfloor \log(K)\rfloor}\frac{N_l}{(1+\nr\Delta t)^l}\geq \frac{K}{(1+\nr\Delta t)^{\log(K)}}\right)\\
	&\leq& \frac{(1+\nr\Delta t)^{\log{K}} }{K}=e^{(\log(1+\nr\Delta t)-1)\log(K)}.
\end{eqnarray*}
\end{proof}Note that there exists a coupling between $\pr^0$ and a BRW of reproduction law $\nu_{\nr}$, displacement law $\mu$, starting with a single particle at $X^{0*}_{\varphi(k)}$  on each time interval $[t_{\varphi(k)},t_{\varphi(k+1)}]$. Thus, if we consider the sequence of rebooting \jj{generations ?} $(\varphi(k))_{k\in\mathbb{N}}$ given by Equation (\ref{def:phi2}), the probability that the population size of $\pr^0$ exceeds $K$ between generations $\varphi(k)$ and $\varphi(k+1)$ is bounded by  $K^{\log(1+\nr \Delta t)-1}$, which tends to $0$ as $K$ tends to infinity as long as $\Delta t<\nr^{-1}$.

\subsection{Comparison with a branching random walk} \label{sec:ub:brw}
In this section, we bound the first (Lemma \ref{ub:condexp} and \ref{lb:condexp}) and the second moment (Lemma \ref{ub:var}) of the increments of the process $(X^{0*}_k)_{k\in\mathbb{N}}$ between generations $\varphi(k)$ and $\varphi(k+1)$, for $\varphi$ defined by (\ref{def:phi2}). Let $(\mathcal{F}_k)_{k\in\mathbb{N}}$ be its natural filtration.

In Lemma \ref{stoppingtimes}, we state a result on some stopping times, that will be needed to construct a coupling between $\pr^0$ and a BRW between generations $\varphi(k)$ and $\varphi(k+1)$. In what follows, we denote by $h$ the function   defined by
\begin{equation}
\tilde{h}(x)=\frac{e^{-\frac{\gamma \sqrt{\nr}}{4\sqrt{2}}x}}{1-e^{-\frac{\gamma \sqrt{\nr}}{4\sqrt{2}}x}}, \quad \forall x>0.
\end{equation}
Note that $\tilde{h}(x)\rightarrow 0$ as $x\rightarrow+\infty$.

\begin{lemma}
\label{stoppingtimes}
Let $\Delta t<\nr^{-1}$ and $\Delta x<\frac{1}{16}\sqrt{2\gamma\ir\Delta t}$. Let $K>0$ and $\eps\leq (4\sqrt{2\nr}\varphi(1)\Delta t)^{-4}$. Let $r\in[\ir,\nr]$. \added[id=j]{Let ${\mathbf{\Xi}}$ be a branching random walk of reproduction law $\nu_r$, displacement law $\mu$, starting with a single particle at $0$}. Denote by $N_k$ the size of this process at generation $k$ and consider
\begin{equation*}
	\tau_K=\inf\left\{k\in\mathbb{N}: N_k \geq K\right\} \textnormal{ and } \tau_{\eps}= \inf\left\{k\in\mathbb{N}: \exists v: |v|=k , |\Xi_v|>\eps^{-1/4}\right\}.
\end{equation*}
Then,
\begin{equation*}
	\mathbb{P}(\tau_K\leq \varphi(1))\leq K^{\log(1+\nr\Delta t)-1},
\end{equation*} and
\begin{equation*}
	\mathbb{P}(\tau_\eps\leq \varphi(1))\leq 2\tilde{h}\left(\frac{\eps^{-1/4}}{\varphi(1)}\right).
\end{equation*}
\end{lemma}
\begin{proof}
The branching random walk \added[id=j]{$\mathbf{\Xi}$} can be coupled with a BRW of reproduction law $\nu_{\nr}$, displacement law $\mu$, starting with a single particle at $0$. Thus, the estimate on $\tau_K$ directly ensues from Lemma \ref{lem:tauK} and it is sufficient to establish the result on $\tau_\eps$ for $r=\nr$.

Thanks to a similar argument to that of Equation (\ref{cramer}), one can prove that, for any particle $v$ such that $|v|=n$\begin{equation*}
	\mathbb{P}(\Xi_v>\eps^{-1/4})\leq e^{-nI\left(\frac{\eps^{-1/4}}{n}\right)}.
\end{equation*}
Thus, by the many-to-one lemma (see Lemma \ref{lem:manyto1}) and by symmetry of $\mu$, we get that, for all $n\leq \varphi(1)$,
\begin{eqnarray*}
	\mathbb{P}(\exists v:|v|=n , |\Xi_v|>\eps^{-1/4})&\leq& 2(1+\nr\Delta t)^ne^{-nI\left(\frac{\eps^{-1/4}}{n}\right)}\\
	&\leq & 2e^{-n\left(I\left(\frac{\eps^{-1/4}}{n}\right)-\log(1+\nr\Delta t)\right)}\\
	&\leq& 2e^{-n\left(I\left(\frac{\eps^{-1/4}}{\varphi(1)}\right)-\log(1+\nr\Delta t)\right)}\\
	&=&2e^{-n\left(I\left(\frac{\eps^{-1/4}}{\varphi(1)}\right)-I(\nc)\right)},
\end{eqnarray*}
where $\nc$ is as in \eqref{def:c5}.
Besides, $I$ is convex, therefore
$I\left(\frac{\eps^{-1/4}}{\varphi(1)}\right)\geq \frac{I(\nc)}{\nc}
\frac{\eps^{-1/4}}{\varphi(1)}$ as long as
$\eps^{-1/4}\geq \nc\varphi(1).$
Thus, if $\eps^{-1/4}\geq 2\nc\varphi(1),$
\begin{equation}
	I\left(\frac{\eps^{-1/4}}{\varphi(1)}\right)-I(c)\geq I(c)\left(\frac{\eps^{-1/4}}{c\varphi(1)}-1\right) \geq \frac{I(c)}{2c\varphi(1)}\eps^{-1/4}\geq \frac{\gamma \sqrt{\nr}}{4\sqrt{2}} \frac{\eps^{-1/4}}{\varphi(1)},
\end{equation}
since $I(\nc)=\log(1+\nr \Delta t)\geq \gamma \nr \Delta t$ and
$\nc\leq 2\sqrt{2\nr}\Delta t$
{(see Lemma \ref{lem:estc1}). Hence, a union bound yields the inequality}
\begin{align*}
	\mathbb{P}(\tau_\eps\leq \varphi(1)) & =\mathbb{P}\left(\exists n\leq \varphi(1):\exists v: |v|=n,|\Xi_v|>\eps^{-1/4}\right)                            \\
	& \leq 2\sum_{n=1}^{\varphi(1)}e^{-n\left(I\left(\frac{\eps^{-1/4}}{\varphi(1)}\right)-\log(1+\nr\Delta t)\right)} \\
	& \leq 2 \tilde{h}\left(\frac{\eps^{-1/4}}{\varphi(1)}\right).\end{align*}
	\end{proof}
	
	\begin{lemma}[Lower bound on the first moment]\label{lb:condexp}
\jj{Assume that $r$ is a smooth growth rate function.}
Let $\Delta t<\nr^{-1}$, $\Delta x<\frac{1}{16}\sqrt{2\gamma\ir\Delta t}$ and
$\eta>0$. There exists $K_0>0$ such that, for all $K>K_0$, there exists
$\eps_0$ such that, for all $\eps<\eps_0$,

\begin{equation*}
	\mathbb{E}\left[X^{0*}_{\varphi(k+1)}-X^{0*}_{\varphi(k)}|\mathcal{F}_{\varphi(k)}\right]\geq (c_{\varphi(k)}-\eta)\varphi(1),
\end{equation*}
with $c_{\varphi(k)}$ the unique positive solution of $I(c_{\varphi(k)})=\log(1+r_{\varphi(k)}\Delta t)$ for
\begin{equation}
	r_{\varphi(k)}=\min\left\{r(\eps t,\eps x), \; (t,x)\in[\varphi(k)\Delta t,\varphi(k+1)\Delta t]\times[X^{0*}_{\varphi(k)}-\eps^{-1/4},X^{0*}_{\varphi(k)}+\eps^{-1/4}]\right\}. \label{def:rphik}
\end{equation}
\end{lemma}
\begin{proof}
For the sake of simplicity, we assume that $\pr^0$ starts at generation $\varphi(k)$ with a deterministic configuration $n^0_{\varphi(k)}=\delta_{X^{0*}_{\varphi(k)}}$. The estimates we obtain will not depend on this initial condition. The proof of the lemma relies on a coupling argument.
% with a branching random walk of reproduction law $\nu_{r_{\varphi(k)}}$.

Consider a branching random walk $\added[id=j]{\mathbf{\Xi}}$  of reproduction law $\nu_{r_{\varphi(k)}}$, displacement law $\mu$, starting with a single particle at $X^{0*}_{\varphi(k)}$ at time $\varphi(k)\Delta t$. Denote by $N_k$  the size of the BRW at generation $k$ and define
\begin{equation*}
	\tau_K'=\inf\left\{l\geq\varphi(k): N_l \geq K\right\} \textnormal{ and } \tau_{\eps}'= \inf\left\{l\geq\varphi(k): \exists v: |v|=l , |\Xi_v-X^{0*}_{\varphi(k)}|>\eps^{-1/4}\right\}.
\end{equation*}
In addition, for $n\geq \varphi(k)$, define
\begin{equation}
	M_n=\max\{\Xi_v, |v|=n\}.
\end{equation}

%Consider a branching random walk $\Xi$  of reproduction law $\nu_{r_{\varphi(k)}}$, displacement law $\mu$, starting with a single particle at $X^{0*}_{\varphi(k)}$ at time $\varphi(k)\Delta t$, and denote by $M_n$  the position of its rightmost particle at generation $n$. We further define 
%\begin{equation*}
%\tau_K'=\inf\left\{l\geq\varphi(k): N_l \geq K\right\} \textnormal{ and } \tau_{\eps}'= \inf\left\{l\geq\varphi(k): \exists v: |v|=l , |\Xi_v-X^{0*}_{\varphi(k)}|>\eps^{-1/4}\right\}. 
%\end{equation*}
According to Lemma \ref{lem:coupl}, one can couple $\pr^0$ and $\added[id=j]{\mathbf{\Xi}}$ such that $\pr^0$ dominates $\added[id=j]{\mathbf{\Xi}}$ until generation $(\tau_K'\wedge\tau_\eps').$  Besides, note that $\tau_K'-\varphi(k)$ (resp.~$\tau_\vep'-\varphi(k)$) follows the same law as $\tau_K$ (resp.~$\tau_\vep$) from Lemma \ref{stoppingtimes}, for $r=r_{\varphi(k)}$. Thus,
\begin{align}
	\nonumber\mathbb{E} & \left[X^{0*}_{\varphi(k+1)}-X^{0*}_{\varphi(k)}\right]
	\\
	\nonumber           & = \mathbb{E}\left[(X^{0*}_{\varphi(k+1)}-X^{0*}_{\varphi(k)})\mathbb{1}_{\tau_K'\wedge \tau_\eps'> \varphi(k+1)}\right] +\mathbb{E}\left[(X^{0*}_{\varphi(k+1)}-X^{0*}_{\varphi(k)})\mathbb{1}_{\tau_K'\wedge \tau_\eps'\leq \varphi(k+1)}\right]
	\\
	\label{eq:tgv}
	& \geq \mathbb{E}\left[(M_{\varphi(k+1)}-X^{0*}_{\varphi(k+1)})\mathbb{1}_{\tau_K'\wedge \tau_\eps'> \varphi(k+1)} \right]+\mathbb{E}\left[(X^{0*}_{\varphi(k+1)}-X^{0*}_{\varphi(k)})\mathbb{1}_{\tau_K'\wedge \tau_\eps'\leq \varphi(k+1)} \right].
\end{align}
We first bound the first term on the RHS of \eqref{eq:tgv}. We have
\begin{multline}
	\mathbb{E}\left[(M_{\varphi(k+1)}-X^{0*}_{\varphi(k)})\mathbb{1}_{\tau_K'\wedge \tau_\eps'> \varphi(k+1)} \right]\\=\mathbb{E}\left[M_{\varphi(k+1)}-X^{0*}_{\varphi(k)}\, \right]-\mathbb{E}\left[(M_{\varphi(k+1)}-X^{0*}_{\varphi(k)})\mathbb{1}_{\tau_K'\wedge \tau_\eps' \added[id=j]{\leq} \varphi(k+1)} \right], \label{lb:1}
\end{multline}
and by the Cauchy-Schwarz inequality,
\begin{equation}
	\mathbb{E}\left[(M_{\varphi(k+1)}-X^{0*}_{\varphi(k)})\mathbb{1}_{\tau_K'\wedge \tau_\eps'\leq \varphi(k+1)}\right] \leq \sqrt{\mathbb{E}\left[(M_{\varphi(k+1)}-X^{0*}_{\varphi(k)})^2\right]}\sqrt{\mathbb{P}(\tau_K\wedge\tau_\eps\leq \varphi(1))}.\label{lb:2}
\end{equation}
%\begin{\equation}
%\mathbb{E}\left[(M_{\varphi(k+1)}-X^{0*}_{\varphi(k)})\mathbb{1}_{\tau_K\wedge \tau_\eps\leq \varphi(1)}|\mathcal{F}_{\varphi(k)}\right]\leq \sqrt{\mathbb{E}\left[(M_{\varphi(k+1)}-X^{0*}_{\varphi(k)})^2|\mathcal{F}_{\varphi(k)}\right]}\sqrt{\mathbb{P}(\tau_K\wedge\tau_\eps\leq \varphi(1))}.\label{lb:2}
%\end{\equation}
According to Lemma \ref{lemunifcv} and Lemma~\ref{lemmom2}, there exists $K_0>0$, that does not depend on $r_{\varphi(k)}$ nor on $X_{\varphi(k)}^{0*}$ such that, if $K>K_0$,
\begin{equation}\label{lb:3}
	\mathbb{E}\left[M_{\varphi(k+1)}-X^{0*}_{\varphi(k)}\right]\geq (c_{\varphi(k)}-\eta)\varphi(1),
\end{equation}
and
\begin{equation}
	\mathbb{E}\left[(M_{\varphi(k+1)}-X^{0*}_{\varphi(k)})^2\right]\leq 4\jj{\nc}^2\varphi(1)^2 \label{lb:4}.
\end{equation}
Hence, combining (\ref{lb:1}), (\ref{lb:2}), (\ref{lb:3}) and (\ref{lb:4}), we get that
\begin{equation}
	\label{eq:one}
	\mathbb{E}\left[(M_{\varphi(k+1)}-X^{0*}_{\varphi(k)})\mathbb{1}_{\tau_K'\wedge \tau_\eps'> \varphi(k+1)}\right]\geq (c_{\varphi(k)}-\eta)\varphi(1)-2\jj{\nc}\varphi(1)\sqrt{\mathbb{P}(\tau_K\wedge\tau_\eps\leq \varphi(1))}.
\end{equation}
As for the second term on the RHS of \eqref{eq:tgv}, we have again by the Cauchy-Schwarz inequality,
\begin{eqnarray}
	&&\mathbb{E}\left[(X^{0*}_{\varphi(k+1)}-X^{0*}_{\varphi(k)})\mathbb{1}_{\tau_K'\wedge \tau_\eps'\leq \varphi(k+1)}\right]\nonumber\\
	&& \ge\mathbb{E}\left[(X^{0*}_{\varphi(k+1)}-X^{0*}_{\varphi(k)}\wedge 0) \mathbb{1}_{\tau_K'\wedge \tau_\eps'\leq \varphi(k+1)}\right]\nonumber
	\\
	&&\ge - \sqrt{\mathbb{E}\left[(X^{0*}_{\varphi(k+1)}-X^{0*}_{\varphi(k)}\wedge 0)^2\right]}\sqrt{\mathbb{P}(\tau_K\wedge\tau_\eps\leq \varphi(1))}.
	\label{eq:two}
\end{eqnarray}
In order to bound the first term on the RHS of \eqref{eq:two}, we note that by Assumption~\ref{Assumption2}, there exists a particle $v$ at generation $\varphi(k+1)$, such that $X^0_v - X^{0*}_{\varphi(k)}$ is equal in law to $S_{\varphi(1)}$, where $(S_n)_{n\ge0}$ is a random walk with displacement distribution $\mu$ (heuristically, $v$ is obtained by choosing at each time step one of the children at random and iterating). It follows that
\begin{align}
	\label{eq:three}
	\mathbb{E}\left[(X^{0*}_{\varphi(k+1)}-X^{0*}_{\varphi(k)}\wedge 0)^2\right] \le \mathbb{E}\left[(X^{0*}_v-X^{0*}_{\varphi(k)}\wedge 0)^2\right] \le \mathbb{E}[S_{\varphi(1)}^2] = \Var(\mu)\varphi(1).
\end{align}

Finally, remark that $c_{\varphi(k)}\leq \jj{\nc}$ since $I$ is increasing on $(0,\infty)$ \added[id=j]{(see Equation \eqref{def:c5})} and that, according to Lemma \ref{stoppingtimes}, there exists $K_1>0$ such that, for all $K>K_1$, there exists $\eps_1>0$ such that for all $\eps<\eps_1$,
\begin{equation}
	\label{eq:four}
	\mathbb{P}\left(\tau_k\wedge \tau_\eps\leq \varphi(1)\right)\leq \frac{\eta^2}{(2\jj{\nc}+\Var(\mu))^2}.
\end{equation}
Combining \eqref{eq:tgv}, \eqref{eq:one}, \eqref{eq:two}, \eqref{eq:three} and \eqref{eq:four},  we get, for $K>\max(K_0,K_1)$ and $\eps<\eps_1$,
\begin{equation*}
	\mathbb{E}\left[X^{0*}_{\varphi(k+1)}-X^{0*}_{\varphi(k)}\right]\geq (c_{\varphi(k)}-2\eta )\varphi(1).
\end{equation*}
This proves the lemma.
\end{proof}

\begin{rem}
\label{rem:K_Deltax}
\jj{
	Lemma~\ref{lb:condexp} is where the assumption that $K$ is sufficiently large is crucial. However, one could replace it by the sole assumption that $\Delta x$ is sufficiently small. Indeed, in this case, the comparison with a branching random walk still holds until a certain time of order $\log(1/\Delta x)$, since one can show that particles do not meet until that time. For simplicity, we leave out the details.}
	\end{rem}
	
	\begin{lemma}[Upper bound on the second moment]\label{ub:var}
\jj{Assume that $r$ is a smooth growth rate function.}
Let $\Delta t<\nr^{-1}$, $\Delta x<\sqrt{2\gamma\ir\Delta t}$. There exists
$K_0>0$ such that, for all $K>K_0$, for all $\eps>0$,
\begin{equation*}
	\mathbb{E}\left[\left(X^{0*}_{\varphi(k+1)}-X^{0*}_{\varphi(k)}\right)^2|
	\mathcal{F}_{\varphi(k)}\right]\leq 2\nc^2\varphi(1)^2.
\end{equation*}
\end{lemma}
\begin{proof}
%First, note that
%\begin{equation*}
%\Var\left(X^{0*}_{\varphi(k+1)}-X^{0*}_{\varphi(k)}|\mathcal{F}_{\varphi(k)}\right)\leq \mathbb{E}\left[\left(X^{0*}_{\varphi(k+1)}-X^{0*}_{\varphi(k)}\right)^2|\mathcal{F}_{\varphi(k)}\right].
%\end{equation*}
By Lemma~\ref{lem:coupl}, one can couple $\pr^0$ and a BRW $\added[id=j]{\mathbf{\Xi}}$ of reproduction law $\nu_{\nr}$, displacement law $\mu$, \added[id=j]{starting from a single particle at $X_{\varphi(k)}^{0*}$ at time $\varphi(k)\Delta t$}, such that $\added[id=j]{\mathbf{\Xi}}$ dominates \added[id=j]{$\pr^0$} until generation $\varphi(k+1)$. Then, if we denote by $M_n$ the position of the rightmost particle in $\added[id=j]{\mathbf{\Xi}}$ at generation $n$, we get that
\begin{eqnarray*}
	\mathbb{E}\left[\left(\left(X^{0*}_{\varphi(k+1)}-X^{0*}_{\varphi(k)}\right)\vee 0\right)^2|\mathcal{F}_{\varphi(k)}\right]&\leq&  \mathbb{E}\left[\left(M_{\varphi(1)}\vee 0\right)^2|\mathcal{F}_{\varphi(k)}\right]\\
	&\leq& \mathbb{E}\left[M_{\varphi(1)}^2|\mathcal{F}_{\varphi(k)}\right]\leq \frac 3 2 \jj{\nc}^2\varphi(1)^2,
\end{eqnarray*}
for $K$ large enough, according to Lemma~\ref{lemmom2}. As in the proof of Lemma~\ref{lb:condexp}, consider $(S_n)_{n\ge0}$ a random walk whose increments are distributed as $\mu$. Then,
\begin{align*}
	\mathbb{E}\left[\left(\left(X^{0*}_{\varphi(k+1)}-X^{0*}_{\varphi(k)}\right)\wedge 0\right)^2|\mathcal{F}_{\varphi(k)}\right]
	& \leq  \mathbb{E}\left[\left(S_{\varphi(1)}\wedge 0\right)^2\right]
	\leq \mathbb{E}\left[S_{\varphi(1)}^2\right]                          \\
	& = \Var(\mu)\varphi(1) \le \frac{1}{2}\jj{\nc}^2\varphi(1)^2,
\end{align*}
for $K$ large enough. Combining the two previous inequalities yields the lemma.
\end{proof}

\begin{lemma}[Upper bound on the first moment]\label{ub:condexp}
\jj{Assume that $r$ is a smooth growth rate function.}
Let $\Delta t<\nr^{-1}$, $\Delta x<\sqrt{2\gamma\ir\Delta t}$ and $\eta>0$.
There exists $K_0>0$ such that, for all $K>K_0$, there exists $\eps_0$ such
that, for all $\eps<\eps_0$,

\begin{equation*}
	\mathbb{E}\left[X^{0*}_{\varphi(k+1)}-X^{0*}_{\varphi(k)}|\mathcal{F}_{\varphi(k)}\right]\leq (\bc_{\varphi(k)}+\eta)\varphi(1)
\end{equation*}
with $\bc_{\varphi(k)}$ the unique positive solution of $I(\bc_{\varphi(k)})=\log(1+\br_{\varphi(k)}\Delta t)$ for
\begin{equation*}
	\br_{\varphi(k)}=\max\left\{r(\eps t,\eps x), \; (t,x)\in[\varphi(k)\Delta t,\varphi(k+1)\Delta t]\times[X^{0*}_{\varphi(k)}-\eps^{-1/4},X^{0*}_{\varphi(k)}+\eps^{-1/4}]\right\}.
\end{equation*}

\end{lemma}
\begin{proof} The proof is similar of the proof of Lemma \ref{lb:condexp} but some details are different, which is why we give a complete proof. Again, we assume that $\pr^0$ starts at generation $\varphi(k)$ with a deterministic configuration $n^0_{\varphi(k)}=\delta_{X^{0*}_{\varphi(k)}}$ and the proof of the lemma relies on a coupling argument.
% with a branching random walk of reproduction law $\nu_{r_{\varphi(k)}}$.

Consider a branching random walk $\tilde{\added[id=j]{\mathbf{\Xi}}}$  of reproduction law $\nu_{\br_{\varphi(k)}}$, displacement law $\mu$, starting with a single particle at $X^{0*}_{\varphi(k)}$ at time $\varphi(k)\Delta t$. \added[id=j]{Denote by $\tilde{N}_k$ the size of the population in $\tilde{\added[id=j]{\mathbf{\Xi}}}$} at generation $k$ and define
\begin{equation*}
	\tau_{\eps}''= \inf\left\{l\geq\varphi(k): \exists v: |v|=l , |\tilde{\Xi}_v-X^{0*}_{\varphi(k)}|>\eps^{-1/4}\right\}.
\end{equation*}In addition, for $n\geq \varphi(k)$, define
\begin{equation}
	\tilde{M}_n=\max\{\tilde{\Xi}_v, |v|=n\}.
\end{equation}

%Consider a BRW $\tilde{\Xi}$ of reproduction law $\nu_{\br_{\varphi(k)}}$, displacement law $\mu$, starting with a single particle at $X^{0*}_{\varphi(k)}$ at time $\varphi(k)\Delta t$, and denote by $M_n$ the position of its rightmost particle at time $n$. We also define 
According to Lemma \ref{lem:coupl}, one can couple $\pr^0$ and $\tilde{\added[id=j]{\mathbf{\Xi}}}$ such that $\tilde{\added[id=j]{\mathbf{\Xi}}}$ dominates \added[id=j]{$\pr^0$} until generation $\tau_\eps''.$ Besides, note that $\tau_{\eps}''-\varphi(k)$ and $\tau_\eps$ (with $r=\br_{\varphi(k)}$) from Lemma \ref{stoppingtimes} follow the same law. Thus, we have
%\begin{eqnarray}
%&&\mathbb{E}\left[X^{0*}_{\varphi(k+1)}-X^{0*}_{\varphi(k)}|\mathcal{F}_{\varphi(k)}\right]\nonumber\\
% &&\quad= \mathbb{E}\left[(X^{0*}_{\varphi(k+1)}-X^{0*}_{\varphi(k)})\mathbb{1}_{\tau_\eps\leq \varphi(1)}|\mathcal{F}_{\varphi(k)}\right]+\mathbb{E}\left[(X^{0*}_{\varphi(k+1)}-X^{0*}_{\varphi(k)})\mathbb{1}_{\tau_\eps> \varphi(1)}|\mathcal{F}_{\varphi(k)}\right]\nonumber\\
%&&\quad\leq \mathbb{E}\left[(X^{0*}_{\varphi(k+1)}-X^{0*}_{\varphi(k)})\mathbb{1}_{\tau_\eps\leq \varphi(1)}|\mathcal{F}_{\varphi(k)}\right]+\mathbb{E}\left[(M_{\varphi(k+1)}-X^{0*}_{\varphi(k)})\mathbb{1}_{\tau_\eps> \varphi(1)}|\mathcal{F}_{\varphi(k)}\right]\nonumber\\
%&&\quad\leq \sqrt{\mathbb{E}\left[(X^{0*}_{\varphi(k+1)}-X^{0*}_{\varphi(k)})^2|\mathcal{F}_{\varphi(k)}\right]} \sqrt{\mathbb{P}(\tau_\eps\leq \varphi(1))}\nonumber\\
%&&\quad + \sqrt{\mathbb{E}\left[(M_{\varphi(k+1)}-X^{0*}_{\varphi(k)})^2|\mathcal{F}_{\varphi(k)}\right]} \sqrt{\mathbb{P}(\tau_\eps> \varphi(1))},
%\label{lem8:1}
%\end{eqnarray}
\begin{eqnarray}
	&&\mathbb{E}\left[X^{0*}_{\varphi(k+1)}-X^{0*}_{\varphi(k)}\right]\label{lem8:1} \\
	&&=\mathbb{E}\left[(X^{0*}_{\varphi(k+1)}-X^{0*}_{\varphi(k)})\mathbb{1}_{\tau_\eps''\leq \varphi(k+1)}\right]+\mathbb{E}\left[(X^{0*}_{\varphi(k+1)}-X^{0*}_{\varphi(k)})\mathbb{1}_{\tau_\eps''> \varphi(k+1)}\right]\nonumber\\
	&&\leq \mathbb{E}\left[(X^{0*}_{\varphi(k+1)}-X^{0*}_{\varphi(k)})\mathbb{1}_{\tau_\eps''\leq \varphi(k+1)}\right]+\mathbb{E}\left[(\added[id=j]{\tilde{M}}_{\varphi(k+1)}-X^{0*}_{\varphi(k)})\mathbb{1}_{\tau_\eps''> \varphi(k+1)}\right]\nonumber\\
	&&\leq\sqrt{\mathbb{E}\left[(X^{0*}_{\varphi(k+1)}-X^{0*}_{\varphi(k)})^2\right]} \sqrt{\mathbb{P}(\tau_\eps\leq \varphi(1))} + \sqrt{\mathbb{E}\left[(\tilde{M}_{\varphi(k+1)}-X^{0*}_{\varphi(k)})^2\right]} \sqrt{\mathbb{P}(\tau_\eps> \varphi(1))},\nonumber
\end{eqnarray}
thanks to the Cauchy-Schwarz inequality. According to Lemma \ref{lemmom2}, there exists $K_0>0$, that does not depend on $r_{\varphi(k)}$ nor on $X^{0*}_\vpk$, such that, for all $K>K_0$,
\begin{equation}
	\mathbb{E}\left[(\tilde{M}_{\varphi(k+1)}-X^{0*}_{\varphi(k)})^2\right]\leq (\bc_{\varphi(k)}+\eta)^2\varphi(1)^2. \label{lem8:2}
\end{equation}
Besides, according to Lemma \ref{ub:var}, there exists $K_1>0$ such that for all $K>K_1$,
% $\mathbb{E}\left[(X^{0*}_{\varphi(k+1)}-X^{0*}_{\varphi(k)})\mathbb{1}_{\tau_\eps\leq \varphi(1)}|\mathcal{F}_{\varphi(k)}\right]$ can be bounded by coupling $\pr^0$ with a BRW of reproduction law $\nu_{\nr}$, displacement law $\mu$, starting with a single particle at $X^{0*}_{\varphi(k)}$, between generations $\varphi(k)$ and $\varphi(k+1)$. Corollary \ref{cor:var}, then implies that there exist $K_1>0$, that does not depend on $\br_{\varphi(k)}$, nor on $X^{0*}_\vpk$, such that, if $K>K_1$,
\begin{equation}
	\mathbb{E}\left[\left(X^{0*}_{\varphi(k+1)}-X^{0*}_{\varphi(k)}\right)^2\right]
	\leq 2\nc^2\varphi(1)^2.
	\label{lem8:3}\end{equation}
Let us now assume that $K>\max(K_0,K_1)$. According to Lemma \ref{stoppingtimes},
there exists $\eps_1>0$ such that for all $\eps<\eps_1$,
\begin{equation}
	\mathbb{P}\left(\tau_\eps\leq \varphi(1)\right)\leq
	\frac{\eta^2}{2\nc^2}.
	\label{lem8:4}
\end{equation} Finally, combining Equations (\ref{lem8:1}), (\ref{lem8:2}), (\ref{lem8:3}) and (\ref{lem8:4}), we get that, for $\eps<\eps_1$,
\begin{equation*}
	\mathbb{E}\left[X^{0*}_{\varphi(k+1)}-X^{0*}_{\varphi(k)}\right]\leq (\bc_{\varphi(k)}+2\eta)\varphi(1),
\end{equation*}
which concludes the proof of the lemma.
\end{proof}

\subsection{Comparison with the solution of (\ref{cauchypb})} \label{sec:ub:ode}
\begin{lemma}
\jj{Assume that $r$ is a smooth growth rate function.}
Let $\delta\in(0,1)$.

%There exists $C_\delta >0$ such that, 
\jj{If $a$ is as in Lemma \ref{est:c21}}
and
\begin{equation}
	\gr{\Delta t<\delta\nr^{-1}}, \quad%\left((1\wedge C_\delta)\nr^{-1}\right), \quad
	\Delta x<\left(\sqrt{2\gamma \ir}\wedge \frac{\sqrt{2\ir}}{3a} \delta \right)\Delta t,
	\tag{H$_\delta$}
	\label{Assumptiondtdx}
\end{equation}
then, for all $K>0$, there exists $\eps_0$ such that for all $\eps<\eps_0$,
\begin{equation*}
	c_{\varphi(k)}\geq (1-\delta)\sqrt{2r(\eps t_{\varphi(k)},\eps X^{0*}_\vpk)}\Delta t, \quad \forall k\in \mathbb{N},
\end{equation*}
and \begin{equation*}
	\bc_{\varphi(k)}\leq (1+\delta)\sqrt{2r(\eps t_{\varphi(k)},\eps X^{0*}_\vpk)}\Delta t, \quad \forall k\in \mathbb{N},
\end{equation*}
for $c_{\varphi(k)}$ and $\bc_{\varphi(k)}$ respectively defined in
Lemmas \ref{lb:condexp} and \ref{ub:condexp}.
\end{lemma}
\begin{proof} Let $\Delta t<\nr^{-1}$ and $\Delta x<\sqrt{2\gamma \ir}\Delta t$.
Thanks to Lemma \ref{est:c21}, we know that
\begin{equation}
	|c_{\varphi(k)}-\sqrt{2\Delta t \log(1+r_{\varphi(k)}\Delta t)}|\leq a\Delta x. \label{estc:531}
\end{equation} Let $\delta \in(0,1)$. 
%Let $C_\delta$ be a positive constant such that $\log(1+x)\geq (1-\delta/3)^2 x$, for all $x\leq C_\delta$. 
\gr{Note that $\log(1+x)\geq (1-\delta/3)^2 x$, for $x\leq \delta$. Indeed, $x\mapsto \log(1+x)-(1-\delta/3)^2 x$ is concave, equal to $0$ when $x=0$, and 
	\begin{align*}
		&\log(1+\delta)-(1-\delta/3)^2 \delta=\log(1+\delta)-\left(\delta -2\delta^2/3+\delta^3/9\right)\\
		&\quad =\int_0^\delta \frac 1{1+y}-\left(1-\frac{4}{3}y+\frac{1}{3}y^2\right)\,dy
		%&=\int_0^\delta \frac 1{1+y}-\frac{1+y-\frac{4}{3}y-\frac{4}{3}y^2+\frac{1}{3}y^2+\frac{1}{3}y^3}{1+y}\,dy\\
		=\int_0^\delta \frac {y+3y^2-y^3}{3(1+y)}\,dy=\int_0^\delta y\frac {1+3y-y^2}{3(1+y)}\,dy\geq 0,
	\end{align*}
	since $1+3y-y^2\geq 0$ for $y\in[0,1]$.
}
Let us now assume that \gr{$\Delta t<\delta \nr^{-1}$}. 
Equation (\ref{estc:531}) then gives
\begin{equation}
	c_{\varphi(k)}\geq \sqrt{2\Delta t \log(1+r_{\varphi(k)}\Delta t)}-a\Delta x\geq (1-\delta/3)\sqrt{2r_{\varphi(k)}}\Delta t-a\Delta x. \label{estc:532}
\end{equation}
Besides, by definition of $r_{\varphi(k)}$ (see Equation \eqref{def:rphik}) and Equation \eqref{eq:approx_r} we have
\begin{equation}
	|\sqrt{2r_{\vpk}}-\sqrt{2r(\eps t_{\varphi(k)},\eps X^{0*}_\vpk)}|\leq L(\eps \varphi(1)\Delta t+ \eps^{3/4}),\label{estc:533}
\end{equation}
\jj{($L$ is as in Equation \eqref{eq:approx_r})}.
Combining (\ref{estc:532}) and (\ref{estc:533}) we get that
\begin{equation*}
	c_\vpk\geq (1-\delta/3)\sqrt{2r(\eps t_{\varphi(k)},\eps X^{0*}_\vpk)}\Delta t-(1-\delta/3)L(\eps \varphi(1)\Delta t+ \eps^{3/4})\Delta t -a\Delta x
\end{equation*}
Remark that Assumption (\ref{Assumptiondtdx}) implies that
$$a\Delta x\leq \frac{1}{3}\sqrt{2\ir}\delta \Delta t\leq
\frac{\delta}{3}\sqrt{2r(\eps t_{\varphi(k)},\eps X^{0*}_\vpk)}\Delta t,$$
so that it is sufficient to choose $\eps$ such that
$$\jj{L}(\eps\varphi(1)\Delta t+\eps^{3/4})\leq \frac{2\underline{r}}{3}\delta,$$ to get
\begin{equation*}
	c_\vpk\geq (1-\delta)\sqrt{2r(\eps t_{\varphi(k)},\eps X^{0*}_\vpk)}\Delta t.
\end{equation*}
Similarly, Equations (\ref{estc:531}) and  (\ref{estc:533}) give that
\begin{equation*}
	\bc_{\vpk}\leq \sqrt{2r(\eps t_{\varphi(k)},\eps X^{0*}_\vpk)}\Delta t+L(\eps \varphi(1)\Delta t+ \eps^{3/4})\Delta t +a\Delta x,
\end{equation*}
and  Assumption (\ref{Assumptiondtdx}) implies that $a\Delta x\leq \frac{1}{2}\sqrt{2\ir}\delta \Delta t,$ so that it is sufficient to choose $\eps$ such that
$\jj{L}(\eps\varphi(1)\Delta t+\eps^{3/4})\Delta t\leq \underline{r}\delta$ to get the result.
\end{proof}

\begin{cor}\label{cor:ode}
\jj{Assume that $r$ is a smooth growth rate function.}
Let $\delta>0$. 
\jj{Suppose}
$\Delta t$ and $\Delta x$ satisfy (\ref{Assumptiondtdx}).
There exists $K_0>0$ such that, for all $K>K_0$,  there exists $\eps_0$ such that for all $\eps<\eps_0$,
\begin{equation*}
	\left|\mathbb{E}\left[ X^{0*}_{\varphi(k+1)}-X^{0*}_{\varphi(k)}|\mathcal{F}_{\varphi(k)}\right]-\sqrt{2r(\eps t_{\varphi(k)},\eps X^{0*}_\vpk)}\varphi(1)\Delta t\right|\leq 2\sqrt{2\nr} \delta \varphi(1)\Delta t .
\end{equation*}
\end{cor}
\begin{proof}
We choose $\eta=\delta\sqrt{2\nr}\Delta t$ in Lemmas \ref{lb:condexp} and \ref{ub:condexp}.
\end{proof}

\begin{proof}[Proof of Proposition \ref{thlowerbound1}]
It suffices to prove the statement for \emph{some} $K$, since by Assumption~\ref{Assumption2} and Lemma~\ref{lem:coupl} the statement then holds for all larger $K$.

Let $\delta \in(0,1)$ and assume that $\Delta t$ and $\Delta x$ satisfy (\ref{Assumptiondtdx}). Let $T>0$ and $N=\left \lfloor \frac{T}{ \eps\varphi(1)\Delta t}\right \rfloor$. Recall that $X^*_0=0$.

\jj{\textit{Step 1.} Let us first assume that $r$ is a smooth growth rate
	function.
	The Euler scheme associated to \eqref{cauchypb} on
	$[0,T]$, with time step $\eps\varphi(1)\Delta t$, is defined
	as the sequence $(y_k)_{k=0}^N$ such that $y_0=0$ and
	\begin{equation*}
		y_{k+1}=y_k+\sqrt{2r(\eps t_{\varphi(k)},y_k)}\eps\varphi(1)\Delta t,
		\quad \forall i \in \llbracket 0, N-1\rrbracket.
\end{equation*}}
We also define the process $(Y_k)_{k=0}^N$ by $Y_k=\eps X^{0*}_{\varphi(k)},$
and consider its Doob decomposition $Y_k = Z_k + W_k$, where
\begin{equation*}
	W_k=\sum_{j=0}^{k-1}\mathbb{E}\left[Y_{j+1}-Y_j|
	\mathcal{F}_{\varphi(j)}\right], \quad Z_k=Y_k-W_k,
	\quad \forall k\in\llbracket0,N\rrbracket.
\end{equation*}
Since $(Z_k)$ is a martingale, we have
\begin{equation}
	\Var(Z_k)=\sum_{j=0}^{k-1}\mathbb{E}[\Var(Y_{j+1}-Y_j|
	\mathcal{F}_{\varphi(j)})], \quad \forall k\in\llbracket0,N\rrbracket\label{ode:1}.
\end{equation}
Besides, Lemma \ref{ub:var} implies that there exists $K_0>1$ such that, for all $K\geq K_0$,
\begin{equation}
	\Var(Y_{j+1}-Y_j|\mathcal{F}_{\varphi(j)})\leq 2\jj{\nc}^2\eps^2\varphi(1)^2, \quad \forall j\in \mathbb{N}.
\end{equation}
Thus,
\begin{equation*}
	\Var(Z_{N})\leq 2\jj{\nc}^2\eps^2\varphi(1)^2N\leq 2\jj{\nc}^2 \varphi(1)\eps T.
\end{equation*}
Moreover, let \jj{$\mathcal{A}_\delta$} be the event
$$\left\{\max_{k=0,...,N}|Z_k|\leq \delta\right\},$$
By Doob's inequality, we have that
\begin{equation}
	\mathbb{P}(\mathcal{A}_\delta ^c)\leq \frac{\Var(Z_{N})}{\delta^2}\leq \frac{2\jj{\nc}^2T\varphi(1)}{\delta ^2} \eps. \label{ode:doob}
\end{equation}
Besides, we know that on the event $\mathcal{A}_\delta$, $|Y_k-W_k|=|Z_k|\leq \delta $ for all $k\in\llbracket 0,N\rrbracket$ and therefore, by the triangle inequality,
\begin{eqnarray}
	\text{on $\mathcal{A}_\delta$,}\quad |Y_{k+1}-y_{k+1}|&\leq&|Y_{k+1}-W_{k+1}|+|W_{k+1}-y_{k+1}|\nonumber\\
	&\leq& \delta + |W_{k+1}-y_{k+1}|,\label{ode:2}
	%&\leq& \delta  +|y_k-W_k-Y_0|+|\mathbb{E}[Y_{k+1}-Y_k|Y_k]-\sqrt{2r(\eps t_\vpk,y_k)}\eps\varphi(1)\Delta t|,
\end{eqnarray}
for all $k\in\llbracket 0,N-1\rrbracket.$
Moreover, using the definition of $W_k$ and $y_k$, we get that
% by the triangle inequality,
\begin{eqnarray}
	|W_{k+1}-y_{k+1}|&\leq& |W_{k}-y_{k}|+\left|\mathbb{E}[Y_{k+1}-Y_k|\mathcal{F}_{\varphi(k)}]-\sqrt{2r(\eps t_\vpk,y_k)}\eps\varphi(1)\Delta t\right|\nonumber\\
	&\leq& |W_{k}-y_{k}|+\left|\mathbb{E}[Y_{k+1}-Y_k|\mathcal{F}_{\varphi(k)}]-\sqrt{2r(\eps t_\vpk,Y_k)}
	\eps\varphi(1)\Delta t\right|\nonumber\\
	&& \quad +\left|\sqrt{2r(\eps t_\vpk,Y_k)}
	\eps\varphi(1)\Delta t-\sqrt{2r(\eps t_\vpk,y_k)}
	\eps\varphi(1)\Delta t\right|.
	\label{eq:W1}
\end{eqnarray}
According to Corollary \ref{cor:ode}, there exists $K_1>0$ such that for all $K\geq K_1$, there exists $\eps_1>0$ such that, for all $\eps<\eps_1$ and $k\in\llbracket 0,N-1\rrbracket$,
\begin{equation}
	\left|\mathbb{E}[Y_{k+1}-Y_k|\mathcal{F}_{\varphi(k)}]-\sqrt{2r(\eps t_\vpk,Y_k)}\eps\varphi(1)\Delta t\right|\leq \sqrt{2\nr} \eps\delta \varphi(1)\Delta t.
	\label{eq:W2}
	%+L|Y_k-y_k|\eps\delta \varphi(1)\Delta t.
\end{equation}
Besides, recall from Equation (\ref{eq:approx_r}) that, for $k\in\llbracket 0,N-1\rrbracket$, we have
\begin{equation}
	\left|\sqrt{2r(\eps t_\vpk,Y_k)}
	\eps\varphi(1)\Delta t-\sqrt{2r(\eps t_\vpk,y_k)}
	\eps\varphi(1)\Delta t\right|\leq L|y_k-Y_k|\eps\varphi(1)\Delta t.
	\label{eq:W3}
\end{equation}
Let us now assume that $K>\max(K_0,K_1)$ and $\eps<\eps_1(K)$. Combining Equations (\ref{eq:W1}), (\ref{eq:W2}) and  (\ref{eq:W3}), we obtain that %on $A_\delta$:
\begin{equation} \label{eq:W5} %\begin{cases}
	%|y_{k+1} -y_{k+1}|& \leq|Y_{k+1}-W_{k+1}-Y_0|+ \delta \\
	|y_{k+1}-W_{k+1}|\leq |y_{k}-W_{k}|+\eps \varphi(1)\Delta t\left(\sqrt{2\nr}\delta+L|Y_k-y_k|\right),
	%\end{cases}
\end{equation}
for all $k\in\llbracket 0,N-1\rrbracket$. Thus, combining Equations (\ref{ode:2}) and (\ref{eq:W5}),
% using that $Y_0=y_0$, $W_0=0$ and $\delta \in (0,1)$,
we get that for all $k\in\llbracket 0,N-1\rrbracket$,
\begin{eqnarray*}
	\text{on $\mathcal{A}_\delta$,}\quad |y_{k+1}-W_{k+1}|&\leq& |y_{k}-W_{k}|+\eps \varphi(1)\Delta t\left(\sqrt{2\nr}\delta+L\big(|y_{k}-W_{k}|+\delta \big)\right)\nonumber\\
	&\leq& \left(1+L\eps \varphi(1)\Delta t\right)|y_{k}-W_{k}|+\eps\delta \varphi(1)\Delta t\left(\sqrt{2\nr}+L\right)\nonumber\\
	&\leq & e^{L\eps\varphi(1)\Delta t}|y_{k}-W_{k}|+\eps\delta \varphi(1)\Delta t\left(\sqrt{2\nr}+L \right).
\end{eqnarray*}
Then, by induction, since $Y_0=y_0=0$ and  $W_0=0$,  for all $k\in\llbracket 0,N-1\rrbracket$
\begin{eqnarray}
	\text{on $\mathcal{A}_\delta$,}  \quad |y_{k+1}-W_{k+1}|&\leq& \eps\delta \varphi(1)\Delta t\left(\sqrt{2\nr}+L \right)\sum_{j=0}^{k}e^{jL\eps \varphi(1)\Delta t}\nonumber\\
	&\leq&N\eps\varphi(1)\Delta t \delta \left(\sqrt{2\nr}+L \right)e^{NL\eps\varphi(1)\Delta t}\nonumber\\
	&\leq& \left(\sqrt{2\nr}+L \right)\delta Te^{LT}.\label{ode:3}
	%\leq \delta Te^{L(1+\delta ) T}\leq \delta Te^{2LT} ,\label{ode:3}
\end{eqnarray}
Combining Equations (\ref{ode:2}) and (\ref{ode:3}) and setting
\jj{
	\begin{equation}
		\label{eq:def_alpha}
		\alpha=1+\left(\sqrt{2\nr}+L \right)Te^{LT},
	\end{equation}
}
we get that
\begin{equation}
	Y_k\geq y_k- \alpha\delta, \quad \forall k\in\llbracket 0,N\rrbracket,\quad \textnormal{on } A_\delta.
	\label{ode:Ad}
\end{equation}
Moreover, note that
\begin{equation}
	\begin{array}{l}
		\displaystyle\mathbb{P}\left(\exists k\in\llbracket0,N-1\rrbracket: \exists l\in\llbracket \varphi(k),\varphi(k+1)\rrbracket: \eps X^{0*}_l<y_k-2\alpha \delta \right) \vspace{0.25cm} \\
		\hspace{0.5cm}\displaystyle \leq \mathbb{P}\left(\exists k\in\llbracket0,N-1\rrbracket: \exists l\in\llbracket \varphi(k),\varphi(k+1)\rrbracket: \eps X^{0*}_l<y_k-2\alpha \delta |A_\delta \right) +\mathbb{P}(A_\delta ^c).
	\end{array} \label{ode:4}
\end{equation}

To show that the first term in (\ref{ode:4}) is small, we prove that the probability of the event
\begin{equation*}
	\left\{\exists l\in\llbracket \varphi(k), \varphi(k+1)\rrbracket :X_l ^{0*}-X_\vpk ^{0*}<-\frac{\alpha\delta}{\eps}  \right\}
\end{equation*}
decays exponentially as $\eps$ goes to zero, for all $k\in\llbracket0,N-1\rrbracket$. Let $l\in\llbracket \varphi(k), \varphi(k+1)\rrbracket$ and consider a random walk $(S_i)_{n\ge0}$ of step distribution $\mu$. As in the proof of Lemma~\ref{lb:condexp}, we have
\begin{eqnarray*}
	\mathbb{P}\left(X_l ^{0*}-X_\vpk ^{0*}<-\frac{\alpha\delta}{\eps} \right)&\leq& \mathbb{P}\left(S_{\varphi(k)}<-\frac{\alpha \delta}{\eps}\right)\\
	&=& \mathbb{P}\left(S_{\varphi(k)}>\frac{\alpha \delta}{\eps}\right)\\
	&\leq& e^{-(l-\varphi(k))I\left(\frac{\alpha \delta }{\eps(l-\varphi(k))}\right)}\\
	& \leq &e^{-I\left(\frac{\alpha \delta }{\eps\varphi(1)}\right)},
\end{eqnarray*} using a similar argument than in Equation (\ref{cramer}). Finally, since $I$ is convex, we know that $I\left(\frac{\alpha\delta}{\eps\varphi(1)}\right)\geq \frac{I(\alpha\delta)}{\alpha\delta}\frac{\alpha\delta}{\eps\varphi(1)}=\frac{I(\alpha\delta)}{\eps\varphi(1)},$ as long as $\eps\varphi(1)\leq 1$. Thus, a union bound yields that, for $\eps\varphi(1)\leq 1$,
\begin{equation}
	\mathbb{P}\left(\exists k\in \llbracket0,N-1 \rrbracket :\exists l\in\llbracket \varphi(k), \varphi(k+1)\rrbracket :X_l ^{0*}-X_\vpk ^{0*}<-\frac{\alpha\delta}{\eps} \right)\leq \frac{T}{\eps\Delta t}e^{-\frac{I(\alpha\delta)}{\eps\varphi(1)}}\label{ode:5}.
\end{equation}
%We conclude the proof by writing that
In addition, according to Equation (\ref{ode:Ad}),
\begin{equation}
	\begin{array}{l}
		\displaystyle
		\mathbb{P}\left(\exists k\in\llbracket0,N-1\rrbracket: \exists l\in\llbracket \varphi(k),\varphi(k+1)\rrbracket: \eps X^{0*}_l<y_k-2\alpha \delta |A_\delta \right)  \vspace{0.25cm}
		\\  \leq \mathbb{P}\left(\exists k\in\llbracket0,N-1\rrbracket: \exists l\in\llbracket \varphi(k),\varphi(k+1)\rrbracket: \eps X^{0*}_l-\eps X^{0*}_{\vpk}<-\alpha \delta \right),\label{ode:6}
		\displaystyle
	\end{array}
\end{equation}
and since the function $x$ is the solution of \eqref{cauchypb}, Equation (\ref{euler:s}) \added[id=j]{from Appendix \ref{sec:EDO}} and the mean value theorem imply that
\begin{eqnarray}
	|y_k-x\left(l\eps\Delta t\right)|&\leq& |y_k-x\left(k\eps\varphi(1)\Delta t\right)|+|x\left(k\eps\varphi(1)\Delta t\right)-x\left(l\eps\Delta t\right)|\nonumber\\&\leq& e^{LT}\eps\varphi(1)\Delta t+ \sqrt{2\nr}\eps\varphi(1)\Delta t, \label{ode:7}
\end{eqnarray}
for all $k\in\llbracket0,N-1\rrbracket$ and $ l\in\llbracket \varphi(k),\varphi(k+1)\rrbracket$.
Thus, if we choose $\eps_1$ small enough so that, \jj{for all $\vep\leq \vep_1$,}
\begin{eqnarray*}
	\eps\varphi(1)\Delta t(\sqrt{2\nr}+e^{LT})&\leq& \alpha\delta,\\
	\jj{\frac{2\nc^2T\varphi(1)}{\delta ^2} \eps}&\leq&  \jj{\frac{\sqrt{\vep}}{2}},\\
	\frac{T}{\eps\Delta t}e^{-\frac{I(\alpha \delta)}{\eps \varphi(1)}} &\leq& \jj{\frac{\sqrt{\vep}}{2}},
\end{eqnarray*}
we get by combining (\ref{ode:doob}),(\ref{ode:4}), (\ref{ode:5}), and (\ref{ode:6}) that
\begin{equation*}
	\mathbb{P}\left(\exists k\in\llbracket0,N-1\rrbracket\ \exists l\in\llbracket \varphi(k),\varphi(k+1)\rrbracket : X_l^{0*}<y_k-2\alpha\delta\right)\leq \jj{\sqrt{\vep}},
\end{equation*}
and finally, using Equation (\ref{ode:7}), we conclude that
\begin{equation*}
	\mathbb{P}\left(\exists n\in \left\llbracket0,\varphi(N)\right\rrbracket :
	X_n^{0*}<x(n\eps\Delta t)-3\alpha\delta\right)\leq \jj{\sqrt{\vep}}.
\end{equation*}
Again, choosing $T+1$ instead of $T$ {in the definition of $N$},
we get that for $K= \max(K_0,K_1),$ there exists $\eps'>0$ such that
for all $\eps<\eps'$,
\begin{equation}
	\mathbb{P}\left(\exists n\in\left\llbracket0,\left\lfloor\frac{T}{\eps\Delta t}\right\rfloor\right\rrbracket : X_n^{0*}<x(n\eps\Delta t)-3\alpha\delta\right)\leq \jj{\sqrt{\vep}}.
	\label{eq:th3}
\end{equation}
This proves the theorem \jj{when $r$ is a smooth growth rate function}.

\jj{\textit{Step 2.} We now assume that $r$ is a good growth rate
	function.
	Let $\tilde \delta>0$. Pick $n$ large enough so that
	\begin{equation}
		\label{approx:lb_smooth}
		\sup_{t\in[0,T]}|x(t)-\underline{x}_n(t)|<\tilde \delta,
	\end{equation}
	where $\underline{x}_n$ refers to the solution of the Cauchy problem 
	\eqref{cauchypb}
	with growth rate function $\theta_n$ (see \ref{eq:cp_approx}).
	By definition, $\theta_n$ is a smooth growth rate function.
	Moreover, it follows from Assumption~\ref{Assumption2} and
	Remark~\ref{rem:coupling_lemma} that the particle system with reproduction
	law $\nu_{r,n,K}$ stochastically dominates
	the system with reproduction
	law $\tilde{\nu}_{\theta_n,n,K}$.
	This observation, combined with Step~1,
	shows that there exists $\vep_0\equiv\vep_0(n,\Delta t,\Delta x, \delta)$,
	$\alpha\equiv \alpha(n)$,
	and $K^*\equiv K^*(n,\Delta t,\Delta x,  \delta)$ such that}
\begin{multline*}
	\forall \vep < \vep_0, \quad \forall K\geq K^*,\quad
	\mathbb{P}\left(\exists k \in \left\llbracket 0,\lfloor T/(\eps\Delta t)
	\rfloor\right\rrbracket : \eps X_k^*<{x}(k\eps\Delta t)
	+3\alpha\delta+\tilde \delta\right)
	\leq \jj{2\sqrt{\vep}}.
\end{multline*}
\jj{It then remains to choose $\delta$ small enough so that 
	$3\alpha \delta< \tilde \delta$ 
	to conclude the proof of Proposition~\ref{thlowerbound1}.
}
\end{proof}

%%%%%%%%%%%%%%%%%%%%%%%%%%%%%%%%%%%%%%%%%%%%%%%%%%%%%%%%%%%%%%%%%%%

\appendix
\section{Appendix: the branching random walk}
\label{SecBRW}

A branching random walk is a branching particle system \added[id=j]{$\mathbf{\Xi}$} governed by a reproduction law $(p_k)_{k\in\mathbb{N}}$ and a displacement law $\mu$.  The process starts with a single particle located at the origin. This particle is replaced by $N$ new particles located at positions $(\zeta_1,...,\zeta_N)$, where $N$ is distributed according to $(p_k)_{k\in\mathbb{N}}$ and $(\zeta_i)$ is an i.i.d. sequence of random variables of law $\mu$, independent of $N$.  These individuals constitute the first generation of the branching random walk. Similarly, the individuals of the $n$-th generation reproduce independently of each other according to $(p_k)_{k\in\mathbb{N}}$ and their offspring are independently distributed around the parental location according to $\mu$.

\added[id=j]{In this section, we assume that the displacement law $\mu$ is given by Equation \eqref{displaw} and that the reproduction law $(p_k)$ satisfies
$$1<m:=\sum_{k\in\mathbb{N}} kp_k<\infty.$$ The notation used below are defined in Section \ref{sec:not} and Section \ref{sec:BRW}}.

\subsection{Many-to-one lemma} \label{sec:manyto1}

%Let $D_n$ denote the set of individuals living during the $n$-th generation of the BRW and remark that $(|D_n|)_{n\in\mathbb{N}}$ constitutes a Galton-Watson process. 

\added[id=j]{For each particle $u$ such that $|u|=n$ , we denote by $u_0,...,u_n$ the set of its ancestors in chronological order (basically, $u_0$ is the particle living at generation $0$  and $u_n=u$)}.
% We assume that $1<m=\sum_{i=0}^\infty kp_k<\infty$.

\begin{lemma}[Many-to-one Lemma, see e.g.~\cite{ShiLectureNotes}, Theorem~1.1] \label{lem:manyto1}
Let $n\geq 1$ and $g:\mathbb{R}^n\rightarrow [0,\infty)$ be a measurable function. Let $(Z_k)_{k\in \mathbb{N}}$ be a sequence of random variables such that $(Z_{k+1}-Z_k)$ is i.i.d.~of law $\mu$. Then,
\begin{equation}
	\mathbb{E}\left[\sum_{u \in D_n}g(\Xi_{u_1},...,\Xi_{u_n})\right]=m^n\mathbb{E}\left[g(Z_1,...,Z_n)\right].
	\label{many21}
\end{equation}
\end{lemma}

\subsection{Regularity of the rate function}
\label{secconvconj}

In this subsection, we state several results on the function $I$ defined by Equation (\ref{convconju}) in Section \ref{partub}. All the notations are introduced in Section \ref{sec:BRW}.

\begin{lemma}
\label{lem51}
Let $y\ge0$ and assume $\Delta x \le 2y$. Then,
\begin{equation}
	I_0\left(y-\frac{\Delta x}{2}\right)\leq I(y)\leq I_0\left(y+\frac{\Delta x}{2}\right).
	\label{enca}
\end{equation}
\end{lemma}
\begin{proof}
First, remark that, for $\lambda\ge0$
\begin{eqnarray*}
	\Lambda(\lambda) & = &
	\log \sum_{i\in \mathbb{Z}}\left( 
	\int_{(i-\frac{1}{2})\Delta x}^{(i+\frac{1}{2})\Delta x} 
	\frac{1}{\sqrt{\Delta t}}\jj{\mu}\left(\frac{z}{\sqrt{\Delta t}}\right)dz 
	\right)e^{\lambda i\Delta x} \\
	&\geq & \log \sum_{i\in \mathbb{Z}}\left( \int_{(i-\frac{1}{2})\Delta x} ^{(i+\frac{1}{2})\Delta x} \frac{1}{\sqrt{\Delta t}}\jj{\mu}\left(\frac{z}{\sqrt{\Delta t}}\right)e^{\lambda(z-\frac12\Delta x)}dz \right)= \Lambda_{0}(\lambda)-\lambda \frac{\Delta x}2.
\end{eqnarray*}
Similarly, one can obtain an upper bound on $\Lambda$ and get the following estimate for any $\lambda \ge0$:
\begin{equation}\Lambda_0(\lambda)-\lambda \frac{\Delta x}2\leq \Lambda(\lambda)\leq \Lambda_{0}(\lambda)+\lambda \frac{\Delta x}2.
	\label{enc23}
\end{equation}
Now note that we have for $y\ge0$, since $\mu$ has expectation $0$,
\[
I(y) = \sup_{\lambda \ge0} \left(\lambda y - \Lambda (\lambda)\right),
\]
and similarly, for every $z\ge0$,
\[
I_0(z) = \sup_{\lambda \ge0} \left(\lambda z - \Lambda_0 (\lambda)\right).
\]
Applying these with $z = y\pm\Delta x/2$ (note that $y-\Delta x/2 \ge0$ by assumption) and using \eqref{enc23}, this  proves the lemma.
\end{proof}

\begin{lemma}
\label{minc}
Let $\rho>1$ and assume $\Delta x<\sqrt{2\Delta t\log(\rho)}$. Then,
\begin{equation}
	\frac{c_0}{2}<c<2c_0,
\end{equation}
where $c_0$ is the unique solution of $I_0(c)=\log(\rho)$ and $c$ is the unique solution of $I(c)=\log(\rho)$.
%(resp. $c$) is the unique positive solution of $I_0(c_0)=\log(\rho)$ (resp. $I(c)=\log(\rho))$.
\end{lemma}
\begin{proof}
Recall from \eqref{def:c0} that $c_0 = \sqrt{2\Delta t\log(\rho)}$, hence $\Delta x < c_0 \le 2c_0$. By Lemma~\ref{lem51} and the definitions of $c$ and $c_0$, it follows that
\[
|c-c_0| \le \frac{\Delta x}{2} < \frac{c_0}{2}.
\]
This yields the lemma.
%
%Let $\rho\in[\underline{\rho},\Bar{\rho}]$ and $y\in\left[0,\frac{c_0}{2}\right]$. 
%%and prove that $I(y)\leq  \frac{\log(\rho)}{2}<I(c),$ for $\Delta x>0$ small enough.
%The function $I$ is increasing on $(0,\infty),$ therefore, $I(y)\leq I(\frac{c_0}{2})$. Besides, since $\Delta x\leq\sqrt{2\Delta t \log(\underline{\rho})}\leq\frac{c_0}{16},$
%\begin{eqnarray*}
%    I\left(\frac{c_0}{2}\right)&\leq& I_0\left(\frac{c_0}{2}\right)+c_0 \frac{\Delta x}{\Delta t}
%    \leq I_0\left(\frac{c_0}{2}\right)+\frac{1}{16}\frac{c_0^2}{\Delta  t}\\
%    &\leq&\frac{1}{4}I_0(c_0)+ \frac{1}{8}\frac{ c_0^2}{2\Delta t}
%    =\frac{3}{8}I_0(c_0)=\frac{3}{8}\log(\rho),
%\end{eqnarray*}
%according to Lemma \ref{lem51} and Equations (\ref{io}) and (\ref{def:c0}). Since $I$ is increasing, this implies that $\frac{c_0}{2}<c$.
%Let us now consider $y>2c_0.$ Then, $I(y)\geq I(2c_0)$ and since $\Delta x\leq 2 c_0,$
%\begin{eqnarray*}
%    I(2c_0)&\geq& I_0(2c_0)-c_0\frac{\Delta x}{\Delta t}=4I_0(c_0)-c_0\frac{\Delta x}{\Delta t}\\
%    &\geq& 4I_0(c_0)-\frac{1}{8}I_0(c_0)\geq 3I_0(c_0)=3\log(\rho),
%\end{eqnarray*}
%according to Lemma \ref{lem51} and Equations (\ref{io}) and (\ref{def:c0}). Since $I$ is increasing, this estimate implies that $c<2c_0$.
\end{proof}

\begin{lemma} \label{lem:convexity} Let $y>0$. If $\Delta x<\frac{y}{2}$, then for all $\eta\ge 0$,
\begin{equation}
	I(y) + \frac{y^2}{4\Delta t}\eta \leq I((1+\eta)y).
	\label{enc}
\end{equation}
\end{lemma}
\begin{proof} The function $I$ is convex on $\mathbb{R}$. Thus, \begin{eqnarray}
	I(y)&=&I\left(\frac{1}{1+\eta}\,(1+\eta)y\right)\nonumber \\
	&\leq& \frac{1}{1+\eta}I((1+\eta)y)+\frac{\eta}{1+\eta}I(0)=\frac{1}{1+\eta}I((1+\eta)y),
	\label{eq:conv71}
\end{eqnarray}
which is equivalent to $I((1+\eta)y)\geq I(y)+\eta I(y)$. Besides, if $\Delta x<\frac{y}{2}$, Lemma \ref{lem51} implies that
\begin{equation}
	I(y)\geq I_0\left(y-\frac{\Delta x} 2\right) =  \frac{(y-\frac{\Delta x} 2)^2}{2\Delta t}\geq \frac{y^2 - y \Delta x}{2\Delta t}\geq\frac{y^2}{4\Delta t}.\label{eq:conv72}
\end{equation}
Combining (\ref{eq:conv71}) and (\ref{eq:conv72}), we get Equation (\ref{enc}).
\end{proof}

\begin{lemma}\label{lem:6}
Let $\underline{y}>0$. If $\Delta x<\frac{\underline{y}}{2}$, then
\begin{equation}
	\frac{\underline{y}}{4\Delta t}|y_1-y_2|\leq |I(y_1)-I(y_2)|,\quad \forall y_1,y_2\ge \underline y.
\end{equation}

\end{lemma}
\begin{proof} Since I is convex on $\mathbb{R}$, we have that
\begin{equation}\label{convex112}
	\frac{I(\underline{y})}{\underline{y}}=\frac{I(\underline{y})-I(0)}{\underline{y}-0}\leq \frac{I(y_1)-I(\underline{y})}{y_1-\underline{y}}\leq \frac{I(y_2)-I(y_1)}{y_2-y_1}.
\end{equation}
Besides, as in \eqref{eq:conv72}, we have that
\begin{equation}
	I(\underline{y})\geq \frac{\underline{y}^2}{4\Delta t},\label{convex113}
\end{equation}
since $\Delta x\leq\frac{\underline{y}}{2}$. We conclude the proof by combining (\ref{convex112}) and (\ref{convex113}) and recalling that $I$ is non-decreasing on $(0,\infty)$ so that $\frac{I(y_2)-I(y_1)}{y_2-y_1}=\frac{|I(y_2)-I(y_1)|}{|y_2-y_1|}.$
%
%If $y_1=\underline{y}$, we write (\ref{convex112}) with $\underline{y}-\eps$ instead of $\underline{y}$ and let $\eps$ tend to $0$. The case $y_2=\bar{y}$ can be dealt with similarly.
\end{proof}

\subsection{First and second moment of the maximum}
For all $n\in \mathbb{N},$ we denote by $M_n$ the position of the right-most particle in the branching random walk. In this section
% Sections \ref{sec:cvexp} and \ref{sec:ubm2}, 
we study  the asymptotic behaviour of the first and second moments of $M_n$. These are used for the proof of the lower bound in Section~\ref{partlow}. We recall that the reproduction law of the BRW is denoted by $(p_k)_{k\in\mathbb{N}}$. %and that $m=\sum_{k\in \mathbb{N}}kp_k$ is assumed to be larger than $1$.
\label{sec:cvexp}

\begin{lemma}[Biggins' theorem~\cite{biggins1977chernoff}]
\label{lem:ptwcv} Let $\Delta t>0$ and $\Delta x>0$. Assume that $m>1$ and $p_0 = 0$. Let $c$ be the unique positive solution of $I(c)=\log(m)$. Then,
$$\lim_{n\rightarrow 0}\mathbb{E}\left[\frac{M_n}{n}\right]=c.$$
\end{lemma}

In fact, Biggins~\cite{biggins1977chernoff} proves almost sure convergence of $M_n/n$. Lemma~\ref{lem:ptwcv} follows easily from Liggett's subadditive ergodic theorem as outlined in Zeitouni \cite{ZeitouniLNBRW}.

\begin{lemma}
\label{lemunifcv}
Let $\Delta t>0$ and $\Delta x>0$. Let $1<\underline{\rho}<\bar{\rho}$. Consider a family of reproduction laws $(p_{m})_{m>0}$ such that $\sum_{i=1}^\infty kp_{m,k}=m$ and $(p_{m})$ is increasing with respect to $m$ (with respect to stochastic domination). Uniformly in $m\in[\underline{\rho},\bar{\rho}]$,
\begin{equation*}
	\lim\limits_{n \rightarrow +\infty}\mathbb{E}\left[\frac{M_n}{n}\right] =c,
\end{equation*}
where $c$ is the unique positive solution of $I(c)=\log(m).$
\end{lemma}
\begin{proof}
Define
$$f_n:\begin{cases}
	(1,\infty) & \rightarrow \mathbb{R}                        \\
	m          & \mapsto \mathbb{E}\left[\frac{M_n}{n}\right].
\end{cases}$$
We claim that $(f_n)$ is a sequence of increasing %and continuous 
functions. Indeed, for any $1<m_1<m_2$, consider $S^1$ (resp. $S^2$) a branching random walk of reproduction law $p_{m_1}$ (resp. $p_{m_2}$) and of displacement law $\mu$. According to Lemma \ref{lem:coupl}, we can construct a coupling between $S^1$ and $S^2$, such that $S^2$ dominates $S^1$. Hence, $M^1_n \le M^2_n$, where $M^i_n$ denotes the position of the maximal particle in the branching random walk $S^i$, $i=1,2$. It follows that $f_n$ is increasing on $(1,\infty)$, for all $n\in \mathbb{N}$.
%Moreover, the monotone convergence theorem implies that $f_n(m_1)$ tends to $f_n(m_2)$ as $m_1$ tends to $m_2$. Hence, $f_n$ is continuous on $(1,\infty)$, for all $n\in \mathbb{N}$. 

Let us now consider the function $c:(1,\infty)\rightarrow (0,\infty)$ that maps $m$ to the unique positive solution of $I(c)=\log(m)$. According to Lemma \ref{lem:6}, the function $c$ is continuous on $(1,\infty)$. Moreover, by Lemma~\ref{lem:ptwcv}, $f_n \to c$ pointwise as $n\to\infty$. Using the monotonicity of $f_n$, Dini's theorem then yields uniform convergence on the compact sets of $[\underline{\rho},\bar{\rho}]$.
\end{proof}

%
%\subsection{Upper bound on the second moment}
%\label{sec:ubm2}

\begin{lemma}
\label{lemmom2}
Let $\Delta t>0$, $1<\underline{\rho}<\bar{\rho}$ and $\Delta x<\sqrt{2\Delta t\log(\underline{\rho})}$. Consider a reproduction law $(p_k)_{k\in\mathbb{N}}$ such that $\sum k p_k =m>1$ and $c$ the unique positive solution or $I(c)=\log(m)$. For any $\eta>0$, uniformly in $m\in[\underline{\rho},\bar{\rho}],$ there exists $N\in\mathbb N$ such that
\[\forall n\geq N,\quad\mathbb{E}\left[\frac{M_n^2}{n^2}\right]\leq (c+\eta)^2.\]
\end{lemma}
\begin{proof}
Let us first remark that, for all $n\in \mathbb{N}$,
$$\mathbb{E}\left[\frac{1}{n^2}\left(\max_{|v|=n}\Xi_v\right)^2\right]\leq \mathbb{E}\left[\frac{1}{n^2}\left(\max_{|v|=n}|\Xi_v|\right)^2\right].$$
Then, we define $\xi_n=\frac{1}{n}\max_{|v|=n}|\Xi_v|,$ and write that, for all $R>0$,
\begin{equation}
	\mathbb{E}[\xi_n^2]=\mathbb{E}[\xi_n^2\mathbb{1}_{\xi_n^2<R^2}]+\mathbb{E}[\xi_n^2\mathbb{1}_{\xi_n^2\geq R^2}].
	\label{eq:spm2}
\end{equation}
Besides, $ \mathbb{E}[\xi_n^2\mathbb{1}_{\xi_n^2<R^2}]\leq R^2$ and,
\begin{equation}
	\mathbb{E}[\xi_n^2\mathbb{1}_{\xi_n^2\geq R^2}] =\int_R^\infty 2u\mathbb{P}(\xi_n\geq u)du.
	\label{eq:fub}
\end{equation}
Thanks to the many-to-one lemma (Lemma~\ref{lem:manyto1}) and Markov's inequality, we know that
\begin{equation}
	\mathbb{P}(\xi_n\geq u)=\mathbb{P}\left(\frac{1}{n}\max_{|v|=n}|\Xi_v|\geq u\right)\leq \mathbb{E}\left[\sum_{|v|=n}\mathbb{1}_{\frac{|\Xi_v]}{n}\geq u}\right]\leq m^n \mathbb{P}\left(\frac{|Z_n|}{n}\geq u \right),\label{manyto2X}
\end{equation}
where $(Z_n)_{n\ge0}$ is a random walk whose increments are distributed as $\mu$. Moreover, by symmetry,
\begin{equation}
	\label{eq:symmetry}
	\mathbb{P}\left(\frac{|\Xi_v|}{n}\geq u \right)=2\mathbb{P}\left(\frac{\Xi_v}{n}\geq u \right).
\end{equation}
Equations \eqref{manyto2X} and \eqref{eq:symmetry}, together with Chernoff's bound then give
\begin{equation}
	\mathbb{P}(\xi_n\geq u)\leq 2m^n e^{-nI(u)}=2e^{-n(I(c)-I(u))}\leq 2e^{-n\frac{I(c)}{c}(u-c)}.
	\label{eq:31}
\end{equation}
According to Lemma \ref{minc} and Equation (\ref{def:c0}), since $\Delta x<\sqrt{2\Delta t\log(\underline{\rho})},$\begin{equation*}
	\frac{1}{2}\sqrt{2\Delta t \log(\underline{\rho})}<\frac{1}{2}\sqrt{2\Delta t \log(m)}<c<2\sqrt{2\Delta t \log(m)}<2\sqrt{2\Delta t \log(m)}<2\sqrt{2\Delta t \log(\bar{\rho})},
\end{equation*} so that
\begin{equation}
	0<\alpha:=\frac{I\left(\frac{1}{2}\sqrt{2\Delta t \log(\underline{\rho})}\right)}{2\sqrt{2\Delta t\log(\bar{\rho})}}<\frac{I(c)}{c}.
	\label{lb:c}
\end{equation}
Let $\eta>0$ and  consider $R=c+\eta$. Equations (\ref{eq:fub}), (\ref{manyto2X}), (\ref{eq:31}) and (\ref{lb:c}) give that
\begin{equation*}
	\mathbb{E}[\xi_n^2\mathbb{1}_{\xi_n^2\geq R^2}] \leq 4 \int_{c+\eta}^\infty ue^{-n\alpha(u-c)}du=4\int_\eta^\infty(u+c)e^{-n\alpha u}du\leq 4\int_\eta^\infty(u+\sqrt{2\Delta t \log(\bar{\rho})})e^{-n\alpha u}du.
\end{equation*}
Remark that the last integral tends to $0$ as $n$ tends to infinity, so that there exists $N_\eta\in \mathbb{N}$ such that, for all $n\geq N_\eta$,
\begin{equation*}
	\mathbb{E}[\xi_n^2\mathbb{1}_{\xi_n^2\geq R^2}] \leq \eta,
\end{equation*}
and, thanks to Equation (\ref{eq:spm2}), for $n$ large enough,
\begin{equation*}
	\mathbb{E}[\xi_n^2]\leq(c+\eta)^2+\eta.
\end{equation*}
Since $\eta>0$ was arbitrary, this yields the result.
\end{proof}

%%%%%%%%%%%%%%%%%%%%%%%%%%%%%%%%%%%%%%%%%%%%%%%%%%%%%%%%%%%%%%%%%%%

\section{Appendix: stability of the solution of (\ref{cauchypb}) and convergence of the Euler scheme}

\label{sec:EDO}

\subsection{Stability}
\begin{lemma}\label{lem:stab:EDO}
\jj{Assume that $r$ is a smooth growth rate function}. Let $\delta>0$ and $T>0$. Consider $x$ the solution of
\begin{equation*}
	\begin{cases}
		\dot{x}(t) & =\sqrt{2r(t,x(t))} \\
		x(0)       & =0,
	\end{cases}
\end{equation*}
and $\tilde{x}$ the solution of
\begin{equation*}
	\begin{cases}
		\dot{\tilde{x}}(t) & =\sqrt{2r(t,\tilde{x}(t))}+\delta \\
		\tilde{x}(0)       & =0.
	\end{cases}
\end{equation*}
Then,
\begin{equation*}
	\sup_{t\in[0,T]}|x(t)-\tilde{x}(t)|\leq \delta (T+1)e^{LT}.
\end{equation*}
\end{lemma}
\begin{proof}
%Recall from Remark \ref{solglobal} that under Assumption \ref{Assumption0}, the solutions $x$ and $\tilde{x}$ are unique and defined on $[0,+\infty)$.
Let $u(t)=\sqrt{(x(t)-\tilde{x}(t))^2+\delta}$ for $t\in[0,T]$. Note that
\begin{equation*}
	|x(t)-\tilde{x}(t)|\leq u(t), \quad \forall t\in[0,T].
\end{equation*}
Besides, the function $u$ is differentiable on $[0,T]$ and,
\begin{eqnarray*}
	u'(t)&=& \frac{1}{2u(t)}\frac{d}{dt}((\tilde{x}(t)-x(t))^2+\delta)=\frac{1}{u(t)}\left(\frac{d}{dt}x(t)-\frac{d}{dt}\tilde{x}(t)\right)(x(t)-\tilde{x}(t))\\
	&=&\frac{1}{u(t)}\left(\sqrt{2r(t,\tilde{x}(t))}-\sqrt{2r(t,x(t))}+\delta\right)(\tilde{x}(t)-x(t))\\
	&\leq& \frac{1}{u(t)} (L|x(t)-\tilde{x}(t)|+\delta)|\tilde{x}(t)-x(t)|\\
	&\leq& Lu(t)+\delta,
\end{eqnarray*}
where the first inequality \jj{follows from \eqref{eq:approx_r}}. Then, by Grönwall's
inequality, we obtain
\begin{equation*}
	u(t)\leq (\delta t+u(0))e^{Lt}\leq \delta (T+1)e^{LT}, \quad \forall t\in[0,T],
\end{equation*}
which concludes the proof of the lemma.
\end{proof}

\subsection{Euler scheme}
%Under Assumption \ref{Assumption0}, we know that each maximal solution $y$ of \eqref{cauchypb} is global $i.e.$ defined on $[0,\infty).$ 
Consider $x$ the solution of \eqref{cauchypb}.
For any $T>0$ and $h>0$, we can define the Euler scheme of this solution on $[0,T]$ by considering the sequence $(y_i)$ defined by
\begin{equation*}
\begin{cases}
	y_0     & = \added[id=j]{x(0)}     \\
	y_{i+1} & =y_i+\sqrt{2r(ih,y_i)}h.
\end{cases}
\end{equation*}
\jj{Recall the definition of $L$ from \eqref{eq:approx_r}}.
Thanks to \jj{standard} convergence results on the Euler method (see Theorem 14.3 from \cite{willoughby1965elements}), we know that
\begin{equation}
\max_{i\in \llbracket 0,\lfloor T/h\rfloor\rrbracket }|x(t_i)-y_i|\leq e^{LT}\frac{h}{2} \label{euler:s}.
\end{equation}
\begin{rem}
For any $\delta >0$, Equation  (\ref{euler:s}) still holds for the function $\tilde{x}$ solution of
\begin{equation*}
	\dot{\tilde{x}}(t)=\sqrt{2r(t,\tilde{x}(t))}+\delta,
\end{equation*}
and its Euler scheme $(\tilde{y}_i)$ on $[0,T]$:
\begin{equation*}
	\begin{cases}
		\tilde{y}_0     & = \added[id=j]{\tilde{x}(0)}                      \\
		\tilde{y}_{i+1} & =\tilde{y}_i+(\sqrt{2r(ih,\tilde{y}_i)}+\delta)h.
	\end{cases}
\end{equation*}
\label{rem:euler:delta}
\end{rem}
\section*{Acknowledgments}

P.M. and J.T. partially supported by grant ANR-20-CE92-0010-01. G.R. et J.T. have recieved partial funding from the ANR project DEEV ANR-20-CE40-0011-01 and the chaire Modélisation Mathématique et Biodiversité of Véolia Environment - École Polytechnique - Museum National d'Histoire Naturelle - Fondation X.

\bibliographystyle{abbrv}
\bibliography{biblio,2019_these_Julie}

\begin{thebibliography}{10}

\bibitem{barles1989wavefront}
G.~Barles, L.~C. Evans, and P.~E. Souganidis.
\newblock Wavefront propagation for reaction-diffusion systems of {PDE}.
\newblock Technical report, Brown University Providence RI Lefschetz Center for
  dynamical systems, 1989.

\bibitem{Barton2013}
N.~H. Barton, A.~Etheridge, and A.~V{\'{e}}ber.
\newblock {Modelling evolution in a spatial continuum}.
\newblock {\em Journal of Statistical Mechanics: Theory and Experiment},
  (01):P01002, jan 2013.

\bibitem{Bensoussan1982}
A.~Bensoussan and J.-L. Lions.
\newblock {\em {Applications of Variational Inequalities in Stochastic
  Control}}.
\newblock Studies in mathematics and its applications 12. North-Holland, first
  edition, 1982.

\bibitem{Berard2010}
J.~B{\'{e}}rard and J.-B. Gou{\'{e}}r{\'{e}}.
\newblock {Brunet-Derrida behavior of branching-selection particle systems on
  the line}.
\newblock {\em Communications in Mathematical Physics}, 298(2):323--342, jun
  2010.

\bibitem{Berestycki:2005tt}
H.~Berestycki, F.~Hamel, and L.~. a. c. d. l. . .-. Roques.
\newblock Analysis of the periodically fragmented environment model : Ii -
  biological invasions and pulsating travelling fronts. %+ centre d'analyse et
  de math{\'e}matique sociales (cams) %+ laboratoire d'analyse, topologie,
  probabilit{\'e}s (latp).
\newblock {\em Journal de Math{\'e}matiques Pures et Appliqu{\'e}es}, page
  1101, 2005.

\bibitem{Berestycki:2015aa}
H.~Berestycki and G.~Nadin.
\newblock Asymptotic spreading for general heterogeneous {Fisher-KPP} type
  equations.
\newblock {\em preprint}, 2015.

\bibitem{Berestycki2010}
J.~Berestycki, N.~Berestycki, and J.~Schweinsberg.
\newblock {The genealogy of branching Brownian motion with absorption}.
\newblock {\em The Annals of Probability}, 41(2):527--618, mar 2013.

\bibitem{Berestycki2009a}
N.~Berestycki, A.~Etheridge, and M.~Hutzenthaler.
\newblock {Survival, extinction and ergodicity in a spatially continuous
  population model}.
\newblock {\em Markov Processes and Related Fields}, 2009.

\bibitem{Bezborodov2018}
V.~Bezborodov, L.~{Di Persio}, T.~Krueger, and P.~Tkachov.
\newblock {Spatial growth processes with long range dispersion: Microscopics,
  mesoscopics and discrepancy in spread rate}.
\newblock {\em Annals of Applied Probability}, 30(3):1091--1129, jul 2020.

\bibitem{biggins1977chernoff}
J.~D. Biggins.
\newblock Chernoff's theorem in the branching random walk.
\newblock {\em Journal of Applied Probability}, 14(3):630--636, 1977.

\bibitem{Birkner2007}
M.~Birkner and A.~Depperschmidt.
\newblock {Survival and complete convergence for a spatial branching system
  with local regulation}.
\newblock {\em The Annals of Applied Probability}, 17(5/6):1777--1807, oct
  2007.

\bibitem{Birkner:2019aa}
M.~Birkner, N.~Gantert, and S.~Steiber.
\newblock Coalescing directed random walks on the backbone of a 1+1-dimensional
  oriented percolation cluster converge to the brownian web.
\newblock {\em Latin American Journal of Probability and Mathematical
  Statistics}, 16:1029, 01 2019.

\bibitem{Blath2007}
J.~Blath, A.~Etheridge, and M.~Meredith.
\newblock {Coexistence in locally regulated competing populations and survival
  of branching annihilating random walk}.
\newblock {\em Annals of Applied Probability}, 17(5-6):1474--1507, 2007.

\bibitem{Brunet1997}
{\'{E}}.~Brunet and B.~Derrida.
\newblock {Shift in the velocity of a front due to a cutoff}.
\newblock {\em Physical Review E}, 56(3):2597--2604, sep 1997.

\bibitem{Brunet2006}
{\'{E}}.~Brunet, B.~Derrida, A.~Mueller, and S.~Munier.
\newblock {Phenomenological theory giving the full statistics of the position
  of fluctuating pulled fronts}.
\newblock {\em Physical Review E}, 73(5):056126, may 2006.

\bibitem{Brunet2006a}
{\'{E}}.~Brunet, B.~Derrida, A.~H. Mueller, and S.~Munier.
\newblock {Noisy traveling waves: Effect of selection on genealogies}.
\newblock {\em Europhysics Letters (EPL)}, 76(1):1--7, oct 2006.

\bibitem{Champagnat2007}
N.~Champagnat and S.~M{\'e}l{\'e}ard.
\newblock Invasion and adaptive evolution for individual-based spatially
  structured populations.
\newblock {\em Journal of Mathematical Biology}, 55(2):147, Jun 2007.

\bibitem{Cortines2016a}
A.~Cortines.
\newblock {The genealogy of a solvable population model under selection with
  dynamics related to directed polymers}.
\newblock {\em Bernoulli}, 22(4):2209--2236, 2016.

\bibitem{Etheridge2004}
A.~Etheridge.
\newblock {Survival and extinction in a locally regulated population}.
\newblock {\em The Annals of Applied Probability}, 14(1):188--214, feb 2004.

\bibitem{Etheridge2017}
A.~Etheridge, N.~Freeman, and D.~Straulino.
\newblock {The Brownian net and selection in the spatial $\Lambda$-Fleming-Viot
  process}.
\newblock {\em Electronic Journal of Probability}, 22:1--37, 2017.

\bibitem{Evans:1989to}
L.~C. Evans and P.~E. Souganidis.
\newblock A pde approach to geometric optics for certain semilinear parabolic
  equations.
\newblock {\em Indiana University mathematics journal}, 38(1):141--172, 1989.

\bibitem{fournier2004microscopic}
N.~Fournier and S.~M{\'e}l{\'e}ard.
\newblock A microscopic probabilistic description of a locally regulated
  population and macroscopic approximations.
\newblock {\em The Annals of Applied Probability}, 14(4):1880--1919, 2004.

\bibitem{Freidlin1985}
M.~Freidlin.
\newblock {Limit Theorems for Large Deviations and Reaction-Diffusion
  Equations}.
\newblock {\em The Annals of Probability}, 13(3):639--675, aug 1985.

\bibitem{Freidlin1986}
M.~Freidlin.
\newblock {Geometric Optics Approach To Reaction-Diffusion Equations.}
\newblock {\em SIAM Journal on Applied Mathematics}, 46(2):222--232, 1986.

\bibitem{Hamel:2010wf}
F.~Hamel, J.~Fayard, and L.~Roques.
\newblock Spreading speeds in slowly oscillating environments.
\newblock {\em Bulletin of Mathematical Biology}, 72(5):1166--1191, 2010.

\bibitem{Hamel:2011ue}
F.~Hamel, G.~Nadin, and L.~Roques.
\newblock A viscosity solution method for the spreading speed formula in slowly
  varying media.
\newblock {\em Indiana University Mathematics Journal}, pages 1229--1247, 2011.

\bibitem{willoughby1965elements}
P.~Henrici.
\newblock Elements of numerical analysis, 1965.

\bibitem{Hutzenthaler2007}
M.~Hutzenthaler and A.~Wakolbinger.
\newblock {Ergodic behavior of locally regulated branching populations}.
\newblock {\em The Annals of Applied Probability}, 17(2):474--501, apr 2007.

\bibitem{jabin2012small}
P.~E. Jabin.
\newblock Small populations corrections for selection-mutation models.
\newblock {\em arXiv preprint arXiv:1203.4123}, 2012.

\bibitem{Kuehn2019}
C.~Kuehn.
\newblock {Travelling Waves in Monostable and Bistable Stochastic Partial
  Differential Equations}.
\newblock {\em Jahresbericht der Deutschen Mathematiker-Vereinigung}, 3:1--30,
  2019.

\bibitem{nbbm}
P.~Maillard.
\newblock {Speed and fluctuations of N-particle branching Brownian motion with
  spatial selection}.
\newblock {\em Probability Theory and Related Fields}, 166(3):1061--1173, 2016.

\bibitem{brwnlc}
P.~Maillard and S.~Penington.
\newblock {Branching random walk with non-local competition}.
\newblock {\em preprint}, 2021.

\bibitem{mirrahimi2012singular}
S.~Mirrahimi, G.~Barles, B.~Perthame, and P.~E. Souganidis.
\newblock A singular hamilton--jacobi equation modeling the tail problem.
\newblock {\em SIAM Journal on Mathematical Analysis}, 44(6):4297--4319, 2012.

\bibitem{Mueller2010}
C.~Mueller, L.~Mytnik, and J.~Quastel.
\newblock {Effect of noise on front propagation in reaction-diffusion equations
  of KPP type}.
\newblock {\em Inventiones mathematicae}, 184(2):405--453, nov 2010.

\bibitem{Nadin:2010vd}
G.~Nadin.
\newblock The effect of the schwarz rearrangement on the periodic principal
  eigenvalue of a nonsymmetric operator.
\newblock {\em SIAM journal on mathematical analysis}, 41(6):2388--2406, 2010.

\bibitem{Nadin:2016aa}
G.~Nadin.
\newblock How does the spreading speed associated with the fisher-kpp equation
  depend on random stationary diffusion and reaction terms?
\newblock {\em arXiv preprint arXiv:1609.01441}, 2016.

\bibitem{Pain2015}
M.~Pain.
\newblock {Velocity of the L-branching Brownian motion}.
\newblock {\em Electronic Journal of Probability}, 21:no. 28, 1--28, oct 2016.

\bibitem{Panja2004}
D.~Panja.
\newblock {Effects of fluctuations on propagating fronts}.
\newblock {\em Physics Reports}, 393(2):87--174, mar 2004.

\bibitem{Schertzer2016}
E.~Schertzer, R.~Sun, and J.~M. Swart.
\newblock {The Brownian web, the Brownian net, and their universality}.
\newblock {\em Advances in Disordered Systems, Random Processes and Some
  Applications}, pages 270--368, 2016.

\bibitem{ShiLectureNotes}
Z.~Shi.
\newblock {\em {Branching random walks}}, volume 2151 of {\em Lecture Notes in
  Mathematics}.
\newblock Springer, Cham, 2015.

\bibitem{Shigesada:1986aa}
N.~Shigesada.
\newblock Traveling periodic waves in heterogeneous environments.
\newblock {\em Theor. Popul. Biol.}, 30:143--160, 1986.

\bibitem{shigesada1997biological}
N.~Shigesada and K.~Kawasaki.
\newblock {\em Biological invasions: theory and practice}.
\newblock Oxford University Press, UK, 1997.

\bibitem{Smaily:2009vv}
M.~E. Smaily, F.~Hamel, and L.~Roques.
\newblock Homogenization and influence of fragmentation in a biological
  invasion model.
\newblock {\em arXiv preprint arXiv:0907.4951}, 2009.

\bibitem{Xin:1991aa}
X.~Xin.
\newblock Existence and stability of traveling waves in periodic media governed
  by a bistable nonlinearity.
\newblock {\em Journal of Dynamics and Differential Equations}, 3(4):541--573,
  1991.

\bibitem{ZeitouniLNBRW}
O.~Zeitouni.
\newblock {Branching random walks and {G}aussian fields}.
\newblock In {\em Probability and statistical physics in {S}t. {P}etersburg},
  volume~91 of {\em Proc. Sympos. Pure Math.}, pages 437--471. Amer. Math.
  Soc., Providence, RI, 2016.

\end{thebibliography}

\end{document}